\documentclass[english]{article}
\usepackage[letterpaper, margin=1.2in]{geometry}
\usepackage{babel}
\usepackage{verbatim}
\usepackage{mathtools}
\usepackage{booktabs}
\usepackage{enumitem}
\usepackage{bm}
\usepackage{algorithm, algpseudocode, multirow}
\providecommand{\algorithmname}{Algorithm}
\floatname{algorithm}{\protect\algorithmname}
\floatstyle{ruled}
\newfloat{function}{tbp}{lofn}
\floatname{function}{Function}
\usepackage{amsmath}
\numberwithin{equation}{section}
\usepackage{amsthm}
\usepackage{amssymb}
\usepackage{minitoc}
\usepackage{color,xcolor}
\usepackage{wrapfig}
\usepackage{pifont}
\usepackage[numbers]{natbib}
\usepackage[pdfusetitle, colorlinks=false]{hyperref} 
\usepackage[all]{hypcap}
\usepackage{cleveref}
\crefname{ALG@line}{line}{lines}
\Crefname{ALG@line}{Line}{Lines}
\makeatletter
\@ifpackageloaded{hyperref}{
    \DeclareRobustCommand{\theHALG@line}{line.\thealgorithm.\thefunction.\arabic{ALG@line}}
}{
}
\makeatother

\makeatletter
\theoremstyle{definition}
\newtheorem{defn}{\protect\definitionname}[]
\theoremstyle{plain}
\newtheorem{thm}{\protect\theoremname}[section]
\theoremstyle{plain}
\newtheorem{proposition}{\protect\propositionname}[section]
\theoremstyle{plain}
\newtheorem{assumption}{\protect\assumptionname}
\theoremstyle{plain}
\newtheorem{lem}{\protect\lemmaname}[section]
\theoremstyle{remark}

\theoremstyle{definition}

\providecommand{\examplename}{Example}
\theoremstyle{corollary}
\newtheorem{corollary}{\protect\corollaryname}[section]

\providecommand{\assumptionname}{Assumption}
\providecommand{\definitionname}{Definition}
\providecommand{\lemmaname}{Lemma}
\providecommand{\propositionname}{Proposition}
\providecommand{\remarkname}{Remark}
\providecommand{\theoremname}{Theorem}
\providecommand{\corollaryname}{Corollary}

\crefdefaultlabelformat{#2\textbf{#1}#3} %
\crefname{section}{\textbf{section}}{\textbf{sections}}
\Crefname{section}{\textbf{Section}}{\textbf{Sections}}
\crefname{thm}{\textbf{Theorem}}{\textbf{theorems}}
\Crefname{thm}{\textbf{Theorem}}{\textbf{Theorems}}
\crefname{lem}{\textbf{Lemma}}{\textbf{lemmas}}
\Crefname{lem}{\textbf{Lemma}}{\textbf{Lemmas}}
\crefname{prop}{\textbf{proposition}}{\textbf{propositions}}
\Crefname{prop}{\textbf{Proposition}}{\textbf{Propositions}}
\crefname{algorithm}{\textbf{algorithm}}{\textbf{algorithms}}
\Crefname{algorithm}{\textbf{Algorithm}}{\textbf{Algorithms}}
\crefname{function}{\textbf{function}}{\textbf{functions}}
\Crefname{function}{\textbf{Function}}{\textbf{Functions}}
\crefname{coro}{\textbf{Corollary}}{\textbf{corollaries}}
\Crefname{coro}{\textbf{Corollary}}{\textbf{corollaries}}
\crefname{defn}{\textbf{Definition}}{\textbf{definitions}}
\Crefname{defn}{\textbf{Definition}}{\textbf{definitions}}
\crefname{table}{\textbf{Table}}{\textbf{tables}}
\Crefname{table}{\textbf{Table}}{\textbf{tables}}
\crefname{figure}{\textbf{Figure}}{\textbf{figures}}
\Crefname{figure}{\textbf{Figure}}{\textbf{figures}}
\crefname{exple}{\textbf{Example}}{\textbf{examples}}
\Crefname{exple}{\textbf{Example}}{\textbf{examples}}
\Crefname{assumption}{\textbf{Assumption}}{\textbf{Assumptions}}
\crefname{assumption}{\textbf{Assumption}}{\textbf{Assumptions}}
\Crefname{rem}{\textbf{Remark}}{\textbf{Remarks}}
\crefname{rem}{\textbf{Remark}}{\textbf{Remarks}}

\usepackage[most]{tcolorbox}
\usepackage{fancybox,ulem,subfig}
\tcbuselibrary{breakable,theorems,skins}

\newtcbtheorem[number within=section]{exbox}{Example}{
  enhanced,
  colback=gray!3,           
  colframe=gray!70!black,   
  boxrule=0.6pt,
  arc=2mm,                  
  left=8pt,right=8pt,top=8pt,bottom=8pt,
  fonttitle=\bfseries,       
  attach boxed title to top left={yshift=-1mm, xshift=2mm},
}{ex}

\providecommand{\corollaryname}{Corollary}

\newcommand{\isarxiv}{}

\makeatother

\begin{document}
\global\long\def\inprod#1#2{\left\langle #1,#2\right\rangle }%

\global\long\def\inner#1#2{\left\langle #1,#2\right\rangle }%

\global\long\def\binner#1#2{\big\langle#1,#2\big\rangle}%

\global\long\def\norm#1{\left\Vert #1\right\Vert }%

\global\long\def\bnorm#1{\big\Vert#1\big\Vert}%

\global\long\def\Bnorm#1{\Big\Vert#1\Big\Vert}%

\global\long\def\red#1{\textcolor{red}{#1}}%

\global\long\def\blue#1{\textcolor{blue}{#1}}%

\global\long\def\brbra#1{\left(#1\right)}%

\global\long\def\Brbra#1{\left(#1\right)}%

\global\long\def\rbra#1{(#1)}%

\global\long\def\lrbra#1{\left(#1\right)}%

\global\long\def\sbra#1{[#1]}%

\global\long\def\bsbra#1{\left[#1\right]}%

\global\long\def\Bsbra#1{\Big[#1\Big]}%

\global\long\def\abs#1{\vert#1\vert}%

\global\long\def\babs#1{\big\vert#1\big\vert}%

\global\long\def\lrabs#1{\left|#1\right|}%

\global\long\def\cbra#1{\{#1\}}%

\global\long\def\bcbra#1{\left\{  #1\right\}  }%

\global\long\def\Bcbra#1{\left\{  #1\right\}  }%

\global\long\def\matr#1{\bm{#1}}%

\global\long\def\til#1{\tilde{#1}}%

\global\long\def\wtil#1{\widetilde{#1}}%

\global\long\def\wh#1{\widehat{#1}}%

\global\long\def\mcal#1{\mathcal{#1}}%

\global\long\def\mbb#1{\mathbb{#1}}%

\global\long\def\mtt#1{\mathtt{#1}}%

\global\long\def\ttt#1{\texttt{#1}}%

\global\long\def\dtxt{\textrm{d}}%

\global\long\def\aeq{\overset{(a)}{=}}%

\global\long\def\bignorm#1{\bigl\Vert#1\bigr\Vert}%

\global\long\def\Bignorm#1{\Bigl\Vert#1\Bigr\Vert}%

\global\long\def\rmn#1#2{\mathbb{R}^{#1\times#2}}%

\global\long\def\deri#1#2{\frac{d#1}{d#2}}%

\global\long\def\pderi#1#2{\frac{\partial#1}{\partial#2}}%

\global\long\def\limk{\lim_{k\rightarrow\infty}}%

\global\long\def\trans{\textrm{T}}%

\global\long\def\onebf{\mathbf{1}}%

\global\long\def\zerobf{\mathbf{0}}%

\global\long\def\zero{\bm{0}}%


\global\long\def\Euc{\mathrm{E}}%

\global\long\def\Expe{\mathbb{E}}%

\global\long\def\rank{\mathrm{rank}}%

\global\long\def\range{\mathrm{range}}%

\global\long\def\diam{\mathrm{diam}}%

\global\long\def\epi{\mathrm{epi} }%

\global\long\def\inte{\operatornamewithlimits{int}}%

\global\long\def\dist{\operatornamewithlimits{dist}}%

\global\long\def\proj{\operatornamewithlimits{Proj}}%

\global\long\def\cov{\mathrm{Cov}}%

\global\long\def\argmin{\operatornamewithlimits{argmin}}%

\global\long\def\argmax{\operatornamewithlimits{argmax}}%

\global\long\def\where{\operatornamewithlimits{where}}%

\global\long\def\conv{\operatornamewithlimits{conv}}%

\global\long\def\tr{\operatornamewithlimits{tr}}%

\global\long\def\dist{\operatorname{dist}}%

\global\long\def\sign{\operatornamewithlimits{sign}}%

\global\long\def\prob{\mathrm{Prob}}%

\global\long\def\st{\operatornamewithlimits{s.t.}}%

\global\long\def\dom{\mathrm{dom}}%

\global\long\def\prox{\mathrm{prox}}%

\global\long\def\diag{\mathrm{diag}}%

\global\long\def\and{\mathrm{and}}%

\global\long\def\as{\textup{a.s.}}%

\global\long\def\ae{\textup{a.e.}}%

\global\long\def\Var{\operatornamewithlimits{Var}}%

\global\long\def\Cov{\operatornamewithlimits{Cov}}%

\global\long\def\raw{\rightarrow}%

\global\long\def\law{\leftarrow}%

\global\long\def\Raw{\Rightarrow}%

\global\long\def\Law{\Leftarrow}%

\global\long\def\vep{\varepsilon}%

\global\long\def\dom{\operatornamewithlimits{dom}}%

\global\long\def\err{\operatorname{err}}%

\global\long\def\soc{\operatorname{soc}}%

\global\long\def\rsoc{\operatorname{rsoc}}%

\global\long\def\tsum{{\textstyle {\sum}}}%

\global\long\def\Cbb{\mathbb{C}}%

\global\long\def\Ebb{\mathbb{E}}%

\global\long\def\Fbb{\mathbb{F}}%

\global\long\def\Nbb{\mathbb{N}}%

\global\long\def\Rbb{\mathbb{R}}%

\global\long\def\extR{\widebar{\mathbb{R}}}%

\global\long\def\Pbb{\mathbb{P}}%

\global\long\def\Mrm{\mathrm{M}}%

\global\long\def\Acal{\mathcal{A}}%

\global\long\def\Bcal{\mathcal{B}}%

\global\long\def\Ccal{\mathcal{C}}%

\global\long\def\Dcal{\mathcal{D}}%

\global\long\def\Ecal{\mathcal{E}}%

\global\long\def\Fcal{\mathcal{F}}%

\global\long\def\Gcal{\mathcal{G}}%

\global\long\def\Hcal{\mathcal{H}}%

\global\long\def\Ical{\mathcal{I}}%

\global\long\def\Kcal{\mathcal{K}}%

\global\long\def\Lcal{\mathcal{L}}%

\global\long\def\Mcal{\mathcal{M}}%

\global\long\def\Ncal{\mathcal{N}}%

\global\long\def\Ocal{\mathcal{O}}%

\global\long\def\Pcal{\mathcal{P}}%

\global\long\def\Scal{\mathcal{S}}%

\global\long\def\Tcal{\mathcal{T}}%

\global\long\def\Xcal{\mathcal{X}}%

\global\long\def\Ycal{\mathcal{Y}}%

\global\long\def\Zcal{\mathcal{Z}}%

\global\long\def\i{i}%

\global\long\def\abf{\mathbf{a}}%

\global\long\def\bbf{\mathbf{b}}%

\global\long\def\cbf{\mathbf{c}}%

\global\long\def\ebf{\mathbf{e}}%

\global\long\def\fbf{\mathbf{f}}%

\global\long\def\hbf{\mathbf{h}}%

\global\long\def\qbf{\mathbf{q}}%

\global\long\def\gbf{\mathbf{g}}%

\global\long\def\lambf{\bm{\lambda}}%

\global\long\def\alphabf{\bm{\alpha}}%

\global\long\def\sigmabf{\bm{\sigma}}%

\global\long\def\thetabf{\bm{\theta}}%

\global\long\def\deltabf{\bm{\delta}}%

\global\long\def\lbf{\mathbf{l}}%

\global\long\def\ubf{\mathbf{u}}%

\global\long\def\pbf{\mathbf{\mathbf{p}}}%

\global\long\def\vbf{\mathbf{v}}%

\global\long\def\wbf{\mathbf{w}}%

\global\long\def\xbf{\mathbf{x}}%

\global\long\def\ybf{\mathbf{y}}%

\global\long\def\zbf{\mathbf{z}}%

\global\long\def\dbf{\mathbf{d}}%

\global\long\def\Wbf{\mathbf{W}}%

\global\long\def\Abf{\mathbf{A}}%

\global\long\def\Gbf{\mathbf{G}}%

\global\long\def\Ubf{\mathbf{U}}%

\global\long\def\Pbf{\mathbf{P}}%

\global\long\def\Ibf{\mathbf{I}}%

\global\long\def\Ebf{\mathbf{E}}%

\global\long\def\Mbf{\mathbf{M}}%

\global\long\def\Dbf{\mathbf{D}}%

\global\long\def\Qbf{\mathbf{Q}}%

\global\long\def\Lbf{\mathbf{L}}%

\global\long\def\Pbf{\mathbf{P}}%

\global\long\def\Xbf{\mathbf{X}}%

\global\long\def\abm{\bm{a}}%

\global\long\def\bbm{\bm{b}}%

\global\long\def\cbm{\bm{c}}%

\global\long\def\dbm{\bm{d}}%

\global\long\def\ebm{\bm{e}}%

\global\long\def\fbm{\bm{f}}%

\global\long\def\gbm{\bm{g}}%

\global\long\def\hbm{\bm{h}}%

\global\long\def\pbm{\bm{p}}%

\global\long\def\qbm{\bm{q}}%

\global\long\def\rbm{\bm{r}}%

\global\long\def\sbm{\bm{s}}%

\global\long\def\tbm{\bm{t}}%

\global\long\def\ubm{\bm{u}}%

\global\long\def\vbm{\bm{v}}%

\global\long\def\wbm{\bm{w}}%

\global\long\def\xbm{\bm{x}}%

\global\long\def\ybm{\bm{y}}%

\global\long\def\zbm{\bm{z}}%

\global\long\def\Abm{\bm{A}}%

\global\long\def\Bbm{\bm{B}}%

\global\long\def\Cbm{\bm{C}}%

\global\long\def\Dbm{\bm{D}}%

\global\long\def\Ebm{\bm{E}}%

\global\long\def\Fbm{\bm{F}}%

\global\long\def\Gbm{\bm{G}}%

\global\long\def\Hbm{\bm{H}}%

\global\long\def\Ibm{\bm{I}}%

\global\long\def\Jbm{\bm{J}}%

\global\long\def\Lbm{\bm{L}}%

\global\long\def\Obm{\bm{O}}%

\global\long\def\Pbm{\bm{P}}%

\global\long\def\Qbm{\bm{Q}}%

\global\long\def\Rbm{\bm{R}}%

\global\long\def\Ubm{\bm{U}}%

\global\long\def\Vbm{\bm{V}}%

\global\long\def\Wbm{\bm{W}}%

\global\long\def\Xbm{\bm{X}}%

\global\long\def\Ybm{\bm{Y}}%

\global\long\def\Zbm{\bm{Z}}%

\global\long\def\lambm{\bm{\lambda}}%

\global\long\def\alphabm{\bm{\alpha}}%

\global\long\def\albm{\bm{\alpha}}%

\global\long\def\taubm{\bm{\tau}}%

\global\long\def\mubm{\bm{\mu}}%

\global\long\def\yrm{\mathrm{y}}%

\global\long\def\rone{\text{\ensuremath{\brbra{\textrm{I}}}}}%

\global\long\def\rtwo{\brbra{\text{II}}}%

\global\long\def\rthree{\brbra{\text{\textrm{III}}}}%

\global\long\def\rfour{\brbra{\text{\textrm{IV}}}}%

\global\long\def\rfive{\brbra{\text{V}}}%

\global\long\def\rsix{\brbra{\text{\textrm{VI}}}}%

\global\long\def\rseven{\brbra{\text{VI\textrm{I}}}}%

\global\long\def\reight{\brbra{\text{VI\textrm{I}I}}}%

\global\long\def\aleq{\overset{(a)}{\leq}}%

\global\long\def\bleq{\overset{(b)}{\leq}}%

\global\long\def\cleq{\overset{(c)}{\leq}}%

\global\long\def\dleq{\overset{(d)}{\leq}}%

\global\long\def\ageq{\overset{(a)}{\geq}}%

\global\long\def\bgeq{\overset{(b)}{\geq}}%

\global\long\def\cgeq{\overset{(c)}{\geq}}%

\global\long\def\beq{\overset{(b)}{=}}%

\global\long\def\ceq{\overset{(c)}{=}}%

\global\long\def\deq{\overset{(d)}{=}}%

\global\long\def\vbfp{\vbf_{\text{p}}}%

\global\long\def\vbfd{\vbf_{\text{d}}}%

\global\long\def\tp{t_{\text{p}}}%

\global\long\def\td{t_{\text{d}}}%

\global\long\def\tab{\qquad}%

\global\long\def\btab{\hspace{1.2cm}}%

\global\long\def\bbtab{\hspace{1.8cm}}%

\global\long\def\Lin{\operatorname{Lin}}%

\global\long\def\Span{\operatorname{Span}}%

\global\long\def\supp{\operatorname{supp}}%

\global\long\def\holder{\text{Hölder}}%
\global\long\def\apex{\text{APEX}}%
\global\long\def\pws{\text{PWS}}%
\global\long\def\rapex{\text{r}\apex}%
\global\long\def\flag{\text{\textbf{Flag}}}%
\global\long\def\false{\text{\textbf{\text{False}}}}%
\global\long\def\true{\text{\textbf{True}}}%
\global\long\def\Wcer{\text{W-certificate}}%

\global\long\def\Adet{\mbb A_{\text{det}}}%
\global\long\def\Azr{\mbb A_{\text{zr}}}%
\global\long\def\calZA{\mcal Z_{\mcal A}}%
\global\long\def\onestep{\text{One-Step}}%
\global\long\def\quarflag{\text{\textbf{Cert-Flag}}}%
\global\long\def\wrapex{\texttt{wrAPEX}}%
\global\long\def\maxquad{\texttt{MAXQUAD}}%

\newcommand{\jim}[1]{\textcolor{red}{\textbf{#1}}}
\newcommand{\lzw}[1]{\textcolor{blue}{\textbf{#1}}}

\global\long\def\holder{\text{Hölder}}%
\global\long\def\apps{\texttt{APPS}}%
\global\long\def\pws{\text{PWS}}%
\global\long\def\rapps{\text{r}\apps}%
\global\long\def\flag{\text{\textbf{Flag}}}%
\global\long\def\false{\text{\textbf{\text{False}}}}%
\global\long\def\true{\text{\textbf{True}}}%
\global\long\def\Wcer{\text{W-certificate}}%
\global\long\def\douflag{\text{\textbf{Dou-Flag}}}%

\global\long\def\Adet{\mbb A_{\text{det}}}%
\global\long\def\Azr{\mbb A_{\text{zr}}}%
\global\long\def\calZA{\mcal Z_{\mcal A}}%
\global\long\def\onestep{\text{One-step}}%
\global\long\def\halfflag{\text{\textbf{Half-Flag}}}%
\global\long\def\paramflag{\text{\textbf{Param-Flag}}}%
\global\long\def\none{\textbf{None}}%
\global\long\def\PWcer{\text{Penalty W-certificate}}%
\global\long\def\PWcerpair{\text{Paired-cut Penalty W-certificate}}%
\global\long\def\onestepplus{\text{One-Step}^{+}}%
\global\long\def\apexpplus{\texttt{APEX}^{+}}%
\global\long\def\xxx{\red{xxxx}}%
\global\long\def\adjustPara{\texttt{ParamAdjust}}%
\global\long\def\holder{\text{Hölder}}%
\global\long\def\apex{\text{APEX}}%
\global\long\def\pws{\text{PWS}}%
\global\long\def\rapex{\text{r}\apex}%
\global\long\def\flag{\text{\textbf{Flag}}}%
\global\long\def\Wcer{\text{W-certificate}}%
\global\long\def\halfflag{\text{\textbf{Half-Flag}}}%
\global\long\def\lowerflag{\text{\textbf{Lower-Flag}}}%
\global\long\def\Adet{\mbb A_{\text{det}}}%
\global\long\def\Azr{\mbb A_{\text{zr}}}%
\global\long\def\calZA{\mcal Z_{\mcal A}}%
\global\long\def\onestep{\text{One-Step}}%
\global\long\def\quadflag{\text{\textbf{Quad-Flag}}}%
\global\long\def\wrapex{\text{wrAPEX}}%
\global\long\def\xxxx{\red{xxxx}}%
\global\long\def\rapexW{\text{rAPEX-W}}%
\global\long\def\rapexC{\text{rAPEX-C}}%
\global\long\def\convexCert{\text{6.2}}%
\global\long\def\papex{\text{Penalty APEX}}%
\global\long\def\constrNum{{m}}
\global\long\def\activeconstrNum{\abs{\mcal A}}
\global\long\def\activeconstrSet{{\mcal A}}%
\global\long\def\guessTime{\mathcal{T}}

\global\long\def\bundleSize{B}%
\global\long\def\Cert{\text{Cert}}%
\global\long\def\pawg{\mathcal{AWG}^+}
\global\long\def\pagr{\mathcal{AGR}^+}
\global\long\def\pawgl{\mathcal{AWGL}^+}
\global\long\def\pagrl{\mathcal{AGRL}^+}
\global\long\def\activeConstrNum{m^+} 

\title{Accelerated Prox-Level Methods for Unknown Piecewise-Smooth Optimization II: Function-constrained Optimization}

\author{Zhenwei Lin\thanks{lin2193@purdue.edu, School of Industrial Engineering, Purdue University} \qquad\qquad\quad Zhe Zhang \thanks{zhan5111@purdue.edu, School of Industrial Engineering, Purdue University}  }

\maketitle

\begin{abstract}
    We introduce an anytime, almost parameter-free algorithm for convex function-constrained optimization, in which the objective and constraint functions are \textit{unknown} piecewise-smooth. Our algorithm, Restarted Penalty APEX (Penalty Accelerated Prox-level method
    for Exploring Piecewise Smoothness), is an accelerated bundle-level method based on the penalty
    approach. For problems satisfying quadratic growth, the proposed method is the first to achieve optimal oracle complexity without knowing the growth modulus or the exact penalty coefficient.
    Furthermore, Restarted Penalty APEX generates a verifiable certificate that bounds both the optimality gap and the constraint violation. This certificate also appears to be new to the literature.
\end{abstract}
\section{Introduction}

We focus on the following convex function-constrained optimization problem:
\begin{equation}
    \min_{x \in X} \ f(x) \quad \text{subject to} \quad \mcal G(x) := \max \left\{g_1(x),\, g_2(x),\, \ldots,\, g_{\constrNum}(x)\right\} \leq 0\ , \label{eq:initial_FP}
\end{equation}
where $X$ is a simple closed convex set, $f : X \to \mathbb{R}$ is a convex objective, and $g_i : X \to \mathbb{R}$ are convex constraint functions.
In this paper, we consider the penalty reformulation of~\eqref{eq:initial_FP}:
\begin{equation}
    \label{eq:exact_penalty_reformulation}
    \min_{x\in X}\
    \phi(x;\gamma)
    := f(x)+\gamma\bsbra{\mcal G(x)}_+ .
\end{equation}
Given a target accuracy $\vep>0$, our goal is to design an algorithm that computes a point $x_{\vep}\in X$ satisfying
\begin{equation}\label{eq:vep_optimal}
    \max\bcbra{f(x_{\vep}) - f^*,\, {\mcal G(x_{\vep})}} \leq \vep\, ,
\end{equation}
where $f^*$ denotes the optimal value of~\eqref{eq:initial_FP}.

We impose two additional structural assumptions.
First, the objective and constraint functions are \textit{unknown} piecewise-smooth (\pws) (see Definition~\ref{def:pws}).
Unknown PWS structure arises naturally in function-constrained models,  including CVaR constraints and hinge-based classification models~\cite{rockafellar00cvar,cheng2022functional,rigollet2011neyman,cotter2019optimization}.
Second, for {the optimal} dual multiplier $\lambda^*$, we assume the penalty function $\phi(\cdot;2\norm{\lambda^*}_1)$ satisfies a quadratic growth (QG) condition with modulus $\mu^*>0$:
\begin{equation}\label{eq:optimal_lag_qg}
    \phi(x;2\norm{\lambda^*}_1) - f^*
    \geq \frac{\mu^*}{2}\dist^2(x,X^*)\ ,\  \forall x\in X\ ,\
\end{equation}
where $X^*$ denotes the optimal solution set.
{QG condition of the optimal Lagrangian function and the optimal level-set function  \cite{lin2025adaptiveparameterfreeprojectionfreerestarting} imply \eqref{eq:optimal_lag_qg}. In general, when the objective or an active constraint is strongly convex or an optimal solution lies on the boundary of the feasible region, condition \eqref{eq:optimal_lag_qg} would be satisfied for some $\mu^*>0$. See Assumption~\ref{assu:quadratic_growth} for further details.}



Both the \pws{} structure and the QG condition fundamentally influence the convergence rates attainable by first-order methods.
A \(\pws\)
function can be treated as a generic nonsmooth function. Such a treatment
degrades the convergence rate
from linear to sublinear, even under the QG condition~\cite{zhang2025linearly,partOne}.
On the other hand,
error-bound conditions can yield improved complexity for first-order methods~\cite{drusvyatskiy2016errorboundsquadraticgrowth,bolte2017error}.
However, estimating the unknown QG modulus $\mu^*$ is particularly challenging as it depends on the penalty parameter $\gamma$ in \eqref{eq:exact_penalty_reformulation}. When $\gamma$ is much smaller than $\norm{\lambda^*}_1$, the penalty function $\phi(\cdot;\gamma)$ may fail to satisfy QG with any positive modulus.
When $\gamma$ is too large, the smoothness constant for $\phi(\cdot;\gamma)$~\footnote{
    Here the smoothness constant is defined in the piecewise smoothness constant of the penalty function $\phi(x;\gamma)$ (see Definition~\ref{def:pws}).} could be too large for efficient optimization.

Against this background, the literature on function-constrained optimization leaves
two major gaps.

\paragraph{\uline{Lack of parameter-free methods with optimal first-order oracle complexity for~\eqref{eq:initial_FP} under smoothness.}}


When the objective and constraint functions are $L$-smooth and the objective is $\mu$-strongly convex, \citet{zhang2022solving} establish the first-order oracle lower bound
$\Omega(\sqrt{(1+\norm{\lambda^*}_1)L/\mu^*}\log(1/\vep))$
and propose accelerated methods attaining the matching upper bound.
Achieving comparable complexity under the more general QG condition~\eqref{eq:optimal_lag_qg} remains less understood.

Recent work by~\citet{lin2025adaptiveparameterfreeprojectionfreerestarting} makes important progress toward this goal by developing a parameter-free level-set method under a related error-bound condition.
A quadratic error bound of this form implies~\eqref{eq:optimal_lag_qg}; see the discussion following Assumption~\ref{assu:quadratic_growth}.
It adapts a parallel root-finding scheme~\cite{renegar2016efficient,renegar2022simple} to search for the unknown optimal value $f^*$ by maintaining multiple level-set subproblems. Crucially, its prescribed number of parallel threads is determined by a computable upper bound on $\norm{\lambda^*}_1$, obtained from a strictly feasible point~\cite{nesterov2018lectures}.
The resulting complexity depends on this potentially conservative upper bound, {leading to a suboptimal complexity bound}. See Table~\ref{tab:complexity_comparison_function_constrained} for more details.
\paragraph{\uline{Lack of efficient and verifiable certificates for function-constrained optimization under PWS.}}
In function-constrained optimization, a point $x$ is commonly called an $\vep$-optimal solution~\cite{nesterov2018lectures} if~\eqref{eq:vep_optimal} holds. Because $f^*$ is generally unknown, this definition does not by itself give a directly verifiable stopping rule: even when a complexity result guarantees that an algorithm eventually reaches an $\vep$-optimal solution, certifying this accuracy for a particular candidate remains challenging.

Approximate KKT conditions~\cite{nocedal2006numerical} are another type of verifiable certificate and are widely used as practical stopping criteria, which is widely used as practical stopping criteria in constrained optimization~\cite{boob2025level,jia2025first,LiQu2021IPALM,xu2021augmented,Xu2022_small_constraints}.
However, in the $\pws$ setting studied here, KKT residuals do not directly provide the certificate we seek.  A typical stationarity requirement for an $\vep$-approximate KKT pair is
\[
    \dist \brbra{0, \partial f(x_{\epsilon})+\sum_{i=1}^{m}\lambda_{\varepsilon,(i)}\partial g_i(x_{\epsilon})+N_X(x_{\epsilon}) }\leq\varepsilon.
\]
For a $\pws$ function, the subdifferential mapping need not be continuous.
Consequently, the residual can remain bounded away from zero even arbitrarily close to a solution.

Taken together, these observations suggest a central question for function-constrained optimization under QG:
\[\ovalbox{\begin{minipage}{0.8\columnwidth - 2\fboxsep - 0.8pt}%
            \centering \it
            Is it possible to develop a parameter-free method for function-constrained optimization that achieves the optimal convergence rate under QG and, at the same time, produces a verifiable  termination certificate?
        \end{minipage}}
\]


In this paper, we provide an affirmative answer.
Building on the APEX framework developed in our previous work~\cite{partOne}, we systematically address both of the aforementioned gaps.
Our contributions are summarized as follows:
\begin{enumerate}

    \item \textbf{Penalty normalized Wolfe certificate for function-constrained optimization.}
          In Section~\ref{subsec:CWcer}, we apply an accelerated bundle-level procedure to the exact-penalty formulation~\eqref{eq:exact_penalty_reformulation} and use finitely many first-order oracle evaluations to construct a penalty normalized Wolfe certificate that extends the normalized Wolfe certificate for unconstrained PWS optimization~\cite{zhang2025linearly,partOne}.
          This certificate directly bounds the constraint violation and, together with the QG modulus, yields an upper bound on the optimality gap. These properties make it a computable and verifiable termination criterion for function-constrained optimization. 
Furthermore, the certified optimality-gap bound depends explicitly on the penalty parameter and QG modulus. This dependence enables the guess-and-check scheme used in the almost parameter-free algorithm described next.

    \item \textbf{Anytime, almost parameter-free algorithm.}
          We introduce Restarted $\papex$, an accelerated bundle-level method that uses guess-and-check restarts to adaptively estimate both the penalty coefficient and the QG modulus.
          The method is anytime: it does not require the target accuracy $\vep$ as input.
          The qualifier "\textit{almost}" refers solely to the bundle-size condition $\bundleSize\geq k_f+k_{\mcal G}$ required by the stated worst-case PWS complexity bound, where $k_f$ and $k_{\mcal G}$ denote the numbers of smooth pieces in $f$ and $\mcal G$, respectively. The method requires no prior knowledge of the penalty coefficient or the QG modulus.
          Notably, when all functions are smooth, Restarted $\papex$ becomes fully parameter-free with bundle size $\bundleSize=\activeConstrNum+1$. Here, $\activeConstrNum$ denotes the number of constraints that are not uniformly inactive on $X$.
          More generally, if every function in~\eqref{eq:initial_FP} is $(k,L)$-\pws, the method preserves the logarithmic dependence on $\vep$ with complexity $O(\activeConstrNum k\sqrt{\norm{\lambda^*}_1L/\mu}\log(1/\vep))$.  To the best of our knowledge, this is the first anytime, almost parameter-free method with this guarantee. Full complexity results are summarized in Table~\ref{tab:complexity_comparison_function_constrained}.
    \item \textbf{Promising practical performance.}
          We evaluate our approach on three benchmark classes: convex QCQPs, Neyman-Pearson classification (NPC), and fairness-aware classification. The simulated QCQP instances let us control smoothness and the number of constraints, providing a controlled test of the theory. The NPC and fairness-aware classification experiments test the method on real-world datasets. Across these settings, the proposed almost parameter-free algorithm exhibits strong practical performance for function-constrained optimization.
\end{enumerate}
\begin{table}[htbp]
    \centering
    \caption{
        First-order oracle complexity comparison for function-constrained optimization when each component is $L$-smooth.
        "$\bar{\gamma}$" is an upper bound on $\norm{\lambda^*}_1$ obtained from a strictly feasible point.
        \label{tab:complexity_comparison_function_constrained}%
    }
    \small
    \resizebox{\textwidth}{!}{%
        \begin{tabular}{@{}lccccc@{}}
            \toprule
            \multicolumn{1}{c}{\multirow{2}[2]{*}{Methods}}                 & Parameter-                    & Any-                          & Handles QG                    & \multicolumn{1}{c}{\multirow{2}[2]{*}{Certificate}} & \multicolumn{1}{c}{First-order oracle complexity}                                             \\
                                                                            & free                          & time                          & \eqref{eq:optimal_lag_qg}     &                                                     & \multicolumn{1}{c}{(up to constant)}                                                          \\
            \midrule
            TIS~\cite{deng2024uniformly}                                    & \ding{51}                     & \ding{51}                     & \ding{55}                     & \ding{51}                                           & $\min\bcbra{\norm{\lambda^*}_1\log\brbra{\norm{\lambda^*}_1},\,\log(1/\vep)}\big/\sqrt{\vep}$ \\
            ACGD-S~\cite{zhang2022solving}                                  & \ding{55}                     & \ding{51}                     & \ding{55}                     & \ding{51}                                           & $\sqrt{\norm{\lambda^*}_1L/\mu^*}\log(1/\vep)$                                                \\
            UFCM~\cite{zoll2025universally}                                 & \ding{55}                     & \ding{55}                     & \ding{55}                     & \ding{55}                                           & $\sqrt{\norm{\lambda^*}_1L/\mu^*}\log(1/\vep)$                                                \\
            RLS~\cite{lin2025adaptiveparameterfreeprojectionfreerestarting} & \ding{51}                     & \ding{55}                     & \ding{51}                     & \ding{55}                                           & $\bar{\gamma}^2\log(\bar{\gamma})\sqrt{L/\mu^*}(\log(1/\vep))^2$                              \\\cmidrule{2-6}
            {Restarted Penalty }                                            & \multirow{2}[1]{*}{\ding{51}} & \multirow{2}[1]{*}{\ding{51}} & \multirow{2}[1]{*}{\ding{51}} & \multirow{2}[1]{*}{\ding{51}}                       & \multirow{2}[1]{*}{$\sqrt{\norm{\lambda^*}_1L/\mu^*}\log\brbra{1/\vep}$}                      \\
            APEX (Alg~\ref{alg:rAPEXC})                                     &                               &                               &                               &                                                     &                                                                                               \\
            \bottomrule
        \end{tabular}
    }
\end{table}%
\subsection{Related work}
For clarity and brevity, our literature review focuses on the works most closely related to function-constrained optimization, and we refer readers to~\cite{partOne} for a more comprehensive discussion of bundle-level methods and certificates for optimization without functional constraints.

\paragraph{First-order methods for function-constrained optimization.} For large-scale function-constrained optimization, existing first-order methods can be broadly organized into several lines: penalty-based methods and augmented Lagrangian approaches~\cite{LanMon13-1,xu2021iteration,xu2021first,Xu2022_small_constraints}, primal-dual methods~\cite{chambolle2011first,xu2020primal,hamedani2021primal,NEURIPS2024_8d8e060d,boob2019proximal}, switching subgradient methods~\cite{huang2023oracle,swtichgradientpolyak}
and level-set approaches~\cite{aravkin2019level,deng2024uniformly,lin2018level_finite_sum,lin2018level,lin2025adaptiveparameterfreeprojectionfreerestarting,nesterov2013introductory,nesterov2018lectures}.
Much of this literature measures computational complexity by the number of matrix-vector products.
In related affinely constrained smooth convex settings, ~\citet{ouyang2018lowercomplexityboundsfirstorder}
establish lower bounds that imply an $\Omega(1/\sqrt{\vep})$ dependence under strong convexity, and several methods attain the optimal computational complexity dependence.

On the other hand, from the perspective of first-order oracle complexity,~\citet{zhang2022solving} showed that function-constrained optimization can be nearly as easy as unconstrained optimization and proposed an accelerated conditional gradient method with complexity $\Ocal(\sqrt{\bar{\gamma}L/\mu}\log(1/\vep))$ for strongly convex objective functions;~\citet{zoll2025universally} extended this result to the setting in which either the objective or one active constraint is strongly convex. These methods, however, are not parameter-free and remain restricted to strongly convex settings. Error-bound and quadratic-growth conditions provide a more general route to fast convergence.
The closest work to ours is~\citet{lin2025adaptiveparameterfreeprojectionfreerestarting}, which develops a parameter-free level-set framework under an error-bound condition. Although it relaxes strong convexity, it does not combine $\norm{\lambda^*}_1$ adaptation, QG adaptation, and anytime execution. Table~\ref{tab:complexity_comparison_function_constrained} summarizes representative first-order oracle complexity results.

For the exact-penalty reformulations~\eqref{eq:exact_penalty_reformulation}, if $\gamma>\norm{\lambda^*}_1$, then the minimizers of the penalty function coincide with the minimizers of the original problem~\eqref{eq:initial_FP}~\cite{han1979exact}. Thus, if a valid upper bound on $\norm{\lambda^*}_1$ is known, one can solve the penalty reformulation with a fixed penalty coefficient, such as those in~\cite{adcock2025restarts,partOne,roulet2017sharpness}. However, in practice, $\norm{\lambda^*}_1$ is often unknown and the upper bound estimates may be conservative.

\paragraph{Certificate for function-constrained optimization.}
Approximate KKT conditions are widely used as termination criteria for constrained
optimization. In the $\pws$ setting considered here, the resulting residual may be unstable. Therefore, \citet{grimmer2023goldstein} proposed Goldstein-KKT stationarity to address nonsmoothness by using nearby subgradients for Lipschitz functional constraints; however, this framework is formulated for problems over $\mathbb{R}^n$ with functional inequalities and does not directly incorporate an additional simple set $X$.

From a primal--dual perspective, a standard certificate of approximate
optimality is the restricted primal--dual gap, which is widely used in
saddle-point analysis
\cite{chambolle2011first,hamedani2021primal,Nemirovski2005mirror}.
Let
$\Lambda_r
    :=
    \{\lambda\in\mathbb R_+^m:\|\lambda\|_2\leq r\}.$
For a candidate pair $(x_{\vep},\lambda_{\vep})$ with
$\lambda_{\vep}\in\Lambda_r$, the restricted primal--dual gap is
\begin{equation*}
    \Delta(x_{\vep},\lambda_{\vep} ) = \max_{\lambda\in \Lambda_{r}}\bcbra{f(x_\vep) + \sum_{i=1}^{m}\lambda_{(i)}g_i (x_{\vep})} - \min_{x\in X} \bcbra{f(x) + \sum_{i=1}^{m}\lambda_{\vep,(i)}g_{i}(x)}\ .
\end{equation*}
Its direct evaluation requires both a suitable dual radius $r$ and the
solution of a global Lagrangian minimization problem over $X$.
To obtain a practical certificate, \citet{zhang2022solving} introduce
a model-based PD-gap for the adaptive ACGD-S method. In their composite
formulation, where $F=f+u$, it takes the form
\begin{equation*}
    \Delta_{\rm PD,t}(r)
    :=
    F(\bar x_t)
    +
    r\sqrt{\sum_{i=1}^{\constrNum} \bsbra{g_i(\bar x_t)}_+^2}
    -
    F_{{\rm PD},t}\ ,\
    F_{{\rm PD},t}
    :=
    \min_{x\in X}
    \left\{
    \ell_f^t(x)+u(x)
    +\sum_{i=1}^m\ell_{g_i}^t(x)
    \right\},
\end{equation*}
where $\ell_f^t$ and $\ell_{g_i}^t$ are weighted affine minorants
constructed from the primal-dual iterates and $u(x)$ is a simple regularizer. This certificate bounds objective suboptimality for every
$r\geq0$ and also controls constraint violation when $r$ is sufficiently
large relative to the norm of an optimal multiplier.
A level-set-based certificate is developed by
\citet{deng2024uniformly}. They consider the value function
$V(\eta):=
    \min_{x\in X}v(x,\eta),
    v(x,\eta)
    :=
    \max\{f(x)-\eta,\mcal G(x)\}.$
For each level $\eta_t$, the APL~\cite{lan2015bundle} subroutine maintains computable bounds $l_t\leq V(\eta_t)\leq
    v(x_t,\eta_t),$ where the lower bound $l_t$ is calculated by minimizing the affine minorants constructed from past iterates over the set $X$.
A common limitation of these model-based certificates is their reliance
on a global lower-model problem. \citet{deng2024uniformly} explicitly assume that
$X$ is bounded, whereas the ACGD-S~\cite{zhang2022solving} certificate requires the
affine-composite problem defining $F_{{\rm PD},t}$ to be bounded below
and efficiently solvable. In particular, on an unbounded domain without
coercive regularization, $F_{{\rm PD},t}$ may equal $-\infty$, making
the certificate uninformative.

\subsection{Outline}
The remainder of the paper is organized as follows. We conclude this section by introducing notation. Section~\ref{sec:preliminaries} reviews the necessary preliminaries. Section~\ref{sec:penalty_apex} develops the penalty accelerated prox-level method for exact-penalty reformulations. Section~\ref{subsec:CWcer} introduces the penalty normalized Wolfe certificate, and Section~\ref{sec:rpapex} presents the  Restarted \papex{} algorithm. Section~\ref{sec:numerical_study} reports the numerical results.

\subsection{Notation}

Throughout the paper, we use the following notation. For \(n\in\mathbb{N}_{+}\), let \([n]:=\{1,\dots,n\}\). For \(v\in\mathbb{R}^{n}\), define \(\norm v_{q}:=\brbra{\sum_{i=1}^{n}\abs{v_{(i)}}^{q}}^{1/q}\), where \(v_{(i)}\) is the \(i\)-th component of \(v\). Unless stated otherwise, \(\norm{\cdot}\) denotes the Euclidean norm. For \(u,v\in\mathbb{R}^{n}\), let \(\inner uv:=\sum_{i=1}^{n}u_{(i)}v_{(i)}\). For \(x\in\mathbb{R}^{n}\) and \(X\subseteq\mathbb{R}^{n}\), define \(\dist(x,X):=\inf_{\bar{x}\in X}\norm{x-\bar{x}}\). For a convex set \(X\subseteq\mathbb{R}^{n}\) and \(x\in X\), define the normal cone \(N_X(x):=\bcbra{v\in\mathbb{R}^{n}:\inner{v}{y-x}\leq0,\ \forall y\in X}\); for \(x\notin X\), set \(N_X(x):=\emptyset\). We write \(f'(x)\) for a subgradient of \(f\) at \(x\); if \(f\) is differentiable at \(x\), then \(f'(x)=\nabla f(x)\). The linearization of \(f\) at \(\bar{x}\) is \(\ell_{f}\brbra{x;\bar{x}}:=f\brbra{\bar{x}}+\inner{f^{\prime}\brbra{\bar{x}}}{x-\bar{x}}\). For an event \(A\), let \(\onebf_{ A}=1\) if \(A\) is true and \(\onebf_{ A}=0\) otherwise. The Euclidean ball centered at \(\bar{x}\) with radius \(\iota\) is \(\mcal B\brbra{\bar{x},\iota}:=\bcbra{x:\norm{x-\bar{x}}\leq\iota}\), and for \(x\in\mathbb{R}\), let \([x]_+:=\max\{0,x\}\).
For \(\gamma\geq0\), define the exact penalty function by \(\phi\brbra{x;\gamma}:=f(x)+\gamma\bsbra{\mcal G(x)}_{+}\).  Given a finite set of oracle evaluation points $\mcal P$, define $\psi_{\mcal P}^{+}(x;\gamma):=\max_{z\in \mcal P}\bcbra{\ell_f(x;z)} + \gamma\bsbra{\max_{z\in \mcal P}\bcbra{\ell_{\mcal G}(x;z)}}_{+}$.
In algorithm calls, we use the vertical bar $\mid$ to separate quantities
that may be updated from fixed algorithmic parameters; arguments after
$\mid$ are chosen once and remain unchanged throughout the entire algorithm.

\section{Preliminaries\label{sec:preliminaries}}
In this section, we review the main concepts used throughout the paper, including piecewise smoothness (Definition~\ref{def:pws}) and  quadratic growth (Definition~\ref{def:QG}). Throughout the paper, we assume that Assumptions~\ref{assu:dual_exist},~\ref{assu:pws_objective_con} and~\ref{assu:quadratic_growth} hold.
\begin{defn}[$(k,L)$-piecewise smoothness, $(k,L)$-\pws\label{def:pws}]
    Let $k\in\mathbb{N}_{+}$ denote the number of pieces, and let $L>0$
    denote the piecewise smoothness constant. A function $f:X\to\mbb R$ is
    $(k,L)$-piecewise smooth if there exist $k$ subsets
    $\bcbra{X_i}_{i=1}^{k}$ of $X$ that cover $X$ such that each restriction
    $f_i:=f_{\mid X_i}$ is $L$-smooth for every $i\in\bsbra{k}$.
    Specifically, we assume access to a first-order oracle $f^{\prime}(x)$
    satisfying
    \begin{equation}\label{eq:pws_L_choice}
        f\brbra x-\brbra{f(\bar{x}) + \inner{f^{\prime}(\bar{x})}{x-\bar{x}}}
        \leq \frac{L}{2}\norm{x-\bar{x}}^{2},
        \qquad \forall i\in\bsbra{k},\quad \forall x,\bar{x}\in X_i.
    \end{equation}
\end{defn}
Further discussion of this definition can be found in~\cite{partOne}.
This convention is analogous to the standard one-dimensional notion of a piecewise $C^1$ path~\cite[Sec.~10.8]{rudin1974real}.
We next introduce the quadratic growth condition, which is a standard assumption for achieving fast convergence rates in convex optimization and is  weaker than strong convexity.
\begin{defn}
    [Quadratic growth, QG\label{def:QG}] We say that $f:X\to\mbb R$
    satisfies $\mu$-quadratic growth if
    $
        f\brbra x-f^{*}\geq {\mu}\dist^{2}\brbra{x,X^{*}}/2$ holds for all $x\in X$,
    where $X^{*}$ is the nonempty minimizer set, i.e., $X^{*}=\argmin_{x\in X}f\brbra x$.
\end{defn}

We next formalize the required properties of the optimal dual multipliers.
\begin{assumption}\label{assu:dual_exist}
    The set $\Lambda^*$ of all optimal dual multipliers satisfying the KKT conditions for~\eqref{eq:initial_FP} is nonempty and bounded. Moreover, suppose that a minimum-$\ell_1$-norm multiplier $\lambda^*\in\argmin_{\lambda\in\Lambda^*}\bcbra{\norm{\lambda}_1}$ satisfies $\norm{\lambda^*}_1\geq 1$.
\end{assumption}

We next impose the following additional piecewise smoothness assumption.

\begin{assumption}\label{assu:pws_objective_con}
    Suppose $f$ is $(k_f,L_f)$-$\pws$ and each $g_i$ is $(k_{g_i},L_{g_i})$-$\pws$. Let
    \[
        \activeconstrSet:=\bcbra{i\in[\constrNum]: \sup_{x\in X}g_i(x)\geq0}
    \]
    denote the set of potentially active constraints. Then
    ${\mcal G(x)}=\max_{i\in\activeconstrSet} g_i(x)$
    is $(k_{\mcal G},L_{\mcal G})$-$\pws$, with
    $k_{\mcal G}\leq \sum_{i\in \activeconstrSet}k_{g_i}$ and
    $L_{\mcal G}\leq \max_{i\in \activeconstrSet}L_{g_i}$. Furthermore, we define $\activeConstrNum = |\activeconstrSet|$ as the number of potentially active constraints.
\end{assumption}

In addition to the existence of an optimal dual multiplier (Assumption~\ref{assu:dual_exist}) and the $\pws$ property of the  penalty function (Assumption~\ref{assu:pws_objective_con}), we also require the following QG condition for the penalty function:
\begin{assumption}
    \label{assu:quadratic_growth}
    Define
    $\phi\brbra{x;2\norm{\lambda^*}_1}
        =f\brbra x+2\norm{\lambda^*}_1\bsbra{\mcal G\brbra x}_{+}$. We assume that
    $\phi\brbra{x;2\norm{\lambda^*}_1}$ has a positive optimal QG modulus, denoted by
    \[
        \mu^* := \sup \bcbra{\mu>0:\phi\brbra{x;2\norm{\lambda^*}_1} \text{ is } \mu\text{-QG}}>0\ .
    \]
\end{assumption}
We provide several remarks regarding Assumption~\ref{assu:quadratic_growth}. First, although the assumption is stated for the exact-penalty function, it follows from several familiar conditions on the original problem:
\begin{enumerate}
    \item The objective function $f(x)$ is $\mu^*$-strongly convex;
    \item If each $g_i(x)$ is $\mu_i$-strongly convex~\cite{lin2018level,NEURIPS2024_8d8e060d} and Assumption~\ref{assu:dual_exist} holds with $\norm{\lambda^*}_1>0$, then the $\phi(x; 2\norm{\lambda^*}_1)$ is QG with modulus at least $\mu^* = \sum_{i=1}^{m}\lambda^*_{(i)}\mu_i$;
    \item If the optimal Lagrangian $\mathcal{L}(x, \lambda^*) = f(x) + \sum_{i=1}^{\constrNum} \lambda_{(i)}^* g_i(x)$ is $\mu^*$-QG, then, noting that $\sum_{i=1}^{\constrNum} \lambda_{(i)}^* g_i(x) \leq \norm{\lambda^*}_1 [\mcal G(x)]_{+}$, it follows that $\phi(x; 2\norm{\lambda^*}_1)$ satisfies $\mu^*$-QG.
\end{enumerate}

Second, even if neither $f(x)$ nor $\mcal G(x)$ individually satisfies the QG property, it remains possible for the optimal Lagrangian to exhibit quadratic growth. For example, consider the setting from~\cite{guigues2024universal} where $f(x) =  e^{x}$ and $\mcal G(x) = e^{-x} - 1 \leq 0$. In this case, the optimal dual multiplier is $\lambda^* = 1$. Although $f(x)$ and $\mcal G(x)$ alone do not satisfy the QG property, the optimal Lagrangian $\mathcal{L}(x, \lambda^*) =  e^{x} +  (e^{-x} - 1)$ is $2$-QG. This matches the third scenario described above, showing that the corresponding penalty function has QG modulus at least $2$ and therefore satisfies Assumption~\ref{assu:quadratic_growth}.

Third, Lin et al.~\cite{lin2025adaptiveparameterfreeprojectionfreerestarting} consider an error bound condition for~\eqref{eq:initial_FP} that, in the context of QG, takes the form $\max\big\{f(x)-f^*,\,\mcal G(x)\big\}\geq {\mu}\dist^2\brbra{x, X^*}/{2}$. This reflects a particular error bound condition for the level set function, where the roles of the objective and constraint are treated symmetrically.
Suppose the norm of an optimal dual variable, $\norm{\lambda^*}_1> 0$, is known. Then one can consider the equivalent formulation:
\begin{equation*}
    \min_{x\in X} \ f(x)  \quad \text{s.t.}\  \norm{\lambda^*}_1 \mcal G(x) \leq 0\ .
\end{equation*}
This allows us to analogously define a QG property for the new level set function:
$$
    \max\big\{f(x)-f^*,\,\norm{\lambda^*}_1 \mcal G(x)\big\} \geq {\mu^*}\dist^2(x,X^*)/{2}.
$$
Moreover, the following inequalities hold:
\begin{equation*}
    \begin{aligned}
         & \max\big\{f(x)-f^*,\,\norm{\lambda^*}_1 \mcal G(x)\big\}
        \leq \max\big\{f(x)-f^*+\norm{\lambda^*}_1 [\mcal G(x)]_+,\,\norm{\lambda^*}_1 [\mcal G(x)]_+\big\} \\
         & \leq f(x)-f^* + 2\norm{\lambda^*}_1 [\mcal G(x)]_+
        = \phi(x; 2\norm{\lambda^*}_1) - f^*.
    \end{aligned}
\end{equation*}
Consequently, the error bound condition considered in~\cite{lin2025adaptiveparameterfreeprojectionfreerestarting} implies Assumption~\ref{assu:quadratic_growth}, up to a constant factor adjustment.

\section{Penalty APEX for exact-penalty reformulations\label{sec:penalty_apex}}

We introduce {\papex} (Penalty \uline{A}ccelerated \uline{P}rox-level method for \uline{Ex}ploring Piecewise Smoothness), an extension of {\apex} for convex function-constrained optimization. For a fixed penalty parameter $\gamma$, {\papex} applies the accelerated bundle-level mechanism to the exact penalty function $\phi(x;\gamma)=f(x)+\gamma[\mcal G(x)]_+$. Like {\apex}, {\papex} adopts a double-loop structure: each outer iteration, implemented by $\onestepplus$ in Algorithm~\ref{alg:One-step_plus}, consists of $\bundleSize$ inner steps, where $\bundleSize$ denotes the bundle size. At each inner step, the method evaluates first-order information at the averaged points $\underline{x}^{t,i}$, projects the fixed reference point $\bar{y}$ onto the current level set $X(t,i)$, and updates the incumbent iterate $\hat{x}^{t + 1}$ so that the penalty objective value is non-increasing.

The key difference from {\apex} is the cutting-plane model. For any point set $\mcal P$, the following inequality holds:
\begin{equation}\label{eq:key_observation}
    \max_{z\in\mcal P}\bcbra{\ell_{f}(x;z)+\gamma\ell_{\bsbra{\mcal G}_{+}}(x;z)}\leq\max_{z\in\mcal P}\bcbra{\ell_{f}(x;z)}+\gamma\bsbra{\max_{z\in\mcal P}\bcbra{\ell_{\mcal G}(x;z)}}_{+}\leq f(x)+\gamma[\mcal G(x)]_{+}\ ,
\end{equation}
where $\ell_{\bsbra{\mcal G}_{+}}(x;z) = \bsbra{\mcal G(z)}_+ + \inner{\bsbra{\mcal G(z)}_+^\prime}{x-z}$.
Thus, {\papex} replaces the standard cutting-plane model in {\apex} (the first term in~\eqref{eq:key_observation}) with the penalty model (the middle term in~\eqref{eq:key_observation}). This modification is tailored to the exact-penalty structure and extends the {\apex} framework of~\cite{partOne} to function-constrained optimization.

\begin{algorithm}[h]
    \caption{$\protect\onestepplus\protect\brbra{\protect\gamma,\bar{y},\hat{x}^{t},x^{t,0},\tilde{l},\bundleSize}$\protect\label{alg:One-step_plus}}
    \begin{algorithmic}[1]
        \State \textbf{Initialize:} $x^{t+1,0}=x^{t,0},\hat{x}^{t+1} = \hat{x}^{t}$, $\phi\brbra{x;\gamma}= f\brbra x+\gamma\bsbra{\mcal G\brbra x}_{+}$, $\bar{X}(t,0)=\bcbra{x:\inner{x-x^{t,0}}{x^{t,0}-\bar{y}}\geq0}$
        \For{$i=1,\ldots,\bundleSize$}
        \State $\underline{x}^{t,i}=\brbra{1-\alpha_{t}}\hat{x}^{t}+\alpha_{t}x^{t,i-1}$
        \State $X(t,i)=\bcbra{x\in X:\max_{j\in[i]}\bcbra{\ell_{f}(x;\underline{x}^{t,j}) + \gamma[\max_{j\in[i]}\ell_{\mcal G}(x;\underline{x}^{t,j})]_+} \leq \tilde{l}}\cap\bar{X}\brbra{t,0}$
        \State Compute
        \begin{equation}
            x^{t,i}\leftarrow\argmin_{x\in X(t,i)}\norm{x-\bar{y}}^{2}\label{eq:subproblems_plus}
        \end{equation}
        \If{the subproblem is infeasible}
        \State \textbf{return} $\brbra{\hat{x}^{t+1},x^{t+1,0},\true}$\label{enu:infeasible_plus}\Comment{$\true$: infeasibility implies a  certificate}
        \EndIf
        \State $\bar{x}^{t,i}=\brbra{1-\alpha_{t}}\hat{x}^{t}+\alpha_{t}x^{t,i}$
        \State Choose $\hat{x}^{t+1}$ such that $\phi\brbra{\hat{x}^{t+1};\gamma}=\min\bcbra{\phi\brbra{\hat{x}^{t};\gamma},\min_{j\leq i}\bcbra{\phi\brbra{\bar{x}^{t,j};\gamma}}}$ and set $x^{t+1,0}=x^{t,i}$
        \EndFor
        \State \textbf{return} $\brbra{\hat{x}^{t + 1},x^{t+1,0},\false}$\label{enu:return_output}\Comment{$\false$: no certificate}
    \end{algorithmic}
\end{algorithm}

\begin{algorithm}
    \caption{\uline{P}enalty \uline{A}ccelerated \uline{P}rox-level Method for \uline{Ex}ploring Piecewise Smoothness ($\protect\papex$)}
    \label{alg:papex}
    \begin{algorithmic}[1]
        \Require penalty parameter $\gamma$, candidate solution $\bar{y}$, test level $\tilde{l}$, and the number of cuts $\bundleSize$
        \State \textbf{Initialize:} $x^{1}=\bar{y},\hat{x}^{1}=\bar{y}$
        \For{$t=1,\ldots,N$}
        \State $\brbra{\hat{x}^{t+1},x^{t+1},\ldots}\leftarrow\onestepplus\brbra{\gamma, \bar{y},\hat{x}^{t},x^{t},\tilde{l},\bundleSize}$
        \EndFor
        \State \textbf{return} $\hat{x}^{N+1}$
    \end{algorithmic}
\end{algorithm}

As in the analysis of $\apex$ in~\cite{partOne}, the convergence bound for $\papex$ is expressed through empirical Lipschitz quantities. Define
\begin{equation}
    \begin{aligned}
        L_{\gamma,t}\brbra{r,l}=\bsbra{\phi(\hat{x}^{t+1};\gamma) - \tilde{l} - (1-\frac{3}{4}\alpha_t)(\phi(\hat{x}^{t};\gamma) - \tilde{l})}_+ \Big/\brbra{\alpha_t^2\|x^{t,r} - x^{t,l}\|^2/2} \\
        L_{\gamma,t}
        := \min_{0\leq l<r\leq \bundleSize}L_{\gamma,t}(r,l)\ ,\
        \bar{L}_{\gamma}\brbra N
        :=\frac{\sum_{t=1}^{N}\omega_{t}\alpha_{t}^{2}L_{\gamma,t}\onebf_{\mcal E(t)}\norm{x^{t,r_{t}}-x^{t,l_{t}}}^{2}/2}{\sum_{t=1}^{N}\norm{x^{t,r_{t}}-x^{t,l_{t}}}^{2}} .
    \end{aligned}
    \label{eq:L_phi_N}
\end{equation}
Here, the slow-descent event is defined as
\begin{equation}\label{eq:slow_event}
    \mcal E(t)=\bcbra{\phi\brbra{\hat{x}^{t+1};\gamma}-\tilde{l}>\brbra{1-\frac{\alpha_t}{2}}\brbra{\phi(\hat{x}^{t};\gamma)-\tilde{l}}}.
\end{equation}
This event occurs when the new residual
$\phi\brbra{\hat{x}^{t+1};\gamma}-\tilde{l}$ fails to achieve the benchmark
contraction of the previous residual at rate $1-\alpha_t/2$. On such
iterations, the proof uses the additional curvature correction measured by
$L_{\gamma,t}$. When $\mcal E(t)$ does not occur, the benchmark contraction
already holds, and the indicator $\onebf_{\mcal E(t)}$ removes this correction
from the numerator of $\bar L_{\gamma}\brbra N$. Thus, the numerator of
$\bar L_{\gamma}\brbra N$ collects curvature contributions only from
slow-descent iterations. We use the convention $0/0=0$; if no slow-descent
iteration occurs and the denominator is positive, then
$\bar L_{\gamma}\brbra N=0$.

The next proposition gives the basic convergence properties of Algorithm~\ref{alg:papex}.
\begin{proposition}\label{prop:property_papex}
    Assume that the sequence $\bcbra{\hat{x}^{t}}$ generated by Algorithm~\ref{alg:papex} satisfies $\phi\brbra{\hat{x}^{t};\gamma}\geq\tilde{l}$ for all $t\geq1$ and that every subproblem of Algorithm~\ref{alg:One-step_plus} is feasible, and that the weights $\bcbra{\omega_{t}}$ satisfy $\omega_{t}\brbra{1-\alpha_{t}/2}=\omega_{t-1}$ for $t\geq2$. Then the following properties hold:
    \begin{enumerate}
        \item \textbf{Monotonicity of proximal distances:} For every $t\geq1$ and $0\leq i<\bundleSize$, it holds that $\norm{x^{t,i}-\bar{y}}\leq\norm{x^{t,i+1}-\bar{y}}$.
        \item \textbf{Convergence bound:} The quantity $\bar{L}_{\gamma}\brbra t$ (defined in~\eqref{eq:L_phi_N}) satisfies the following inequality:
              $$\omega_{t}\brbra{\phi\brbra{\hat{x}^{t+1};\gamma}-\tilde{l}}-\omega_{1}\brbra{1-\frac{\alpha_{1}}{2}}\brbra{\phi\brbra{\hat{x}^{1};\gamma}-\tilde{l}}\leq\bar{L}_{\gamma}\brbra t\norm{x^{t+1}-\bar{y}}^{2}\ , \ \forall t\geq 2.$$
    \end{enumerate}
\end{proposition}

\noindent\emph{Proof deferred.} The detailed proof is collected in Subsection~\ref{subsec:papex-detailed-proofs}.

The following theorem specializes Proposition~\ref{prop:property_papex} to the ideal level $\tilde{l}=f^*$ and the weights $\omega_t$ and $\alpha_t$ used by $\papex$, yielding a sublinear convergence bound.
\begin{thm}\label{thm:papex_f_star}
    Set $\omega_t = (t+2)(t+3)/2$ and $\alpha_t = 4 / (t+3)$ in $\papex$.  Then, for the sequence $\bcbra{\hat{x}^t}$ generated by $\papex$ with level $\tilde{l}=f^*$ and $\gamma > \norm{\lambda^*}_1$, then the quantity  $\bar{L}_{\gamma}\brbra{t}$ defined in~\eqref{eq:L_phi_N} satisfies
    \begin{equation}\label{eq:ideal_descent}
        \phi(\hat{x}^{t+1};\gamma)-f^* \leq \frac{6}{(t+2)(t+3)}\brbra{\phi(\hat{x}^{1};\gamma)-f^*}+\frac{2\bar{L}_{\gamma}(t)}{(t+2)(t+3)}\dist^2\brbra{ \bar{y}, X^*}\ .
    \end{equation}
\end{thm}

\noindent\emph{Proof deferred.} The detailed proof is collected in Subsection~\ref{subsec:papex-detailed-proofs}.

Although $\bar{L}_{\gamma}\brbra{t}$, the average of $L_{\gamma,t}$ defined
in~\eqref{eq:L_phi_N}, is computable in practice, it remains unclear
how to bound this quantity in Proposition~\ref{prop:property_papex}
and Theorem~\ref{thm:papex_f_star}. Proposition~\ref{prop:upper_bound_Lip}
shows that $\bar{L}_{\gamma}\brbra{t}$ can be upper bounded by a weighted
average of empirical Lipschitz smoothness constants
$\tilde{L}_{\gamma,t}$ in~\eqref{eq:defn_L_tilde_phi}, where $\tilde{L}_{\gamma,t}$ denotes the minimum local empirical-curvature estimate
between $\bar{x}^{t,r}$ and
$\underline{x}^{t,l+1}$ for $r>l$ in the $t$-th outer iteration.
\begin{proposition}\label{prop:upper_bound_Lip}
    Let $\bcbra{\hat{x}^{t},x^{t,i}}$ be the iterates generated by Algorithm~\ref{alg:papex} with $\phi(\hat{x}^{t};\gamma)>\tilde{l}$.
    Assume that, for every outer iteration $t\geq1$, there exist indices
    $0\leq l_{f,t},l_{\mcal G,t}<\tilde r_{t}\leq \bundleSize$ that attain the minimum in
        {\small
            \begin{equation}\label{eq:defn_L_tilde_phi}
                \tilde{L}_{\gamma,t}=\min_{\substack{l_{f},l_{\mcal G},r:
                \\
                0\leq l_{f},l_{\mcal G}<r\leq\bundleSize}}
                \frac{\bsbra{\phi\brbra{\bar{x}^{t,r};\gamma}-\Big(\ell_{f}\brbra{\bar{x}^{t,r};\underline{x}^{t,l_{f}+1}}+\gamma\bsbra{\ell_{\mcal G}\brbra{\bar{x}^{t,r};\underline{x}^{t,l_{\mcal G}+1}}}_{+}\Big)-\frac{\alpha_{t}}{4}\Big(\phi\brbra{\hat{x}^{t};\gamma}-\tilde{l}\Big)}_+}{\frac{1}{2}\Big(\norm{\bar{x}^{t,r}-\underline{x}^{t,l_{f}+1}}^{2}+\norm{\bar{x}^{t,r}-\underline{x}^{t,l_{\mcal G}+1}}^{2}\Big)},
            \end{equation}}
    where the quotient in~\eqref{eq:defn_L_tilde_phi} is interpreted with the convention \(0/0=0\).
    Then $L_{\gamma,t}\leq2\tilde{L}_{\gamma,t}$.
    Let $\omega_{t}={(t+2)(t+3)/2}$ and $\alpha_{t}=4/\brbra{t+3}$,
    and suppose $\bundleSize\geq k_f + k_{\mcal G}$, then $\bar{L}_{\gamma}(N)\leq O(1)(L_f+\gamma L_{\mcal G})$.
\end{proposition}

\noindent\emph{Proof deferred.} The detailed proof is collected in Subsection~\ref{subsec:papex-detailed-proofs}.

We make two remarks about Proposition~\ref{prop:property_papex},
Theorem~\ref{thm:papex_f_star}, and Proposition~\ref{prop:upper_bound_Lip}.

First, these results extend the corresponding APEX guarantees for convex
optimization without functional constraints in~\cite{partOne} to the
present exact-penalty setting. They show that the new cutting-plane model based
on the middle term of~\eqref{eq:key_observation} preserves the key properties of
the previous construction: the residual recursion, the ideal-level $O(1/t^2)$
descent bound, and the empirical-curvature bound.

Second, the definition of $\tilde L_{\gamma,t}$ in
\eqref{eq:defn_L_tilde_phi} uses two independent cut indices. For the same test
point $\bar{x}^{t,r}$, the objective term may be bounded by a cut at
$\underline{x}^{t,l_f+1}$, while the constraint term may be bounded by a
different cut at $\underline{x}^{t,l_{\mcal G}+1}$. This is because
the model keeps the two cutting-plane envelopes
$\max_{j\in[i]}\bcbra{\ell_f(x;\underline{x}^{t,j})}$ and
$\bsbra{\max_{j\in[i]}\bcbra{\ell_{\mcal G}(x;\underline{x}^{t,j})}}_+$ separate.
Consequently, the separated construction has additive worst-case dependence on
the pieces of $f$ and $\mcal G$. It is enough to require
$\bundleSize\geq \max\bcbra{k_f , k_{\mcal G}}$.
An alternative is to use a single combined model,
\[
    \max_{j\in[i]}
    \bcbra{\ell_f(x;\underline{x}^{t,j})
        +\gamma\bsbra{\ell_{\mcal G}(x;\underline{x}^{t,j})}_{+}}.
\]
This model can yield an analogous guarantee if the bundle is large enough to
capture the pieces of the composite penalty. It uses one indexed cut family and
can make the projection subproblem slightly simpler. However, it forces the
objective and constraint cuts to share the same index, so the model must track
the pieces of $\phi(\cdot;\gamma)=f(\cdot)+\gamma[\mcal G(\cdot)]_+$. To capture the same
$\pws$ structure, the required worst-case number of cuts grows to
$k_f(1+k_{\mcal G})\le k_f\brbra{1+\sum_{i\in \activeconstrSet}k_{g_i}}.$
Hence, the combined model trades a slightly simpler projection subproblem for a
larger worst-case cut requirement.

We begin with the idealized restart scheme for $\papex$ in
Algorithm~\ref{alg:ideal_papex}; this scheme highlights the restart mechanism
that the practical algorithm later implements with estimated parameters.
\begin{algorithm}
    \caption{Idealized $\papex$}
    \label{alg:ideal_papex}
    \begin{algorithmic}[1]
        \Require  penalty parameter $\gamma> \norm{\lambda^*}_1$, initial candidate solution $\bar{y}^1$, optimal value $f^*$, number of cuts $\bundleSize$, shrinking factor $\theta\in(0,1)$.
        \For{$s=1,\ldots,S$}
        \State Run $\hat{x}^{t_s + 1}_{s} \leftarrow \papex(\gamma,\bar{y}^{s},f^*,\bundleSize)$ until
        $
            \phi(\hat{x}^{t_s + 1}_{s};\gamma)-f^*
            \le \theta \brbra{\phi(\bar{y}^{s};\gamma)-f^*}
        $
        \State Set $\bar{y}^{s+1}=\hat{x}^{t_s + 1}_{s}$
        \EndFor
    \end{algorithmic}
\end{algorithm}
\begin{thm}\label{thm:ideal_papex}
    Suppose that the assumptions of Theorem~\ref{thm:papex_f_star} hold,
    $\phi(x;\gamma)=f(x)+\gamma\bsbra{\mcal G(x)}_+$ is
    $\mu(\gamma;\phi)$-QG and that $\gamma>\norm{\lambda^*}_1$. Let $\mu(\gamma;\phi)$ be the QG modulus of $\phi(x;\gamma)$.
    For the $s$-th stage of Algorithm~\ref{alg:ideal_papex}, let
    $\bar{L}_{\gamma,s}(t)$ denote the quantity $\bar{L}_{\gamma}(t)$ in
    Theorem~\ref{thm:papex_f_star}, let $\tilde{L}_{s,t}$ denote the
    quantity $\tilde{L}_{\gamma,t}$ in Proposition~\ref{prop:upper_bound_Lip}
    evaluated at the $t$-th outer iteration of the $s$-th stage with
    $\tilde{l}=f^*$, and let $\mcal E_s(t)$ denote the corresponding event
    $\mcal E(t)$.
    Then, for any target accuracy
    $\vep>0$, if
    Algorithm~\ref{alg:ideal_papex} is stopped at the first stage $S$ such that
    \[
        \phi\brbra{\bar y^{S+1};\gamma}-f^*
        =
        f(\bar{y}^{S+1})-f^*+\gamma\bsbra{\mcal G(\bar{y}^{S+1})}_+
        \leq \vep,
    \]
    then the total number of first-order oracle evaluations is at most
    \begin{equation}\label{eq:ideal_complexity}
        O(1)\bundleSize \max\bcbra{\sqrt{\frac{L_{\text{avg}}}{\mu(\gamma;\phi)}},1}\log_{\frac{1}{\theta}}\frac{\phi(\bar{y}^{1};\gamma)-f^*}{\vep},
    \end{equation}
    where $L_{\text{avg}}=\max_{s=1,\ldots,S}\bcbra{\bar{L}_{\gamma,s}(t_s-1)}$ and we adopt the convention $\bar{L}_{\gamma,s}(0)=0$.
    For each $s=1,\ldots,S$ with $t_s>1$, $\bar{L}_{\gamma,s}\brbra{t_s - 1}$ admits the following upper bound:
    \begin{equation}\label{eq:upper_L_phi_s_t_s}
        \bar{L}_{\gamma,s}\brbra{t_s-1}\leq \frac{24}{\theta}\cdot \frac{\sum_{t=1}^{t_s - 1}\onebf_{\mcal E_{s}(t)}(t+3)}{\sum_{t=1}^{t_s - 1}\brbra{\onebf_{\mcal E_{s}(t)}(t+3)/\tilde{L}_{s,t}}}\ .
    \end{equation}
    Here, we adopt the standard convention that $0/0=0$. Moreover, if $\bundleSize\geq k_f + k_{\mcal G}$, then $L_{\text{avg}}\leq O\brbra{1}\brbra{L_f + \gamma L_{\mcal G}}$.
\end{thm}

\noindent\emph{Proof deferred.} The detailed proof is collected in Subsection~\ref{subsec:papex-detailed-proofs}.

Three remarks on this complexity result are in order.

First, if $\bundleSize\geq k_f + k_{\mcal G}$, the complexity~\eqref{eq:ideal_complexity} in Theorem~\ref{thm:ideal_papex} can be written as
$$O(1)\bundleSize \max\bcbra{\sqrt{\frac{L_f + \gamma L_{\mcal G}}{\mu(\gamma;\phi)}},1}\log_{\frac{1}{\theta}}\frac{\phi(\bar{y}^{1};\gamma)-f^*}{\vep},$$
which is derived under the idealized setting in which an upper bound on the norm of an optimal dual variable and the value of $f^*$ are known a priori. In practice, an upper bound on the dual norm can be obtained from a strictly feasible point $\bar{x}$ via $\bar{\gamma}=-(f(\bar{x})-f^*)/\mcal G(\bar{x})$. This technique is commonly used in previous work, e.g.,~\cite{nesterov2018lectures,lin2025adaptiveparameterfreeprojectionfreerestarting,NEURIPS2024_8d8e060d}. To eliminate the need to know $f^*$, one may further apply $\rapex$ from our previous work~\cite{partOne} for $\phi(x;\bar{\gamma})$, which yields condition-number dependence of order $\sqrt{(L_f + \gamma L_{\mcal G})/\mu(\bar{\gamma};\phi)}$. However, this approach has two potential drawbacks: the upper bound obtained from a strictly feasible point may be overly conservative when $\mcal G(\bar{x})\to 0^{-}$, which can make the dependence substantially worse than $\sqrt{(L_f + \norm{\lambda^*}_1 L_{\mcal G})/\mu(\norm{\lambda^*}_1;\phi)}$. Furthermore, obtaining a strictly feasible point itself may be nontrivial in practice.

Second, \eqref{eq:upper_L_phi_s_t_s} controls
$\bar{L}_{\gamma,s}\brbra{t_s-1}$ by a ratio that, when the event set is nonempty,
is $(24/\theta)$ times a weighted harmonic mean of the local empirical
Lipschitz constants $\tilde{L}_{s,t}$ over the indices for which $\mcal E_s(t)$
occurs, with weights $t+3$. Since these weights increase with $t$, more recent
values of $\tilde{L}_{s,t}$ within the stage have greater influence on the
bound. This harmonic-mean form dampens the effect of a small number of large
local constants. Moreover, the corresponding counting result,
Proposition~3.3 of~\cite{partOne}, implies that, for long stages, the fraction
of indices satisfying $\mcal E_s(t)$ approaches $\sqrt{\theta}$. In particular,
when $\theta=1/2$, about $70.7\%$ of the indices satisfy $\mcal E_s(t)$, so the
harmonic mean is taken over a majority of the stage.

Third, the complexity bound~\eqref{eq:ideal_complexity} requires the precondition $\bundleSize\geq k_f + k_{\mcal G}$. However, the practical dependence of our algorithm on the number of cuts is often much milder than this worst-case estimate. As emphasized in our earlier work~\cite{partOne}, the algorithm depends on the number of pieces encountered along its trajectory rather than on the global number of pieces. Moreover, although the worst-case number of cuts may increase with the number of constraints, each iteration only uses the gradient $\mcal G^\prime(\underline{x}^{t,i})$, so redundant constraints mainly make the worst-case bound conservative. The corresponding empirical evidence is presented in Section~\ref{sec:numerical_study}.

\subsection{Detailed proofs\label{subsec:papex-detailed-proofs}}

The statements above isolate the descent recursion and the idealized restart guarantee for \papex. We now collect the full proofs in the same order as the corresponding results.

\begin{proof}[\underline{Proof of Proposition~\ref{prop:property_papex}}]
    We prove the two claims separately.

    Part 1 (Monotonicity). For each $t\geq 1$ and $1\leq i<\bundleSize$,
    $x^{t,i}\in\argmin_{x\in X(t,i)}\norm{x-\bar{y}}^{2}$, which implies
    \begin{equation}
        \norm{x^{t,i}-\bar{y}}^{2}+\norm{x^{t,i}-x}^{2}\leq\norm{x-\bar{y}}^{2}\ ,\ \forall x\in X(t,i)\ .
    \end{equation}
    Because $x^{t,i+1}\in X(t,i+1)\subseteq X(t,i)$,
    substituting $x=x^{t,i+1}$ gives
    \begin{equation}
        \norm{x^{t,i}-\bar{y}}^{2}+\norm{x^{t,i}-x^{t,i+1}}^{2}\leq\norm{x^{t,i+1}-\bar{y}}^{2}\ ,
    \end{equation}
    which implies $\norm{x^{t,i}-\bar{y}}\leq\norm{x^{t,i+1}-\bar{y}}$.

    For $i=0$, by expanding the square,
    \begin{equation}
        \begin{split}\norm{x^{t,1}-\bar{y}}^{2} & =\norm{x^{t,1}-x^{t,0}}^{2}+2\inner{x^{t,1}-x^{t,0}}{x^{t,0}-\bar{y}}+\norm{x^{t,0}-\bar{y}}^{2}\end{split}
        \ .
    \end{equation}
    The linear constraint in $\bar{X}(t,0)$ ensures the cross term is nonnegative,
    so $\norm{x^{t,0}-\bar{y}}^{2}\leq\norm{x^{t,1}-\bar{y}}^{2}$. Hence
    $\norm{x^{t,i}-\bar{y}}\leq\norm{x^{t,i+1}-\bar{y}}$, which completes
    the proof of monotonicity.

    Part 2 (Convergence bound).  Let $0\leq l_t<r_t\leq\bundleSize$ be an index pair attaining the minimum in the definition of  $L_{\gamma,t}$.  If $l_t\geq1$, the
    projection optimality of $x^{t,l_t}$ over $X(t,l_t)$, together with
    $x^{t,r_t}\in X(t,r_t)\subseteq X(t,l_t)$, gives
    \[
        \norm{x^{t,l_t}-\bar y}^2+\norm{x^{t,l_t}-x^{t,r_t}}^2
        \le
        \norm{x^{t,r_t}-\bar y}^2.
    \]
    If $l_t=0$, the same inequality follows by expanding the square and using
    $x^{t,r_t}\in \bar{X}(t,0)$.  Therefore, in all cases, the monotonicity of Part 1 implies
    \[
        \norm{x^{t,0}-\bar y}^2+\norm{x^{t,l_t}-x^{t,r_t}}^2
        \le
        \norm{x^{t,\bundleSize}-\bar y}^2.
    \]
    Summing this inequality over $t=1,\dots,N$ and using
    $x^{1,0}=\bar y$ and $x^{t+1,0}=x^{t,\bundleSize}$ gives
    \[
        \sum_{t=1}^N \norm{x^{t,l_t}-x^{t,r_t}}^2
        \le
        \norm{x^{N,\bundleSize}-\bar y}^2.
    \]

    Next, by the definition of $L_{\gamma,t}$,
    \[
        \phi(\hat x^{t+1};\gamma)-\tilde l
        \le
        \Bigl(1-\frac{3\alpha_t}{4}\Bigr)\bigl(\phi(\hat x^t;\gamma)-\tilde l\bigr)
        +\frac{L_{\gamma,t}\alpha_t^2}{2}\norm{x^{t,r_t}-x^{t,l_t}}^2.
    \]
    Since $\phi(\hat x^t;\gamma)-\tilde l\geq0$, the definition of
    $\mcal E(t)$ yields the uniform recursion
    \[
        \phi(\hat x^{t+1};\gamma)-\tilde l
        \le
        \Bigl(1-\frac{\alpha_t}{2}\Bigr)\bigl(\phi(\hat x^t;\gamma)-\tilde l\bigr)
        +\onebf_{\mcal E(t)}
        \frac{L_{\gamma,t}\alpha_t^2}{2}\norm{x^{t,r_t}-x^{t,l_t}}^2.
    \]
    Multiplying by $\omega_t$ gives
    \[
        \omega_t\bigl(\phi(\hat x^{t+1};\gamma)-\tilde l\bigr)
        \le
        \omega_t\Bigl(1-\frac{\alpha_t}{2}\Bigr)\bigl(\phi(\hat x^t;\gamma)-\tilde l\bigr)
        +\onebf_{\mcal E(t)}
        \frac{\omega_tL_{\gamma,t}\alpha_t^2}{2}\norm{x^{t,r_t}-x^{t,l_t}}^2.
    \]
    Summing from $t=1$ to $N$ and using
    $\omega_t(1-\alpha_t/2)=\omega_{t-1}$ to telescope the terms with indices
    $2,\ldots,N$ gives
    \[
        \omega_N\bigl(\phi(\hat x^{N+1};\gamma)-\tilde l\bigr)
        -\omega_1\Bigl(1-\frac{\alpha_1}{2}\Bigr)\bigl(\phi(\hat x^1;\gamma)-\tilde l\bigr)
        \le
        \sum_{t=1}^N
        \onebf_{\mcal E(t)}
        \frac{\omega_tL_{\gamma,t}\alpha_t^2}{2}\norm{x^{t,r_t}-x^{t,l_t}}^2.
    \]
    By the definition of $\bar{L}_{\gamma}\brbra N$ and by $x^{N+1}=x^{N,\bundleSize}$ in Algorithm~\ref{alg:papex}, we have,
    \[
        \omega_N\bigl(\phi(\hat x^{N+1};\gamma)-\tilde l\bigr)
        -\omega_1\Bigl(1-\frac{\alpha_1}{2}\Bigr)\bigl(\phi(\hat x^1;\gamma)-\tilde l\bigr)
        \le
        \bar{L}_{\gamma}\brbra N\sum_{t=1}^N \norm{x^{t,l_t}-x^{t,r_t}}^2
        \le
        \bar{L}_{\gamma}\brbra N\norm{x^{N+1}-\bar y}^2.
    \]
    Replacing $N$ by $t$ completes the proof.
    \ifdefined\isarxiv
\end{proof}
\else
\qedsymbol\end{proof}
\fi

\begin{proof}[\underline{Proof of Theorem~\ref{thm:papex_f_star}}]
    Since $\gamma > \norm{\lambda^*}_1$, we have $\min_{x\in X}\phi\brbra{x;\gamma}=f^*$. Therefore, every subproblem in Algorithm~\ref{alg:One-step_plus} is feasible and feasible region contains $X^*$. Because $x^{t+1}$ is the minimum distance point to $\bar{y}$, we have $\norm{x^{t+1} - \bar{y}}\leq \dist\brbra{\bar{y}, X^*}$. Moreover, the choices of $\omega_t$ and $\alpha_t$ satisfy $\omega_t\brbra{1-\frac{\alpha_t}{2}} = \omega_{t-1}$ and $\omega_1(1-\alpha_1/2)=3$. The result then follows from the second part of Proposition~\ref{prop:property_papex} after dividing by $\omega_t=(t+2)(t+3)/2$.
\end{proof}

\begin{proof}[\underline{Proof of Proposition~\ref{prop:upper_bound_Lip}}]
    It follows from $\phi\brbra{\hat{x}^{t+1};\gamma}\leq\phi\brbra{\bar{x}^{t,\tilde r_{t}};\gamma}$
    and the definition of $\bar{x}^{t,\tilde r_{t}}=\brbra{1-\alpha_{t}}\hat{x}^{t}+\alpha_{t}x^{t,\tilde r_{t}}$
    that
    \begin{equation*}
        \begin{split}                    & \phi\brbra{\hat{x}^{t+1};\gamma}\leq\phi\brbra{\bar{x}^{t,\tilde r_{t}};\gamma}                                                                                                                                                                                                                                              \\
                                 & \aleq\ell_{f}\brbra{\bar{x}^{t,\tilde r_{t}};\underline{x}^{t,l_{f,t}+1}}+\gamma\bsbra{\ell_{\mcal G}\brbra{\bar{x}^{t,\tilde r_{t}};\underline{x}^{t,l_{\mcal G,t}+1}}}_{+}+\frac{\alpha_{t}}{4}\brbra{\phi\brbra{\hat{x}^{t};\gamma}-\tilde{l}}                                                                            \\
                                 & \ \ +\frac{\tilde{L}_{\gamma,t}}{2}\brbra{\norm{\bar{x}^{t,\tilde r_{t}}-\underline{x}^{t,l_{f,t}+1}}^{2}+\norm{\bar{x}^{t,\tilde r_{t}}-\underline{x}^{t,l_{\mcal G,t}+1}}^{2}}                                                                                                                                             \\
                                 & \bleq\brbra{1-\alpha_{t}}\phi\brbra{\hat{x}^{t};\gamma}+\alpha_{t}\tilde{l}+\frac{\alpha_{t}}{4}\brbra{\phi\brbra{\hat{x}^{t};\gamma}-\tilde{l}}+\frac{\tilde{L}_{\gamma,t}}{2}\brbra{\norm{\bar{x}^{t,\tilde r_{t}}-\underline{x}^{t,l_{f,t}+1}}^{2}+\norm{\bar{x}^{t,\tilde r_{t}}-\underline{x}^{t,l_{\mcal G,t}+1}}^{2}}
        \end{split},
    \end{equation*}
    where $(a)$ holds by the definition of $\tilde{L}_{\gamma,t}$ in~\eqref{eq:defn_L_tilde_phi}, $(b)$ follows from the definition of $\bar{x}^{t,\tilde r_{t}}$, the convexity of $f$, $\mcal G$ and $x^{t,\tilde r_{t}}\in X\brbra{t,\tilde r_{t}}$.
    Subtracting $\tilde{l}$ on both sides and using the definition of $\bar{x}^{t,\tilde r_{t}}$, $\underline{x}^{t,l_{f,t}+1}, \underline{x}^{t,l_{\mcal G,t}+1}$ gives
    \begin{equation*}
        \begin{aligned}
            \phi\brbra{\hat{x}^{t+1};\gamma}-\tilde{l} & \leq\brbra{1-\frac{3}{4}\alpha_{t}}\brbra{\phi\brbra{\hat{x}^{t};\gamma}-\tilde{l}}+\frac{\tilde{L}_{\gamma,t}\alpha_{t}^{2}}{2}\brbra{\norm{x^{t,\tilde r_{t}}-x^{t,l_{f,t}}}^{2}+\norm{x^{t,\tilde r_{t}}-x^{t,l_{\mcal G,t}}}^{2}} \\&\leq\brbra{1-\frac{3}{4}\alpha_{t}}\brbra{\phi\brbra{\hat{x}^{t};\gamma}-\tilde{l}}+\tilde{L}_{\gamma,t}\alpha_{t}^{2}\brbra{\max\bcbra{\norm{x^{t,\tilde r_{t}}-x^{t,l_{f,t}}}^{2},\norm{x^{t,\tilde r_{t}}-x^{t,l_{\mcal G,t}}}^{2}}}\ ,
        \end{aligned}
    \end{equation*}
    which implies $L_{\gamma,t}\leq2\tilde{L}_{\gamma,t}$. Next, we consider
    the upper bound of $\bar{L}_{\gamma}\brbra N$. By the definition of
    $\bar{L}_{\gamma}\brbra N$ and $\onebf_{\mcal E\brbra t}$, the upper
    bound of $\bar{L}_{\gamma}\brbra N$ is given as follows:
    \begin{equation*}
        \begin{split}\bar{L}_{\gamma}\brbra N & =\frac{\sum_{t=1}^{N}\omega_{t}\alpha_{t}^{2}L_{\gamma,t}\onebf_{\mcal E\brbra t}\norm{x^{t,r_{t}}-x^{t,l_{t}}}^{2}/2}{\sum_{t=1}^{N}\norm{x^{t,r_{t}}-x^{t,l_{t}}}^{2}}\leq\frac{8\sum_{t=1}^{N}\tilde{L}_{\gamma,t}\onebf_{\mcal E\brbra t}\norm{x^{t,r_{t}}-x^{t,l_{t}}}^{2}}{\sum_{t=1}^{N}\norm{x^{t,r_{t}}-x^{t,l_{t}}}^{2}}\ ,
        \end{split}
    \end{equation*}
    where the inequality holds by $\omega_t \alpha_t^2 \leq 8$ and $L_{\gamma,t} \leq 2\tilde{L}_{\gamma,t}$.

    If $\bundleSize\geq k_f + k_{\mcal G}$, then by
    the pigeonhole principle we can choose indices $l_{f,t}$, $l_{\mcal G,t}$, and
    $\tilde r_t$ such that $\bar{x}^{t,\tilde r_t}$ and
    $\underline{x}^{t,l_{f,t}+1}$ lie in the same smooth piece of $f(x)$, while
    $\bar{x}^{t,\tilde r_t}$ and $\underline{x}^{t,l_{\mcal G,t}+1}$ lie in the
    same smooth piece of $\bsbra{\mcal G(x)}_+$. Therefore, the numerator in the definition of
    $\tilde{L}_{\gamma,t}$ can be upper bounded by
    \begin{equation}\label{eq:cut_large_enough_imply_L_upper}
        \begin{split}                    & \phi\brbra{\bar{x}^{t,\tilde r_{t}};\gamma}-\brbra{\ell_{f}\brbra{\bar{x}^{t,\tilde r_{t}};\underline{x}^{t,l_{f,t}+1}}+\gamma\bsbra{\ell_{\mcal G}\brbra{\bar{x}^{t,\tilde r_{t}};\underline{x}^{t,l_{\mcal G,t}+1}}}_{+}}-\frac{\alpha_{t}}{4}\brbra{\phi\brbra{\hat{x}^{t};\gamma}-\tilde{l}} \\
                                 & \leq \frac{L_{f}}{2}\norm{\bar{x}^{t,\tilde r_{t}}-\underline{x}^{t,l_{f,t}+1}}^{2}+ \frac{\gamma L_{\mcal G}}{2}\norm{\bar{x}^{t,\tilde r_{t}}-\underline{x}^{t,l_{\mcal G,t}+1}}^{2}                                                                                                           \\
                                 & \leq \frac{L_{f}+\gamma L_{{\cal G}}}{2}\brbra{\norm{\bar{x}^{t,\tilde r_{t}}-\underline{x}^{t,l_{f,t}+1}}^{2}+\norm{\bar{x}^{t,\tilde r_{t}}-\underline{x}^{t,l_{\mcal G,t}+1}}^{2}}
        \end{split}\ ,
    \end{equation}
    which implies
    $\tilde{L}_{\gamma,t}\leq L_f + \gamma L_{\mcal G}$. This completes the proof.
    \ifdefined\isarxiv
\end{proof}
\else
\qedsymbol\end{proof}
\fi

\begin{proof}[\underline{Proof of Theorem~\ref{thm:ideal_papex}}]
    If $\phi\brbra{\bar y^s;\gamma}=f^*$ at the beginning of a stage, the
    claims for that stage are immediate. Otherwise, applying
    Theorem~\ref{thm:papex_f_star} to the $s$-th call to $\papex$ with
    $\hat{x}_s^1=\bar y^s$ and $\tilde l=f^*$ gives
    \[
        \phi\brbra{\hat{x}_s^{t+1};\gamma}-f^*
        \leq
        \frac{6}{(t+2)(t+3)}\brbra{\phi\brbra{\bar y^s;\gamma}-f^*}
        +\frac{2\bar{L}_{\gamma,s}(t)}{(t+2)(t+3)}
        \dist^2\brbra{\bar y^s,X^*}.
    \]
    Since $\phi(\cdot;\gamma)$ is $\mu(\gamma;\phi)$-QG and
    $\gamma>\norm{\lambda^*}_1$ gives $\min_{x\in X}\phi(x;\gamma)=f^*$,
    \[
        \dist^2\brbra{\bar y^s,X^*}
        \leq
        \frac{2}{\mu(\gamma;\phi)}\brbra{\phi\brbra{\bar y^s;\gamma}-f^*},
    \]
    and therefore
    \[
        \phi\brbra{\hat{x}_s^{t+1};\gamma}-f^*
        \leq
        \frac{6+4\bar{L}_{\gamma,s}(t)/\mu(\gamma;\phi)}{(t+2)(t+3)}
        \brbra{\phi\brbra{\bar y^s;\gamma}-f^*}.
    \]
    Define $t_s+1$ is the first inner iteration satisfying the stage-$s$
    stopping rule. The stage has not stopped at iteration $t_s$ whenever
    $t_s>1$. Hence,
    \[
        \theta
        <
        \frac{\phi\brbra{\hat{x}_s^{t_s};\gamma}-f^*}
        {\phi\brbra{\bar y^s;\gamma}-f^*}
        \leq
        \frac{6+4\bar{L}_{\gamma,s}(t_s-1)/\mu(\gamma;\phi)}
        {(t_s+1)(t_s+2)} .
    \]
    The geometric decrease follows directly from the stopping rule:
    \[
        \phi\brbra{\bar y^{s+1};\gamma}-f^*
        =
        \phi\brbra{\hat{x}_s^{t_s+1};\gamma}-f^*
        \leq
        \theta\brbra{\phi\brbra{\bar y^s;\gamma}-f^*}.
    \]
    Iterating this inequality over $s=1,\ldots,S$ gives the claimed bound on
    $\phi(\bar y^{S+1};\gamma)-f^*$ and, at the first stage satisfying the
    target accuracy, $S\leq \log_{1/\theta}
        \brbra{\brbra{\phi(\bar y^1;\gamma)-f^*}/\vep}$. The stopping-index
    inequality implies $t_s\leq O(1)\sqrt{L_{\text{avg}}/\mu(\gamma;\phi)}$.
    Since one call to $\onestepplus$ uses $\bundleSize$ first-order oracle evaluations, the
    stated first-order oracle evaluation complexity follows.

    It remains to prove~\eqref{eq:upper_L_phi_s_t_s}. If no event
    $\mcal E_s(t)$ occurs for $1\leq t\leq t_s-1$, then
    $\bar{L}_{\gamma,s}(t_s-1)=0$ by definition and the convention $0/0=0$
    gives the result. Otherwise, let $L_{\gamma,s,t}$ be the stage-$s$
    counterpart of $L_{\gamma,t}$ in~\eqref{eq:L_phi_N}. The definition of
    $\bar L_{\gamma,s}$ and the identity
    $\omega_t\alpha_t^2/2=4(t+2)/(t+3)\leq4$ give
    \[
        \bar{L}_{\gamma,s}\brbra{t_s-1}
        \leq
        \frac{4\sum_{t=1}^{t_s-1}\onebf_{\mcal E_s(t)}
        L_{\gamma,s,t}\norm{x_s^{t,r_t}-x_s^{t,l_t}}^2}
        {\sum_{t=1}^{t_s-1}\norm{x_s^{t,r_t}-x_s^{t,l_t}}^2}.
    \]
    For every $t$ such that $\mcal E_s(t)$ occurs, the definition of
    $L_{\gamma,s,t}$ gives
    \[
        \frac{\alpha_t^2}{2}L_{\gamma,s,t}\norm{x_s^{t,r_t}-x_s^{t,l_t}}^2
        =
        \phi\brbra{\hat{x}_s^{t+1};\gamma}-f^*
        -\brbra{1-\frac{3\alpha_t}{4}}
        \brbra{\phi\brbra{\hat{x}_s^{t};\gamma}-f^*}.
    \]
    Using the event definition and the nonincreasing update of
    $\phi(\hat{x}_s^t;\gamma)$, we obtain
    \[
        \frac{\alpha_t}{4}\brbra{\phi\brbra{\hat{x}_s^{t};\gamma}-f^*}
        \leq
        \frac{\alpha_t^2}{2}L_{\gamma,s,t}\norm{x_s^{t,r_t}-x_s^{t,l_t}}^2
        \leq
        \frac{3\alpha_t}{4}\brbra{\phi\brbra{\hat{x}_s^{t};\gamma}-f^*}.
    \]
    For $1\leq t\leq t_s-1$, the first-stopping definition and monotonicity imply
    \[
        \theta\brbra{\phi\brbra{\bar y^s;\gamma}-f^*}
        <
        \phi\brbra{\hat{x}_s^{t};\gamma}-f^*
        \leq
        \phi\brbra{\bar y^s;\gamma}-f^*.
    \]
    Since $\alpha_t=4/(t+3)$, for event iterations with
    $1\leq t\leq t_s-1$,
    \[
        \frac{\theta(t+3)}{8}\brbra{\phi\brbra{\bar y^s;\gamma}-f^*}
        \leq
        L_{\gamma,s,t}\norm{x_s^{t,r_t}-x_s^{t,l_t}}^2
        \leq
        \frac{3(t+3)}{8}\brbra{\phi\brbra{\bar y^s;\gamma}-f^*}.
    \]
    Substituting these two bounds into the preceding upper bound on
    $\bar{L}_{\gamma,s}(t_s-1)$ and using
    $L_{\gamma,s,t}\leq2\tilde L_{s,t}$ from
    Proposition~\ref{prop:upper_bound_Lip} proves
    \[
        \bar{L}_{\gamma,s}\brbra{t_s-1}
        \leq
        \frac{24}{\theta}\cdot
        \frac{\sum_{t=1}^{t_s-1}\onebf_{\mcal E_s(t)}(t+3)}
        {\sum_{t=1}^{t_s-1}\brbra{\onebf_{\mcal E_s(t)}(t+3)/\tilde L_{s,t}}},
    \]
    which is~\eqref{eq:upper_L_phi_s_t_s}. Under the additional condition
    $\bundleSize\geq k_f + k_{\mcal G}$, the stagewise analogue of
    the argument leading to~\eqref{eq:cut_large_enough_imply_L_upper} gives
    $\tilde{L}_{s,t}\leq O(1)(L_f+\gamma L_{\mcal G})$ for every relevant $s$ and $t$.
    Taking the maximum over stages yields
    $L_{\text{avg}}\leq O(1)(L_f+\gamma L_{\mcal G})$, completing the proof.
\end{proof}
\section{\label{subsec:CWcer}Penalty normalized Wolfe certificate}

The idealized guarantee for Algorithm~\ref{alg:ideal_papex} in
Section~\ref{sec:penalty_apex} assumes both a penalty parameter satisfying
$\gamma> \norm{\lambda^*}_1$ and prior knowledge of the optimal value
$f^*$. These quantities are generally unavailable before solving the problem.
To remove these dependencies, we extend the normalized Wolfe certificate
($\Wcer$) from~\cite{partOne} to the function-constrained exact-penalty setting;
the resulting certificate (Definition~\ref{def:CWcer}) is the key tool for the parameter-free algorithm
developed below.

This section proceeds as follows. We first define the certificate in
Definition~\ref{def:CWcer}. Proposition~\ref{prop:QG_for_lower_bound} then
shows that, under convexity and QG assumptions, a $\PWcer$ yields an upper
bound on the penalty optimality gap. Proposition~\ref{prop:APWG_W_cer_gen_complexity}
gives the iteration bound for generating such a certificate, and
Proposition~\ref{prop:CW_certificate_transfer} shows how to transfer a
certificate to a larger penalty parameter.

\begin{defn}
    [Penalty normalized Wolfe certificate, \PWcer]\label{def:CWcer}
    Let $\iota>0$, $\nu\geq0$, $\gamma > 0$. We say that a point $\bar{y}$
    is an $\brbra{\iota,\nu,\gamma}$-Penalty normalized Wolfe stationary point
    with respect to~\eqref{eq:initial_FP} if there exists a finite  evaluation
    point set $\mcal P=\bcbra{z^{0}=\bar{y},z^{1},\ldots,z^{q}}\subseteq\mcal B\brbra{\bar{y},\iota}\cap X$
    such that
    \begin{equation}
        \mcal V_{\mcal P,\gamma}^{+}\brbra{\iota;\bar{y}}=\frac{1}{\iota}\max_{x\in\mcal B\brbra{\bar{y},\iota}\cap X}\brbra{\psi_{\mcal P}^{+}\brbra{\bar{y};\gamma}-\psi_{\mcal P}^{+}\brbra{x;\gamma}}\leq\nu\ \text{and}\ \bsbra{\mcal G\brbra{\bar{y}}}_+\leq \frac{2\iota\nu}{ \gamma}\ ,\label{eq:V+_defn}
    \end{equation}
    where $\ensuremath{\psi}_{\mcal P}^{+}\brbra{x;\gamma}=\max_{z\in\mcal P}\bcbra{\ell_{f}\brbra{x;z}}+\gamma\bsbra{\max_{z\in\mcal P}\bcbra{\ell_{\mcal G}\brbra{x;z}}}_{+}$.  In this case, we call the evaluation point set $\mcal P$ an $\brbra{\iota,\nu,\gamma}$-$\PWcer$
    for the point $\bar{y}$ and problem~\eqref{eq:initial_FP}.
\end{defn}
The $\PWcer$ is a variant of $\Wcer$ tailored to function-constrained optimization.
It uses the separated exact-penalty model $\psi_{\mcal P}^{+}(x;\gamma)$ from
Definition~\ref{def:CWcer}, rather than the standard combined cutting-plane
model in the first term of~\eqref{eq:key_observation}, and it also explicitly
controls the constraint violation at the current point in the second term of~\eqref{eq:V+_defn}. These two features make
the certificate transferable to larger penalty parameters, as formalized in
Proposition~\ref{prop:CW_certificate_transfer}.


%

Proposition~\ref{prop:QG_for_lower_bound} shows that, under the QG condition,
a $\PWcer$ yields an explicit upper bound on the optimality gap of the penalty
function $\phi\brbra{\bar{y};\gamma}$.
\begin{proposition}
    \label{prop:QG_for_lower_bound}
    Suppose $f(x)$ and $\mcal G(x)$ are convex, and that the penalty function $\phi(x;\gamma):=f(x)+\gamma [\mcal G(x)]_{+}$ satisfies the $\mu$-QG condition over $X$. If there exists a  point set $\mcal P$ forming an $\brbra{\iota,\nu,\gamma}$-$\PWcer$ for the point $\bar{y}$ in problem~\eqref{eq:initial_FP}, then
    $\phi\brbra{\bar{y};\gamma}-f^*\leq\phi\brbra{\bar{y};\gamma}-\min_{x\in X}\phi\brbra{x;\gamma}\leq\max\bcbra{\iota\nu,{2\nu^{2}}/{\mu}}.$
\end{proposition}

\noindent\emph{Proof deferred.} The detailed proof is collected in Subsection~\ref{subsec:pwcer-detailed-proofs}.

The $\PWcer$ is generated by running a penalty variant of $\mcal{AWG}$,
namely ${\pawg}$ (Algorithm~\ref{alg:APWG}). Given a candidate gap estimate
$\Delta$ and a factor $\beta>0$, ${\pawg}$ tests the lower level
$\tilde l=\phi(\bar y;\gamma)-(1+\beta)\Delta$. If $\Delta$ is a valid upper
bound on $\phi(\bar y;\gamma)-f^*$, this level lies below the optimal value
$f^*$ by at least the additional margin $\beta\Delta$. If the initial constraint violation is too large, the procedure
returns $\false$ immediately. Otherwise, each call to $\onestepplus$ attempts
to certify the lower level through the penalty cutting-plane model. When the
subproblem becomes infeasible, i.e., $\flag=\true$, or when the projection
leaves the ball of radius $\iota_{\max}$ around $\bar y$, the accumulated cuts
form a $\PWcer$ and ${\pawg}$ returns $\true$. If instead a point with penalty
value below $\phi(\bar y;\gamma)-\Delta$ is found, the procedure returns $\false$. Compared with
$\mcal{AWG}$, the only change in the inner step is that $\onestep$ is replaced
by $\onestepplus$, whose cutting-plane model is the penalty model
in~\eqref{eq:key_observation}.


\begin{algorithm}[h]
    \caption{\protect\label{alg:APWG}Accelerated $\protect\PWcer$ Generation, $\pawg\protect\brbra{\gamma,\bar{y},\Delta,\iota_{\max},\bundleSize,\beta}$}
    \begin{algorithmic}[1]
        \State \textbf{Initialize:} $x^{1}=\bar{y}$, $\hat{x}^{1}=\bar{y}$, $\tilde{l}= f\brbra{\bar{y}}+\gamma\bsbra{\mcal G\brbra{\bar{y}}}_{+}-\brbra{1+\beta}\Delta$
        \If{$\bsbra{\mcal G(\bar{y})}_+ > 2(1+\beta)\Delta / \gamma$}
        \State \textbf{return} ${\false}$\label{enu:false_direct}
        \EndIf
        \For{$t=1,\ldots$}
        \State $\brbra{\hat{x}^{t+1},x^{t+1},\flag}=\onestepplus\brbra{\gamma,\bar{y},\hat{x}^{t},x^{t},\tilde{l},\bundleSize}$
        \If{$\flag$ is $\true$ or $\norm{x^{t+1}-\bar{y}}>\iota_{\max}$}
        \State \textbf{return} ${\true}$ \Comment{$\true$: certificate is generated}\label{enu:return_W_cer}
        \EndIf
        \If{$\phi(\hat{x}^{t+1};\gamma)<\phi(\bar{y};\gamma)-\Delta$}
        \State \textbf{return} ${\false}$ \Comment{$\false$: no certificate}\label{enu:WG_mu_wrong}
        \EndIf
        \EndFor
    \end{algorithmic}
\end{algorithm}

We establish the properties of $\pawg$ in Proposition~\ref{prop:APWG_W_cer_gen_complexity} below.

\begin{proposition}
    \label{prop:APWG_W_cer_gen_complexity}
    Let $\omega_{t}=(t+2)(t+3)/2$ and $\alpha_{t}=4/(t+3)$ in
    Algorithm~\ref{alg:APWG}. Suppose Algorithm~\ref{alg:APWG} terminates during
    its main loop at iteration $t_{\mcal W}$. Then the weighted empirical
    Lipschitz constant average $\bar{L}_{\gamma}(t_{\mcal W}-1)$ and the termination
    iteration $t_{\mcal W}$ satisfy
        {\small
            \begin{equation}\label{eq:weighted_harmonic_mean_APWG}
                \bar{L}_{\gamma}\brbra{t_{\mcal W}-1}\leq
                \frac{24\brbra{1+\beta}}{\beta}
                \frac{\sum_{t=1}^{t_{\mcal W}-1}\onebf_{\mcal E(t)}\brbra{t+3}}
                {\sum_{t=1}^{t_{\mcal W}-1}\brbra{\onebf_{\mcal E(t)}\brbra{t+3}/\tilde{L}_{\gamma,t}}}
                \,\text{and}\,\,
                t_{\mcal W}\leq
                \min\bcbra{t:t>\sqrt{\frac{2\iota_{\max}^{2}\bar{L}_{\gamma}\brbra t+6\brbra{1+\beta}\Delta}{\beta\Delta}}}.
            \end{equation}}
    Here $\tilde{L}_{\gamma,t}$ is the local empirical Lipschitz constant defined
    in~\eqref{eq:defn_L_tilde_phi}, $\mcal E(t)$ is defined after~\eqref{eq:L_phi_N},
    and empty sums are interpreted with the convention $0/0=0$. If, in addition,
    $\phi(\bar y;\gamma)-f^{*}\leq\Delta$ and $\gamma \geq 2\norm{\lambda^*}_1$, then
    Algorithm~\ref{alg:APWG} must return $\true$ (Line~\ref{enu:return_W_cer}). Consequently,
    there exists a point set $\mcal P_{t_{\mcal W}}$ that forms an
    $\brbra{\iota_{\max},\brbra{1+\beta}\Delta/\iota_{\max},\gamma}$-$\PWcer$ for
    $\bar y$.
\end{proposition}

\noindent\emph{Proof deferred.} The detailed proof is collected in Subsection~\ref{subsec:pwcer-detailed-proofs}.

Proposition~\ref{prop:APWG_W_cer_gen_complexity} bounds the iteration count
for a single call to $\pawg$. Its dependence on the local Lipschitz constants
has the same weighted-harmonic-mean structure as in
Theorem~\ref{thm:ideal_papex}, with the interpretation discussed after that
theorem.

Proposition~\ref{prop:CW_certificate_transfer} demonstrates that a certificate for one penalty function can be used for another penalty function with a larger penalty coefficient.
\begin{proposition}
    \label{prop:CW_certificate_transfer}
    Suppose that the point set $\mcal P$ is
    an $\brbra{\iota,\nu,\gamma}-\PWcer$ for point $\bar{y}$ in problem~\eqref{eq:initial_FP},
    then for any $\tilde{\gamma}>0$, we have $\mcal P$ is also an
    $\brbra{\iota,\tilde{\nu},\max\bcbra{\tilde{\gamma},\gamma}}-\PWcer$ for point $\bar{y}$ in problem~\eqref{eq:initial_FP},
    where $\tilde{\nu}=\max \bcbra{\tilde{\gamma}-\gamma,0}\cdot \frac{2 \nu}{\gamma}+\nu$.
\end{proposition}

\noindent\emph{Proof deferred.} The detailed proof is collected in Subsection~\ref{subsec:pwcer-detailed-proofs}.

We close this discussion with three remarks on
Proposition~\ref{prop:CW_certificate_transfer}.

First, the middle model in~\eqref{eq:key_observation} is
essential for transferring certificates to penalty parameters no smaller than
the current one.
The proof of
Proposition~\ref{prop:CW_certificate_transfer} uses the monotonicity
\[
    \psi_{\mcal P}^{+}(x;\gamma)\leq \psi_{\mcal P}^{+}(x;\tilde{\gamma})
    \quad \text{whenever}\quad \tilde{\gamma}\geq\gamma,
\]
which follows because $\gamma$ multiplies the single nonnegative term
$\bsbra{\max_{z\in\mcal P}\ell_{\mcal G}(x;z)}_{+}$. This monotonicity
does not hold for the first model in~\eqref{eq:key_observation},
where the positive-part operation is linearized before the cuts are
aggregated.

    {Second, a standard $\Wcer$ is tied to a fixed penalty
        parameter.} Applying the certificate framework of~\cite{partOne} directly
to a fixed penalty function $\phi(\cdot;\gamma)$ can certify approximate
optimality for that single penalty problem under an appropriate QG condition.
However, the adaptive
scheme in Section~\ref{sec:rpapex} updates both $\gamma$ and the QG
estimate $\mu$, and a certificate for $\phi(\cdot;\gamma_1)$ does not by
itself provide the transfer guarantee needed after replacing
$\gamma_1$ by a larger penalty parameter.
This limitation motivates the new certificate for the function-constrained setting.

    {Third, the $\PWcer$ supplies the transfer mechanism needed by
        the parameter-free algorithm.} Proposition~\ref{prop:CW_certificate_transfer}
shows that the same point set remains a $\PWcer$ after increasing the
penalty parameter, with the certificate constant updated to
$\tilde{\nu}=\max\bcbra{\tilde{\gamma}-\gamma, 0}\cdot 2\nu/\gamma+\nu$. Together with the
gap estimate in Proposition~\ref{prop:QG_for_lower_bound}, this transfer
property provides the certificate information used later to maintain
valid lower bounds in Lemma~\ref{lem:lower_bound_valid_alg} and to
support the complexity guarantees in
Theorems~\ref{thm:convergence_CW_certificate}--\ref{thm:convergence_opt_convio}.

\subsection{Detailed proofs\label{subsec:pwcer-detailed-proofs}}

We first establish a preliminary monotonicity lemma for the certificate
residual, and then use it to prove the certificate properties, generation
guarantee, and transfer mechanism stated above.

\begin{lem}\label{lem:V_plus_non_increasing}
    $\mcal V_{\mcal P,\gamma}^{+}\brbra{\iota;\bar{y}}$ (defined in~\eqref{eq:V+_defn}) is non-increasing with respect to $\iota$.
\end{lem}
\begin{proof}
    Let some $\iota_{1}>\iota_{2}>0$ be given and let $x_{\iota_1}$ and
    $x_{\iota_2}$ denote the respective optimal solutions of $\mcal V_{\mcal P,\gamma}^{+}\brbra{\iota_{1};\bar{y}}$
    and $\mcal V_{\mcal P,\gamma}^{+}\brbra{\iota_{2};\bar{y}}$. By the convexity
    of $X$, we have point $\tilde{x}:=\bar{y}+\frac{\iota_{2}}{\iota_{1}}\brbra{x_{\iota_1}-\bar{y}}\in X.$
    Since $\psi_{\mcal P}^{+}\brbra{x;\gamma}$ is convex and $\bar{y},\tilde{x}$
    and $x_{\iota_1}$ lie on the same line, we get from the monotonicity
    of the secant slopes for a convex function:
    \begin{equation}
        \frac{\psi_{\mcal P}^{+}\brbra{\bar{y};\gamma}-\psi_{\mcal P}^{+}\brbra{\tilde{x};\gamma}}{\norm{\tilde{x}-\bar{y}}}\geq\frac{\psi_{\mcal P}^{+}\brbra{\bar{y};\gamma}-\psi_{\mcal P}^{+}\brbra{x_{\iota_1};\gamma}}{\norm{x_{\iota_1}-\bar{y}}}\ .\label{eq:secant}
    \end{equation}
    Combining~\eqref{eq:secant}, $\norm{\bar{y}-\tilde{x}}={\iota_{2}}\norm{x_{\iota_1}-\bar{y}}/{\iota_{1}}$
    gives $\brbra{\psi_{\mcal P}^{+}\brbra{\bar{y};\gamma}-\psi_{\mcal P}^{+}\brbra{\tilde{x};\gamma}}/{\iota_{2}}\geq\brbra{\psi_{\mcal P}^{+}\brbra{\bar{y};\gamma}-\psi_{\mcal P}^{+}\brbra{x_{\iota_1};\gamma}}/{\iota_{1}}$.
    By the definition of $x_{\iota_2}$, we have
    \[
        \frac{\psi_{\mcal P}^{+}\brbra{\bar{y};\gamma}-\psi_{\mcal P}^{+}\brbra{x_{\iota_1};\gamma}}{\iota_{1}}\leq\frac{\psi_{\mcal P}^{+}\brbra{\bar{y};\gamma}-\psi_{\mcal P}^{+}\brbra{\tilde{x};\gamma}}{\iota_{2}}\leq\frac{\psi_{\mcal P}^{+}\brbra{\bar{y};\gamma}-\psi_{\mcal P}^{+}\brbra{x_{\iota_2};\gamma}}{\iota_{2}}.
    \]
    Hence, this proves that $\mcal V_{\mcal P,\gamma}^{+}\brbra{\iota;\bar{y}}$
    is non-increasing with respect to $\iota$.
    \ifdefined\isarxiv
\end{proof}
\else
\qedsymbol\end{proof}
\fi

\begin{lem}
    \label{lem:lower_bound_of_ball}
    Fix $\gamma$ and $\tilde l$. After the first $t$ calls to $\onestepplus$ in
    Algorithm~\ref{alg:APWG}, let $$\mcal P_t = \bcbra{\underline{x}^{i,j}:1\leq i\leq t,\ 1\leq j\leq \bundleSize}$$ be the  set of evaluation points
    generated in those calls, and let $x^{t+1}$ be the projection point returned
    at iteration $t$. For any $\iota\in(0,\norm{x^{t+1}-\bar y})$, define
    $\tilde{\mcal P}_t(\iota):=\mcal P_t\cap\mcal B(\bar y,\iota)$. Then
    $$\min_{x\in\mcal B(\bar y,\iota)\cap X}
        \psi_{\tilde{\mcal P}_t(\iota)}^{+}(x;\gamma)>\tilde l.$$
\end{lem}
\begin{proof}
    The proof follows the proof of Lemma~\red{4.2} in~\cite{partOne}, with the
    unconstrained cutting-plane model $\psi_{\tilde{\mcal P}(\iota),f}$ replaced
    by the penalty model $\psi_{\tilde{\mcal P}_t(\iota)}^{+}(\cdot;\gamma)$, and
    with $\onestep$ replaced by $\onestepplus$. The projection-separation argument
    is otherwise unchanged, so we omit the details.
\end{proof}

\begin{proof}[\underline{Proof of Proposition~\ref{prop:QG_for_lower_bound}}]
    Since $f$ and $\mcal G$ are convex, $\psi_{\mcal P}^{+}\brbra{x;\gamma}\leq\phi\brbra{x;\gamma}$.
    Furthermore, condition $\bar{y}\in\mcal P$ and the convexity  of $f$ and $\mcal G$ imply $\psi_{\mcal P}^{+}\brbra{\bar{y};\gamma}=\phi\brbra{\bar{y};\gamma}$.
    We consider two cases based on the distance between $x_{\gamma}^{*}$ and
    $\bar{y}$, where $x_{\gamma}^{*}=\argmin_{x\in X_{\gamma}^{*}}\bcbra{\norm{\bar{y}-x}}$ and $X_{\gamma}^{*}:=\argmin_{x\in X}\phi\brbra{x;\gamma}$.

    Case 1: $x_{\gamma}^{*}\in\mcal B\brbra{\bar{y},\iota}$. Then we have
    \begin{equation}
        \begin{split}\phi\brbra{\bar{y};\gamma}-f^{*} & \aleq\phi\brbra{\bar{y};\gamma}-\phi\brbra{x_{\gamma}^{*};\gamma}                                                                                        \\
                                              & \bleq\phi\brbra{\bar{y};\gamma}-\psi_{\mcal P}^{+}\brbra{x_{\gamma}^{*};\gamma}                                                                          \\
                                              & \cleq\max_{x\in\mcal B\brbra{\bar{y},\iota}}\bcbra{\psi_{\mcal P}^{+}\brbra{\bar{y};\gamma}-\psi_{\mcal P}^{+}\brbra{x_{\gamma}^{*};\gamma}}\leq\iota\nu
        \end{split}
        \ .\label{eq:ineq_0}
    \end{equation}
    In the display above, $(a)$ follows the definition of $x_{\gamma}^{*}$ such that
    $\phi\brbra{x_{\gamma}^{*};\gamma}\leq\phi\brbra{x^{*};\gamma}$,
    where $x^{*}$ is an optimal solution of~\eqref{eq:initial_FP}; $(b)$ follows from
    $\psi_{\mcal P}^{+}\brbra{x_{\gamma}^{*};\gamma}\leq \phi\brbra{x_{\gamma}^{*};\gamma}$ by convexity of function $f$ and $\mcal G$; $(c)$ holds by
    $\bar{y}\in\mcal P$ implies $\phi\brbra{\bar{y};\gamma}=\psi_{\mcal P}^{+}\brbra{\bar{y};\gamma}$.

    Case 2: $x_{\gamma}^{*}\in\mcal B^{c}\brbra{\bar{y},\iota}$. From
    the radius argument that $\norm{x_{\gamma}^{*}-\bar{y}}\geq\iota$
    and $\mcal V_{\mcal P,\gamma}^{+}\brbra{\iota;\bar{y}}$ is non-increasing
    with respect to $\iota$, we have
    \begin{equation}
        \begin{split}\phi\brbra{\bar{y};\gamma}-\phi\brbra{x_{\gamma}^{*};\gamma} & \leq\psi_{\mcal P}^{+}\brbra{\bar{y};\gamma}-\psi_{\mcal P}^{+}\brbra{x_{\gamma}^{*};\gamma}                                                           \\
                                                                          & \leq\max_{x\in\mcal B\brbra{\bar{y},\norm{\bar{y}-x_{\gamma}^{*}}}}\bcbra{\psi_{\mcal P}^{+}\brbra{\bar{y};\gamma}-\psi_{\mcal P}^{+}\brbra{x;\gamma}} \\
                                                                          & =\norm{\bar{y}-x_{\gamma}^{*}}\mcal V_{\mcal P,\gamma}^{+}\brbra{\norm{\bar{y}-x_{\gamma}^{*}};\bar{y}}                                                \\
                                                                          & \leq\norm{\bar{y}-x_{\gamma}^{*}}\mcal V_{\mcal P,\gamma}^{+}\brbra{\iota;\bar{y}}\leq\norm{\bar{y}-x_{\gamma}^{*}}\nu
        \end{split}
        \ ,\label{eq:ineq_1}
    \end{equation}
    where the third inequality holds by Lemma~\ref{lem:V_plus_non_increasing}.
    Combining the above inequality and $\phi\brbra{\bar{y};\gamma}-\phi\brbra{x_{\gamma}^{*};\gamma}\geq\frac{\mu}{2}\norm{\bar{y}-x_{\gamma}^{*}}^{2}$,
    we have
    \begin{equation}
        \norm{\bar{y}-x_{\gamma}^{*}}\leq\frac{2\nu}{\mu}\ .\label{eq:radius_upper}
    \end{equation}
    Putting~\eqref{eq:ineq_1} and~\eqref{eq:radius_upper} together,
    we have
    \begin{equation}
        \phi\brbra{\bar{y};\gamma}-\phi\brbra{x_{\gamma}^{*};\gamma}\leq\frac{2\nu^{2}}{\mu}\ .
    \end{equation}
    The inequality $\brbra a$ in~\eqref{eq:ineq_0} also holds. This completes the proof.
    \ifdefined\isarxiv
\end{proof}
\else
\qedsymbol\end{proof}
\fi

\begin{proof}[\underline{Proof of Proposition~\ref{prop:APWG_W_cer_gen_complexity}}]
    Set $\tilde l=\phi(\bar y;\gamma)-(1+\beta)\Delta$. Algorithm~\ref{alg:APWG}
    is the penalty version of the $\mcal{AWG}$ routine, with
    $\alpha_t=4/(t+3)$ and $\omega_t=(t+2)(t+3)/2$. Before termination,
    we have
    $\beta\Delta\leq \phi(\hat{x}^{t+1};\gamma)-\tilde l\leq(1+\beta)\Delta$
    for $1\leq t\leq t_{\mcal W}-1$. The first bound in
    \eqref{eq:weighted_harmonic_mean_APWG} follows from the weighted-average
    estimate~\eqref{eq:upper_L_phi_s_t_s} and Proposition~\ref{prop:upper_bound_Lip},
    which replaces $L_{\gamma,t}$ by $2\tilde L_{\gamma,t}$ and gives the
    constant $24(1+\beta)/\beta$.

    We next prove the second inequality in~\eqref{eq:weighted_harmonic_mean_APWG}.
    It suffices to show that any integer $t$ satisfying
    \[
        t>\sqrt{\frac{2\iota_{\max}^{2}\bar L_{\gamma}\brbra t
                +6(1+\beta)\Delta}{\beta\Delta}}
    \]
    cannot be a completed nonterminating iteration. Suppose, to the contrary, that
    Algorithm~\ref{alg:APWG} has not terminated by the end of such an iteration $t$.
    Then all subproblems up to iteration $t$ are feasible, the radius test at
    Line~\ref{enu:return_W_cer} gives $\norm{x^{t+1}-\bar y}\leq\iota_{\max}$, and
    the false-decrease test at Line~\ref{enu:WG_mu_wrong} gives
    \[
        \phi(\hat x^{t+1};\gamma)-\tilde l\geq\beta\Delta .
    \]
    Since $\beta>0$, such a $t$ is at least $2$. Applying
    Proposition~\ref{prop:property_papex} with
    $\omega_t=(t+2)(t+3)/2$, $\alpha_t=4/(t+3)$, and
    $\omega_1(1-\alpha_1/2)=3$ yields
    \[
        \frac{1}{2}(t+2)(t+3)
        \brbra{\phi(\hat x^{t+1};\gamma)-\tilde l}
        -3(1+\beta)\Delta
        \leq \bar L_{\gamma}\brbra t\norm{x^{t+1}-\bar y}^{2}.
    \]
    Hence
    \[
        \frac{1}{2}(t+2)(t+3)\beta\Delta-3(1+\beta)\Delta
        \leq \bar L_{\gamma}\brbra t\,\iota_{\max}^{2}.
    \]
    The displayed condition on $t$ implies the reverse strict inequality, since
    $(t+2)(t+3)>t^{2}$, a contradiction. Thus no iteration before
    $t_{\mcal W}$ can satisfy the displayed condition, proving the second
    inequality in~\eqref{eq:weighted_harmonic_mean_APWG}.

    Suppose next that Algorithm~\ref{alg:APWG} reaches Line~\ref{enu:return_W_cer}. If $\flag$ is $\false$, then we have $\norm{x^{t_{\mcal W}}-\bar y}> \iota_{\max}$.
    Let
    $\mcal P_{t_{\mcal W}}:=\bcbra{\underline{x}^{i,j}:1\leq i\leq t_{\mcal W},
            1\leq j\leq\bundleSize}\cap\mcal B\brbra{\bar y,\iota_{\max}}$.
    Then $\mcal P_{t_{\mcal W}}\subseteq\mcal B\brbra{\bar y,\iota_{\max}}\cap X$.
    By Lemma~\ref{lem:lower_bound_of_ball} with $\iota=\iota_{\max}$ in the
    radius-triggered case,
    $\psi_{\mcal P_{t_{\mcal W}}}^{+}(x;\gamma)>\tilde l$ for all
    $x\in\mcal B\brbra{\bar y,\iota_{\max}}\cap X$, while convexity gives
    $\psi_{\mcal P_{t_{\mcal W}}}^{+}(\bar y;\gamma)\leq \phi(\bar y;\gamma)$.
    Hence
    \[
        \mcal V_{\mcal P_{t_{\mcal W}},\gamma}^{+}\brbra{\iota_{\max};\bar y}
        < \frac{\phi\brbra{\bar y;\gamma}-\tilde l}{\iota_{\max}}
        =\frac{(1+\beta)\Delta}{\iota_{\max}} .
    \]
    Under the additional assumptions, $\gamma\geq 2\norm{\lambda^*}_1$ implies
    $\norm{\lambda^*}_1\leq \gamma/2$. The KKT lower bound gives
    $f(\bar y)-f^*\geq-\norm{\lambda^*}_1\bsbra{\mcal G(\bar y)}_+$, and therefore
    $\Delta\geq\phi(\bar y;\gamma)-f^*
        \geq(\gamma-\norm{\lambda^*}_1)\bsbra{\mcal G(\bar y)}_+
        \geq \gamma\bsbra{\mcal G(\bar y)}_+/2$. Thus
    $\bsbra{\mcal G(\bar y)}_+\leq2\Delta/\gamma\leq2(1+\beta)\Delta/\gamma$.
    The Wolfe-gap and constraint-violation bounds prove that
    $\mcal P_{t_{\mcal W}}$ is an
    $\brbra{\iota_{\max},(1+\beta)\Delta/\iota_{\max},\gamma}$-$\PWcer$. If $\flag$ is $\true$, then the subproblem is infeasible. Similar point set $\mcal P_{t_{\mcal W}}$ can be constructed and it is also a $\brbra{\iota_{\max},(1+\beta)\Delta/\iota_{\max},\gamma}$-$\PWcer$.

    It remains to rule out termination with $\false$ under the additional assumptions.
    The preceding constraint bound excludes Line~\ref{enu:false_direct}. Since
    $\gamma\geq 2\norm{\lambda^*}_1$, exact penalization gives
    $\phi(x;\gamma)\geq f^*$ for all $x\in X$; because
    $\phi(\bar y;\gamma)-\Delta\leq f^*$, the decrease test in
    Line~\ref{enu:WG_mu_wrong} cannot hold. Therefore the main-loop termination
    must occur at Line~\ref{enu:return_W_cer}.
    \ifdefined\isarxiv
\end{proof}
\else
\qedsymbol\end{proof}
\fi

\begin{proof}[\underline{Proof of Proposition~\ref{prop:CW_certificate_transfer}}]
    If $\tilde{\gamma}\leq \gamma$, the claim follows immediately with $\max\bcbra{\gamma, \tilde{\gamma}} = \gamma$ and $\tilde{\nu} = \nu$. It remains to consider the case $\tilde{\gamma}>\gamma$. By the definition of $\brbra{\iota,\nu,\gamma}-\PWcer$, we have $\mcal P\subseteq\mcal B\brbra{\bar{y},\iota}$.
    Hence, it remains to verify $\mcal V_{\mcal P,\tilde{\gamma}}^{+}\brbra{\iota;\bar{y}}\leq\tilde{\nu}$.
    It follows from the definition of $\mcal V_{\mcal P,\tilde{\gamma}}^{+}\brbra{\iota;\bar{y}}$
    that
    \begin{equation*}
        \begin{split}\iota\mcal V_{\mcal P,\tilde{\gamma}}^{+}\brbra{\iota;\bar{y}}
             & =\max_{x\in\mcal B\brbra{\bar{y},\iota}\cap X}\bcbra{\psi_{\mcal P}^{+}\brbra{\bar{y};\tilde{\gamma}}-\psi_{\mcal P}^{+}\brbra{\bar{y};\gamma}+\psi_{\mcal P}^{+}\brbra{\bar{y};\gamma}-\psi_{\mcal P}^{+}\brbra{x;\gamma}+\psi_{\mcal P}^{+}\brbra{x;\gamma}-\psi_{\mcal P}^{+}\brbra{x;\tilde{\gamma}}} \\
             & \aleq\psi_{\mcal P}^{+}\brbra{\bar{y};\tilde{\gamma}}-\psi_{\mcal P}^{+}\brbra{\bar{y};\gamma}+\iota\nu+\max_{x\in\mcal B\brbra{\bar{y},\iota}\cap X}\bcbra{\psi_{\mcal P}^{+}\brbra{x;\gamma}-\psi_{\mcal P}^{+}\brbra{x;\tilde{\gamma}}}                                                                \\
             & \beq\phi\brbra{\bar{y};\tilde{\gamma}}-\phi\brbra{\bar{y};\gamma}+\iota\nu+\max_{x\in\mcal B\brbra{\bar{y},\iota}\cap X}\bcbra{\psi_{\mcal P}^{+}\brbra{x;\gamma}-\psi_{\mcal P}^{+}\brbra{x;\tilde{\gamma}}}                                                                                             \\
             & =\brbra{\tilde{\gamma}-\gamma}\bsbra{\mcal G\brbra{\bar{y}}}_{+}+\iota\nu+\max_{x\in\mcal B\brbra{\bar{y},\iota}\cap X}\bcbra{\psi_{\mcal P}^{+}\brbra{x;\gamma}-\psi_{\mcal P}^{+}\brbra{x;\tilde{\gamma}}}                                                                                              \\
             & \cleq\brbra{\tilde{\gamma}-\gamma}\bsbra{\mcal G\brbra{\bar{y}}}_{+}+\iota\nu \dleq \brbra{\tilde{\gamma}-\gamma}\cdot \frac{2\iota \nu}{\gamma}+\iota\nu = \iota \tilde{\nu}
        \end{split}
        \ ,
    \end{equation*}
    In above, $(a)$ holds due to the definition of an $\brbra{\iota,\nu,\gamma}$-$\PWcer$, which ensures that
    \[
        \max_{x\in\mcal B\brbra{\bar{y},\iota}\cap X} \left\{ \psi_{\mcal P}^{+}\brbra{\bar{y};\gamma} - \psi_{\mcal P}^{+}\brbra{x;\gamma} \right\} \leq \iota\nu\ .
    \]
    $(b)$ holds because $\bar{y}=z^{0}\in\mcal P$, and for every $z\in\mcal P$ we have $\ell_f(\bar y;z)\le f(\bar y)$ and $\ell_{\mcal G}(\bar y;z)\le \mcal G(\bar y)$, with equality at $z=\bar y$; hence $\psi_{\mcal P}^{+}\brbra{\bar{y};\gamma} = \phi\brbra{\bar{y};\gamma}$ and $\psi_{\mcal P}^{+}\brbra{\bar{y};\tilde{\gamma}} = \phi\brbra{\bar{y};\tilde{\gamma}}$. Step $(c)$ follows from the definition of $\psi_{\mcal P}^{+}$: when $\tilde\gamma > \gamma$, we have $\psi_{\mcal P}^{+}(x;\gamma)\le \psi_{\mcal P}^{+}(x;\tilde\gamma)$ for every $x$, and therefore $\max_{x\in\mcal B\brbra{\bar{y},\iota}\cap X}\bcbra{\psi_{\mcal P}^{+}(x;\gamma)-\psi_{\mcal P}^{+}(x;\tilde\gamma)}\le0$. Step $(d)$ follows from the definition of $\brbra{\iota,\nu,\gamma}$-$\PWcer$.
    Finally, we verify the second relationship in $\PWcer$ as the following inequality:
    \begin{equation}
        \begin{aligned}
            \bsbra{\mcal G(\bar{y})}_+\leq \frac{2\iota\nu}{\gamma} = \frac{2\brbra{\tilde{\gamma}-\gamma}\iota\nu}{\tilde{\gamma}\gamma}+\frac{2\iota\nu}{\tilde{\gamma}}\leq\frac{2\iota\tilde{\nu}}{\tilde{\gamma}}\ .
        \end{aligned}
    \end{equation}
    \ifdefined\isarxiv
\end{proof}
\else
\qedsymbol\end{proof}
\fi


\begin{figure}[t]
    \centering
    \begin{minipage}[t]{0.48\textwidth}
        \centering
        \includegraphics[width=\linewidth]{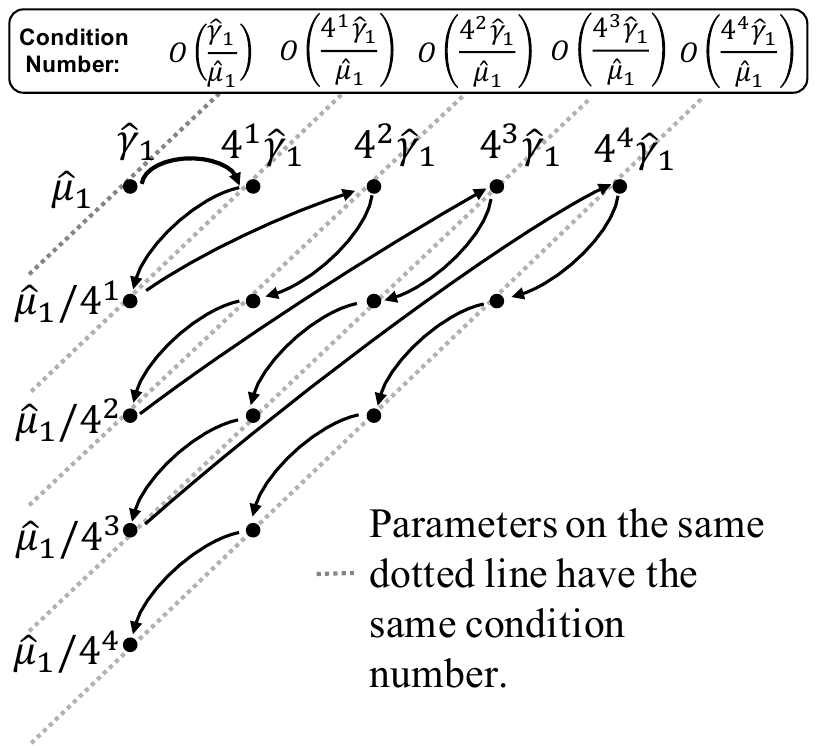}
        \captionof{figure}{Search order of parameter pairs in $\adjustPara$.\protect\label{fig:Parameter-Adjustment-Strategy}}
    \end{minipage}\hfill
    \begin{minipage}[t]{0.50\textwidth}
        \centering
        \includegraphics[width=\linewidth]{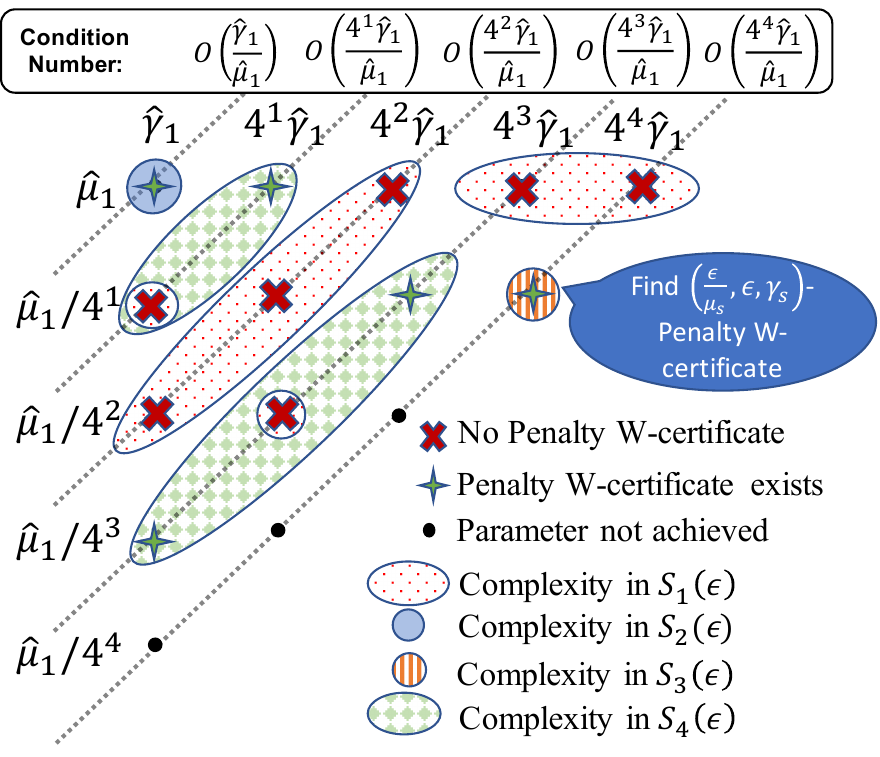}
        \captionof{figure}{Complexity decomposition for the function-constrained setting.\protect\label{fig:function-constrained-complexity}}
    \end{minipage}
\end{figure}
\section{\protect\label{sec:rpapex}Restarted \papex}
In this section, we propose an almost parameter-free restarted penalty method that
does not require a priori knowledge of either  $\norm{\lambda^*}_1$ or the QG modulus $\mu^*$.
\begin{function}
    \caption{ParamAdjust$\brbra{\bar{y},\protect\brbra{\hat{\protect\gamma},\hat{\mu}},\Cert,\bar{y}^{0},\protect\brbra{\hat{\protect\gamma}_{1},\hat{\mu}_{1}}}$\label{alg:parameter_adjust}}
    \begin{algorithmic}[1]
        \State $\brbra{y_c,\iota_c,\nu_c,\gamma_c}\leftarrow \Cert$
        \Statex \vspace{-0.5ex}{\footnotesize\dotfill\ \textit{Update parameter  estimates along the search direction in Figure~\ref{fig:Parameter-Adjustment-Strategy}}}\vspace{-0.5ex}
        \State \label{enu:adjust_parameter_begin}$\hat{\gamma}_{\text{new}}=\hat{\gamma}/4$, $\hat{\mu}_{\text{new}}=\hat{\mu}/4$
        \If{$\hat{\gamma}_{\text{new}}<\hat{\gamma}_{1}$}
        \State \label{enu:adjust_parameter_end}$\hat{\gamma}_{\text{new}}=\hat{\gamma}_{1}\cdot\hat{\mu}_{1}/\hat{\mu}_{\text{new}}$, $\hat{\mu}_{\text{new}}=\hat{\mu}_{1}$
        \EndIf
        \Statex \vspace{-0.5ex}{\footnotesize\dotfill\ \textit{Certificate transformation and optimality gap estimate update}}\vspace{-0.5ex}
        \State \label{enu:certificate_transfer_begin}$\bar{\nu}=\max\bcbra{\hat{\gamma}_{\text{new}}-\gamma_c,0}\cdot\frac{2\nu_c}{\gamma_c}+\nu_c$
        \State $\hat{y}=\argmin_{x\in\bcbra{\bar{y},y_c, \bar{y}^0}}\bcbra{\phi\brbra{x;\hat{\gamma}_{\text{new}}}}$\label{enu:bar_y}
        \State \label{enu:lower_bound_compute}$\underline{\phi} = \max\bcbra{\phi\brbra{y_c;\max\bcbra{\gamma_c,\hat{\gamma}_{\text{new}}}}-\max\bcbra{\iota_c \bar{\nu},\frac{2\bar{\nu}^{2}}{\hat{\mu}_{\text{new}}}},\max_{y\in \bcbra{\bar{y}, y_c, \bar{y}^0}}\bcbra{\phi(y;\hat{\gamma}_{\text{new}})-\frac{2\norm{\phi^{\prime}(y;\hat{\gamma}_{\text{new}})}^2}{\hat{\mu}_{\text{new}}}}}$
        \State \label{enu:certificate_transfer_end}$\tilde{\Delta}_{\text{new}}=\phi(\hat{y};\hat{\gamma}_{\text{new}}) - \underline{\phi}$\label{enu:gap_parameter_adjust}
        \State \textbf{return} $\brbra{\hat{\gamma}_{\text{new}},\hat{\mu}_{\text{new}},\hat{y},\tilde{\Delta}_{\text{new}}}$
    \end{algorithmic}
\end{function}

Algorithm~\ref{alg:rAPEXC} maintains two estimates, $\hat\gamma_s$ and
$\hat\mu_s$, and the induced gap estimate $\tilde\Delta_s$. The record $\Cert$
is a four-entry tuple
$\Cert=\brbra{y_{\mathrm c},\iota_{\mathrm c},\nu_{\mathrm c},\gamma_{\mathrm c}}$,
where $y_{\mathrm c}$ is the reference solution, $\iota_{\mathrm c}$ is the certificate radius, $\nu_{\mathrm c}$ is
the normalized Wolfe tolerance, and
$\gamma_{\mathrm c}$ is the penalty parameter used for that certificate (see Definition~\ref{def:CWcer}). The
tuple stores only these certificate parameters, not the evaluation point set;
Proposition~\ref{prop:certificate_generation} guarantees the existence of the
corresponding point set whenever $\Cert$ is refreshed at
Line~\ref{enu:store_certificate}. The initial value of $\Cert$ is only a
placeholder until the first finite certificate is generated.

At outer stage $s$, Algorithm~\ref{alg:rAPEXC} first calls $\pagr$
(Algorithm~\ref{alg:ACGR}). If $\lowerflag=\true$ and
$[\mcal G(\bar y^s)]_+\le 2\theta\tilde\Delta_s/\hat\gamma_s$, the algorithm
bypasses parameter validation; otherwise, it proceeds to the validation branch.

During parameter validation, the algorithm first tests $\hat x^s$ against the
lower bound $\phi(\bar y^s;\hat\gamma_s)-\tilde\Delta_s$. A violation rejects
$(\hat\gamma_s,\hat\mu_s)$ immediately; if no violation occurs, the algorithm
calls $\pawg$, and $\paramflag=\false$ also rejects the pair.
Proposition~\ref{prop:flag_false} shows that either outcome certifies
$\hat\gamma_s<2\norm{\lambda^*}_1$ or $\hat\mu_s>\mu^*$. Function~\ref{alg:parameter_adjust}
then updates the pair, recomputes $\tilde\Delta_s$ using the stored certificate,
and reruns $\pagr$.

When the stage is accepted, the unified refresh sets $\bar y^{s+1}=\hat x^s$
and retains the current parameter estimates. It stores the certificate
when $\lowerflag=\true$ and the feasibility threshold holds; otherwise, it
stores the certificate justified by the $\pawg$ validation. Finally, it computes
$\tilde\Delta_{s+1}$ using the lower bound
$\phi(\bar y^s;\hat\gamma_s)-\theta\tilde\Delta_s$ if $\lowerflag=\true$, and
$\phi(\bar y^s;\hat\gamma_s)-\tilde\Delta_s$ otherwise.

The main difference from the unconstrained parameter-free scheme
of~\cite{partOne} is that the parameter guess is two-dimensional. In the unconstrained
case, $\rapex$ adjusts only the QG estimate; here, an invalid stage may come
from either an underestimated penalty parameter or an overestimated QG modulus,
as formalized in Proposition~\ref{prop:flag_false}. Subroutine $\adjustPara$
(Function~\ref{alg:parameter_adjust}) organizes this search by the induced
condition number: Line~\ref{enu:adjust_parameter_begin} quarters both
estimates; the block ending at Line~\ref{enu:adjust_parameter_end} moves to the
next search diagonal when the penalty estimate would fall below its initial
value; Line~\ref{enu:certificate_transfer_begin} transfers the stored
certificate to the updated penalty parameter; and
Line~\ref{enu:lower_bound_compute} forms a lower bound by taking the maximum of
the transferred-certificate bound and the local-subgradient bound under the new QG
estimate; Line~\ref{enu:gap_parameter_adjust} then recomputes the gap estimate
from this lower bound. This mechanism
lets Algorithm~\ref{alg:rAPEXC} remain almost parameter-free while preserving
certified lower-bound information across parameter updates.

\begin{algorithm}[htp]
    \caption{\protect\label{alg:ACGR}Accelerated Gap Reduction for the exact penalty function, $\pagr \protect\brbra{\protect\gamma,\bar{y},\Delta,\mu,\bundleSize, \theta}$}
    \begin{algorithmic}[1]
        \State $\tilde{l}= \phi\brbra{\bar{y};\gamma}-\theta\Delta$, $x^{1}=\bar{y},\hat{x}^{1} = \bar{y}$, $\underline{\phi}=\phi\brbra{\bar{y};\gamma}-\Delta$
        \For{$t=1,\ldots$}
        \State $\brbra{\hat{x}^{t+1},x^{t+1},\flag}=\onestepplus\brbra{\gamma,\bar{y},\hat{x}^{t},x^{t}\mid \tilde{l},\bundleSize}$
        \If{$\flag $ is $ \true$ or $\norm{x^{t+1}-\bar{y}}^{2}>\frac{2\theta\Delta}{\mu}$\label{enu:radius_test}}
        \State \textbf{return} $\brbra{\hat{x}^{t+1},\true}$ \Comment{$\true$: certificate is generated} \label{enu:lower_improve}
        \ElsIf{\(\phi(\hat x^{t+1};\gamma)-\underline\phi \le \theta\Delta\)}
        \State \textbf{return} $\brbra{\hat{x}^{t+1},\false}$ \Comment{$\false$: no certificate}\label{enu:prox_center_change}
        \EndIf
        \EndFor
    \end{algorithmic}
\end{algorithm}

\begin{algorithm}
    \caption{Restarted $\papex$\protect\label{alg:rAPEXC}}
    \begin{algorithmic}[1]
        \State \textbf{Input}: initial penalty parameter $\hat{\gamma}_{1}$ (recommended to be 2), initial QG coefficient $\hat{\mu}_{1}>0$, number of cuts $\bundleSize$, shrink factor $\theta\in\brbra{\frac{1}{2},1}$, extra factor $\beta>0$, initial solution $\bar{y}^{0}$
        \State \textbf{Initialize}: $\bar{y}^{1}=\bar{y}^{0}$, $\tilde{\Delta}_{1}=\frac{2\norm{\phi^{\prime}\brbra{\bar{y}^{1};\hat{\gamma}_{1}}}^{2}}{\hat{\mu}_{1}},\Cert=\brbra{\bar{y}^0, 1,+\infty,\hat{\gamma}_1}$
        \For{$s=1,\ldots$}

        \State $\brbra{\hat{x}^{s},\lowerflag}=\pagr\brbra{\hat{\gamma}_{s},\bar{y}^{s},\tilde{\Delta}_{s},\hat{\mu}_{s}\mid\bundleSize, \theta}$\label{enu:run_ACGR}
        \vspace{-1.5ex}
        \Statex {\footnotesize\dotfill\ \textit{No certificate by $\pagr$: run $\pawg$ and check parameters}}\vspace{-0.3ex}
        \If{$\lowerflag$ is $\false$ or $[\mcal G(\bar{y}^s)]_+> \frac{2\theta\tilde{\Delta}_s}{\hat{\gamma}_s}$}\label{enu:lower_and_upper_improve}
        \State $\paramflag=\true$
        \If{$\phi(\hat{x}^s;\hat{\gamma}_s)\geq \phi(\bar{y}^s;\hat{\gamma}_s) - \tilde{\Delta}_{s}$}\label{enu:upper_candidate_check}
        \State ${\paramflag}=\pawg\brbra{\hat{\gamma}_{s},\bar{y}^{s},\tilde{\Delta}_{s},\sqrt{\frac{2\brbra{1+\beta}\tilde{\Delta}_{s}}{\hat{\mu}_{s}}}\mid\bundleSize,\beta}$\label{enu:run_APWG}
        \EndIf

        \If{$\paramflag$ is $\false$ or $\phi(\hat{x}^s;\hat{\gamma}_s)< \phi(\bar{y}^s;\hat{\gamma}_s) - \tilde{\Delta}_{s}$}\label{enu:parameter_reject}
        \State $\brbra{\hat{\gamma}_{s},\hat{\mu}_{s},\bar{y}^{s},\tilde{\Delta}_{s}}=\adjustPara\brbra{\bar{y}^{s},\brbra{\hat{\gamma}_{s},\hat{\mu}_{s}},\Cert \mid \bar{y}^{0},\brbra{\hat{\gamma}_{1},\hat{\mu}_{1}}}$ and go to Line~\ref{enu:run_ACGR}\label{enu:adjust_parameter1}
        \EndIf

        \EndIf
        \vspace{-1.2ex}
        \Statex {\footnotesize\dotfill\ \textit{Use current best solution  as next candidate; update certificate}}\vspace{-0.3ex}
        \State \label{enu:store_certificate}$\bar{y}^{s+1}=\hat{x}^{s}$, $\hat{\gamma}_{s+1}=\hat{\gamma}_{s}$, $\hat{\mu}_{s+1}=\hat{\mu}_{s}$ and set
            {\footnotesize\begin{equation*}
                    \Cert=\protect\begin{cases}
                        \protect\brbra{\bar{y}^{s},\sqrt{\frac{2\theta\tilde{\Delta}_{s}}{\hat{\mu}_{s}}},\sqrt{\frac{\theta\hat{\mu}_{s}\tilde{\Delta}_{s}}{2}},\hat{\gamma}_{s}}       & \lowerflag\ \text{is}\ \true\  \text{and}\ [\mcal G(\bar{y}^s)]_+ \leq  \frac{2\theta\tilde{\Delta}_s}{\hat{\gamma}_s} \\
                        \protect\brbra{\bar{y}^{s},\sqrt{\frac{2(1+\beta)\tilde{\Delta}_{s}}{\hat{\mu}_{s}}},\sqrt{\frac{(1+\beta)\hat{\mu}_{s}\tilde{\Delta}_{s}}{2}},\hat{\gamma}_{s}} & \text{otherwise}
                        \protect\end{cases}
                \end{equation*}}{\footnotesize\begin{equation*}
                    \tilde{\Delta}_{s+1}=\protect\begin{cases}
                        \phi\protect\brbra{\bar{y}^{s+1};\hat{\gamma}_{s}}-\protect\brbra{\phi\protect\brbra{\bar{y}^{s};\hat{\gamma}_{s}}-\theta\tilde{\Delta}_{s}} & \lowerflag\ \text{is}\ \true  \\
                        \phi\protect\brbra{\bar{y}^{s+1};\hat{\gamma}_{s}}-\protect\brbra{\phi\protect\brbra{\bar{y}^{s};\hat{\gamma}_{s}}-\tilde{\Delta}_{s}}       & \lowerflag\ \text{is}\ \false
                        \protect\end{cases}
                \end{equation*}}

        \EndFor
    \end{algorithmic}
\end{algorithm}

We now state the convergence results for restarted $\papex$.
Proposition~\ref{prop:convergence_rapex_gap_reduction} bounds the cost of one
call to $\pagr$, and Proposition~\ref{prop:flag_false} bounds the number
of parameter-adjustment steps. Theorem~\ref{thm:convergence_CW_certificate}
then gives the certificate-complexity bound. Finally,
Theorem~\ref{thm:convergence_opt_convio} and
Corollary~\ref{cor:convergence_opt_convio_small_mu} give the corresponding
optimality-gap and constraint-violation guarantees under the two initial-QG
regimes.


We begin with the single-call estimate for the gap-reduction routine
$\pagr$.

\begin{proposition}\label{prop:convergence_rapex_gap_reduction}
    Choose $\omega_t=(t+2)(t+3)/2$, $\alpha_t=4/(t+3)$, and $\theta\in(1/2,1)$ in Algorithm~\ref{alg:ACGR}. If Algorithm~\ref{alg:ACGR}
    returns at the $t_{\mcal G}$-th iteration, then the weighted mean of empirical local Lipschitz constants
    $\bar L_{\gamma}\brbra{t_{\mcal G}-1}$ and the iteration number $t_{\mcal G}$
    satisfy
    \begin{equation}\label{eq:res_PAGR_converge}
        \bar{L}_{\gamma}\brbra{t_{\mcal G}-1}\leq
        \frac{24\theta}{2\theta-1}
        \frac{\sum_{t=1}^{t_{\mcal G}-1}\onebf_{\mcal E(t)}(t+3)}
        {\sum_{t=1}^{t_{\mcal G}-1}\brbra{\onebf_{\mcal E(t)}(t+3)/\tilde L_{\gamma,t}}}
        \,\text{and}\,\,
        t_{\mcal G}\leq
        \min\bcbra{t:t\geq\left\lceil
            \sqrt{\frac{6\theta+4\theta\bar L_{\gamma}\brbra t/\mu}{2\theta-1}}
            \right\rceil}.
    \end{equation}
    Here $\tilde L_{\gamma,t}$ is defined in~\eqref{eq:defn_L_tilde_phi}, $\mcal E(t)$ is defined after~\eqref{eq:L_phi_N}, and $0/0=0$.
\end{proposition}

\noindent\emph{Proof deferred.} The detailed proof is collected in Subsection~\ref{subsec:rpapex-detailed-proofs}.

The two preceding propositions (Proposition~\ref{prop:convergence_rapex_gap_reduction} and Proposition~\ref{prop:APWG_W_cer_gen_complexity}) give comparable single-call iteration costs for
$\pagr$ and $\pawg$. For $\pawg$ with
$\iota_{\max}^{2}=2(1+\beta)\Delta/\mu$, both bounds reduce to
$O(1)\sqrt{(L_f+\gamma L_{\mcal G})/\mu}$ under the bundle-size condition
$\bundleSize\geq k_f + k_{\mcal G}$. Their outputs play
different roles. The routine $\pawg$ attempts to generate a $\PWcer$; if it
fails, Proposition~\ref{prop:APWG_W_cer_gen_complexity} shows that the current parameter pair must be
invalid. In contrast, $\pagr$ reduces the gap by improving either the
upper or lower bound, and it leads to parameter adjustment only when the
resulting bounds are inconsistent. The next proposition formalizes these
adjustment events and bounds their number.

%
%


%
\begin{proposition}
    \label{prop:flag_false}
    If Algorithm~\ref{alg:rAPEXC} enters Line~\ref{enu:adjust_parameter1},
    then the current pair $(\hat{\gamma}_{s},\hat{\mu}_{s})$ satisfies
    $\hat{\mu}_{s}>\mu^{*}$ or $\hat{\gamma}_{s}<2\norm{\lambda^*}_1$. Consequently,
    Algorithm~\ref{alg:rAPEXC} enters Line~\ref{enu:adjust_parameter1} at most $\bar{N}=\guessTime\brbra{\guessTime+3}/2$ times, where
    \begin{equation}\label{eq:defmcalT}
        \guessTime:=\max\bcbra{\left\lceil \log_{4}\frac{\hat{\mu}_{1}}{\mu^{*}}\right\rceil ,0}
        +\max\bcbra{\left\lceil \log_{4}\frac{2\norm{\lambda^*}_1}{\hat{\gamma}_{1}}\right\rceil ,0}\ .
    \end{equation}
    Moreover, for every $s\geq 1$, we have $\hat{\mu}_{s}\ge \hat{\mu}_{1}/4^{\guessTime}$ and
    $\hat{\gamma}_{s}\le \hat{\gamma}_{1}\cdot4^{\guessTime}$.
\end{proposition}

\noindent\emph{Proof deferred.} The detailed proof is collected in Subsection~\ref{subsec:rpapex-detailed-proofs}.

%

%

Proposition~\ref{prop:flag_false} is the two-parameter analogue of the
one-parameter guess-and-check result of~\cite{partOne}. That result counts only
reductions in the QG estimate, whereas $\guessTime$ in~\eqref{eq:defmcalT}
accounts for the search over both $\hat{\mu}_s$ and $\hat{\gamma}_s$.

To state the complexity theorem, we first associate an empirical Lipschitz
average with each subroutine execution. Consider any execution $\mcal C$ of
$\pagr$ or ${\pawg}$. Its current parameter pair can be written as
\[
    (\hat{\mu}_s,\hat{\gamma}_s)
    =\brbra{\hat{\mu}_{1}/4^j,\hat{\gamma}_{1}4^{i-j}},
    \qquad 0\leq j\leq i\leq \guessTime .
\]
To compare empirical Lipschitz constants obtained with different penalty
parameters, we normalize each by the factor $\hat{\gamma}_1/\hat{\gamma}_s$,
which leads to the following definition.

\begin{defn}[Penalty-rescaled average empirical Lipschitz constant]
    \label{defn:scaled_empirical_lip_avg}
    For a target accuracy $\delta>0$, let
    $\mathcal S_\delta$ collect all completed calls $\mcal C$ to $\pagr$ or ${\pawg}$ before
    the first time at which Algorithm~\ref{alg:rAPEXC}
    stores a certificate guaranteeing a
    $(\delta/\hat\mu_s,\delta,\hat\gamma_s)$-Penalty W-certificate. For each such call $\mcal C$, let $\hat\gamma_{\mcal C}$ be its penalty parameter,
    let
    $\bar L^{\mcal C}_{\hat\gamma_{\mcal C}}(N_{\mcal C})$ be the corresponding empirical Lipschitz average, where $(N_{\mcal C}+1)$ is the number of iterations for execution $\mcal C$.
    Define
    \begin{equation}\label{eq:def_L_avg}
        L_{\rm avg}(\delta):=
        \begin{cases}
            \max_{\mcal C\in\mathcal S_\delta}
            \frac{\hat\gamma_1}{\hat\gamma_{\mcal C}}
            \bar L^{\mcal C}_{\hat\gamma_{\mcal C}}(N_{\mcal C}),
               & \mathcal S_\delta\neq\emptyset, \\[1ex]
            0, & \mathcal S_\delta=\emptyset .
        \end{cases}
    \end{equation}

\end{defn}
Two remarks clarify this definition. First, the rescaling
in~\eqref{eq:def_L_avg} is applied to each execution before taking the maximum.
Indeed,
$\hat{\gamma}_1/\hat{\gamma}_{\mcal C}$ normalizes the
empirical Lipschitz average for $\phi(\cdot;\hat{\gamma}_{\mcal C})$ to the scale
of the initial penalty function $\phi(\cdot;\hat{\gamma}_1)$. Second, for each execution $\mcal C$, the
corresponding subroutine guarantee,
\eqref{eq:weighted_harmonic_mean_APWG} or~\eqref{eq:res_PAGR_converge}, gives
the bound
\[
    \bar L_{\hat{\gamma}_{\mcal C}}^{\mcal C}(N_{\mcal C})
    \leq \max\Bcbra{\frac{24(1+\beta)}{\beta},
        \frac{24\theta}{2\theta-1}}
    \frac{\sum_{t=1}^{N_{\mcal C}}\onebf_{\mcal E_{\mcal C}(t)}(t+3)}
    {\sum_{t=1}^{N_{\mcal C}}\brbra{\onebf_{\mcal E_{\mcal C}(t)}(t+3)/
        \tilde L_{\hat{\gamma}_{\mcal C},t}^{\mcal C}}},
\]
where $\tilde L_{\hat{\gamma}_{\mcal C},t}^{\mcal C}$ and $\mcal E_{\mcal C}(t)$ denote the
corresponding quantities in the subroutine execution $\mcal C$, and the
convention $0/0=0$ is used.
When the event set is nonempty, the right-hand side is a constant multiple of a
weighted harmonic mean of the local empirical Lipschitz constants over the
indices satisfying $\mcal E_{\mcal C}(t)$. The corresponding counting result
implies that, for long subroutine runs, the participation ratio of this harmonic mean,
namely the fraction of indices satisfying $\mcal E_{\mcal C}(t)$, approaches
$\sqrt{\theta}$; see the second remark following Theorem~\ref{thm:ideal_papex}.

Now, we are ready to present the main complexity results.
Theorem~\ref{thm:convergence_CW_certificate} bounds the oracle cost until the
record $\Cert$ guarantees the existence of a point set forming an
$(\vep/\hat{\mu}_s,\vep,\hat{\gamma}_s)$-$\PWcer$.
Theorem~\ref{thm:convergence_opt_convio} and
Corollary~\ref{cor:convergence_opt_convio_small_mu} then establish
guarantees for $\vep$-optimality and $\vep$-constraint violation under two
different initial parameter settings.
\begin{thm}
    \label{thm:convergence_CW_certificate}
    In Algorithm~\ref{alg:rAPEXC}, suppose the initial parameter estimate
    $(\hat{\mu}_1,\hat{\gamma}_1)$ satisfies $\hat{\mu}_1 > \mu^*$ or
    $\hat{\gamma}_1 < 2\norm{\lambda^*}_1$. Then, up to the first time at which the
    certificate stored in $\Cert$ guarantees the existence of a point set
    $\mcal P_s$ that forms an
    $\brbra{\vep/\hat{\mu}_s,\vep,\hat{\gamma}_s}$-$\PWcer$ for $\bar{y}^{s}$ in
    problem~\eqref{eq:initial_FP}, the total number of first-order oracle calls is at most
    \begin{equation}\label{eq:complexity_rapexc_lavg}
        O(1) \cdot \bundleSize \cdot \max\bcbra{2^{\guessTime}\sqrt{\frac{L_{\text{avg}}}{\hat{\mu}_{1}}},1}
        \left(
        \guessTime + \bsbra{\log_{\tfrac{1}{\theta}} \frac{ \norm{f^{\prime}\brbra{\bar{y}^{0}}}^{2}+\hat{\gamma}_1^2\norm{\mcal G^{\prime}\brbra{\bar{y}^{0}}}^{2} }{ \vep^{2} }}_+
        \right),
    \end{equation}
    where $\guessTime$ is defined in~\eqref{eq:defmcalT}, and
    $L_{\text{avg}}=L_{\text{avg}}(\vep)$ is defined in~\eqref{eq:def_L_avg}.
    Furthermore, if $\bundleSize\geq k_f + k_{\mcal G}$, then
    $L_{\text{avg}}\leq O(1)(L_f+ {\hat{\gamma}}_{1} L_{\mcal G})$, and the preceding
    bound becomes
    \begin{equation}\label{eq:complexity_rapexc}
        O(1) \cdot \bundleSize \cdot \max\bcbra{2^{\guessTime} \sqrt{\frac{L_{f}}{\hat{\mu}_{1}} + \frac{\hat{\gamma}_{1}}{\hat{\mu}_{1}} L_{\mcal G} },1}
        \left(
        \guessTime + \bsbra{\log_{\tfrac{1}{\theta}} \frac{ \norm{f^{\prime}\brbra{\bar{y}^{0}}}^{2}+\hat{\gamma}_1^2\norm{\mcal G^{\prime}\brbra{\bar{y}^{0}}}^{2} }{ \vep^{2} }}_+
        \right),
    \end{equation}
\end{thm}

\noindent\emph{Proof deferred.} The detailed proof is collected in Subsection~\ref{subsec:rpapex-detailed-proofs}.

The preceding Theorem~\ref{thm:convergence_CW_certificate} gives the complexity of obtaining an
$(\vep/\hat{\mu}_s,\vep,\hat{\gamma}_s)$-$\PWcer$, which can be verified during the
algorithm. We next use this result to derive a complexity bound for the
objective gap and constraint violation.
\begin{thm}
    \label{thm:convergence_opt_convio}
    In Algorithm~\ref{alg:rAPEXC}, suppose that $\hat{\mu}_{1}>\mu^{*}$ and
    set the initial penalty estimate to $\hat{\gamma}_{1}=2$. Let
    $\delta:=\sqrt{\frac{\mu^{*}\vep}{16\norm{\lambda^*}_1(2\norm{\lambda^*}_1-1)}}$.
    Then Algorithm~\ref{alg:rAPEXC} produces a point $\bar{y}^{s}$ such that
    \begin{equation}\label{eq:epsilon_condition}
        \max\bcbra{f\brbra{\bar{y}^{s}}-f^{*}, \bsbra{\mcal G\brbra{\bar{y}^{s}}}_{+}}\leq\vep
    \end{equation}
    after at most
    \begin{equation}\label{eq:num_grad_oracle_lavg}
        O(1)\,\bundleSize\cdot\max\bcbra{\sqrt{\frac{\norm{\lambda^*}_1L_{\text{avg}}}{\mu^{*}}},1}
        \brbra{\guessTime+\bsbra{\log_{\tfrac{1}{\theta}}
                \frac{2\norm{\lambda^*}_1
                    \brbra{\norm{f^{\prime}\brbra{\bar{y}^{0}}}^{2}
                        +\norm{\mcal G^{\prime}\brbra{\bar{y}^{0}}}^{2}}}
                {\vep}}_+}
    \end{equation}
    first-order oracle calls, where
    $\guessTime:=\left\lceil \log_{4}\frac{\hat{\mu}_{1}}{\mu^{*}}\right\rceil
        +\left\lceil \log_{4}\norm{\lambda^*}_1\right\rceil$ and $L_{\text{avg}} = L_{\text{avg}}(\delta)$.
    Furthermore, if $\bundleSize\geq k_f + k_{\mcal G}$, then
    $L_{\text{avg}}\leq O(1)(L_f+2 L_{\mcal G})$, and the preceding bound becomes
    \begin{equation}\label{eq:num_grad_oracle}
        O(1)\,\bundleSize\cdot\max\bcbra{\sqrt{\frac{\norm{\lambda^*}_1}{\mu^{*}}(L_{f}+{2}L_{\mcal G})},1}
        \brbra{\guessTime+\bsbra{\log_{\tfrac{1}{\theta}}
                \frac{2\norm{\lambda^*}_1
                    \brbra{\norm{f^{\prime}\brbra{\bar{y}^{0}}}^{2}
                        +\norm{\mcal G^{\prime}\brbra{\bar{y}^{0}}}^{2}}}
                {\vep}}_+}
    \end{equation}
    first-order oracle calls.
\end{thm}

A few remarks regarding this complexity results in Theorem~\ref{thm:convergence_CW_certificate} and Theorem~\ref{thm:convergence_opt_convio} are in order.

First, the bounds~\eqref{eq:complexity_rapexc_lavg} and~\eqref{eq:num_grad_oracle_lavg}
are stated in terms of the trajectory-dependent quantity $L_{\text{avg}}$ in
\eqref{eq:def_L_avg}. The additional condition
$\bundleSize\geq k_f + k_{\mcal G}$ is only needed to
replace $L_{\text{avg}}$ by the worst-case smoothness bound
$O(1)(L_f+\hat{\gamma}_1 L_{\mcal G})$ in Theorem~\ref{thm:convergence_CW_certificate} or $O(1)(L_f+2L_{\mcal G})$ in Theorem~\ref{thm:convergence_opt_convio}. As discussed after
the remarks following Theorem~\ref{thm:ideal_papex}, the quantity
$\max\bcbra{k_f,\sum_{i\in \activeconstrSet}k_{g_i}}$ is only a worst-case bound on the number
of pieces. In practice, Algorithm~\ref{alg:rAPEXC} often requires far fewer
cuts, since its performance depends more on the pieces encountered along the
trajectory, especially those induced by active constraints along the trajectory. A special case is when the objective and constraints are smooth, so $k_f=k_{g_i}=1$ for all $i\in\activeconstrSet$. In this case, $\bundleSize = {k_f+k_{\mcal G}}$ suffices to ensure $L_{\text{avg}}=O(1)(L_f+\hat{\gamma}_1 L_{\mcal G})$ and the final complexity bound in~\eqref{eq:num_grad_oracle} can be written as
\begin{equation*}
    O(1)\,(\activeConstrNum+1)\cdot\max\bcbra{\sqrt{\frac{\norm{\lambda^*}_1}{\mu^{*}}(L_{f}+{2}L_{\mcal G})},1}
    \brbra{\guessTime+\bsbra{\log_{\tfrac{1}{\theta}}
            \frac{2\norm{\lambda^*}_1
                \brbra{\norm{f^{\prime}\brbra{\bar{y}^{0}}}^{2}
                    +\norm{\mcal G^{\prime}\brbra{\bar{y}^{0}}}^{2}}}
            {\vep}}_+}\ .
\end{equation*}

Second, if the objective function $f$ and each constraint $g_i$ are $L$-smooth,
then after suppressing the dependence on the numbers of constraints and pieces in~\eqref{eq:num_grad_oracle}, the bound
simplifies to
$$
    O(1)\sqrt{\frac{\norm{\lambda^*}_1 L}{\mu^*}}\brbra{\log\frac{1}{\vep}+\left\lceil \log_{4}\frac{\hat{\mu}_{1}}{\mu^{*}}\right\rceil+\left\lceil \log_{4}\norm{\lambda^*}_1\right\rceil}.
$$
This matches the lower bound in~\cite{zhang2022solving} up to the additional logarithmic overhead required by the parameter-free anytime setting. To the best of our knowledge, this is the first almost parameter-free anytime algorithm to achieve such a guarantee.

Third, the initial parameter estimates in Algorithm~\ref{alg:rAPEXC} can be chosen flexibly.
If $\hat{\gamma}_{1} \leq 2\norm{\lambda^*}_1$ and
$\hat{\mu}_{1} \geq \mu^{*}$, then
$\guessTime=\lceil \log_{4}(\hat{\mu}_{1}/\mu^{*})\rceil
    +\lceil \log_{4}(2\norm{\lambda^*}_1/\hat{\gamma}_{1})\rceil$,
and the leading factor in~\eqref{eq:complexity_rapexc} obeys
$2^{\guessTime}\sqrt{L_{f}/\hat{\mu}_{1}+(\hat{\gamma}_{1}/\hat{\mu}_{1})L_{\mcal G}}
    \le O(1)\sqrt{(\norm{\lambda^*}_1/\mu^{*})\brbra{L_{f}/\hat{\gamma}_{1}+L_{\mcal G}}}$.
Hence, a more accurate initial estimate of $\norm{\lambda^*}_1$ improves the leading
condition dependence. In particular, because $\norm{\lambda^*}_1\ge 1$, choosing
$\hat{\gamma}_{1}=2$ in Theorem~\ref{thm:convergence_opt_convio}
automatically ensures $\hat{\gamma}_{1}\le 2\norm{\lambda^*}_1$ and yields, up to
absolute constants, the same dependence
$
    \sqrt{\brbra{\norm{\lambda^*}_1 / \mu^*}(L_f + L_{\mcal G})}.
$
Obtaining an upper bound on $\mu^{*}$ is more subtle. Suppose that a strict feasible
point $\bar{x}$ satisfying $\mcal G(\bar{x})<0$ is available. Then standard
exact-penalty arguments imply
$$2\norm{\lambda^*}_1\leq \bar{\gamma}:=2\max\bcbra{1,\frac{f(\bar{x})-\min_{x\in X}f(x)}{{-\mcal G(\bar{x})}}}.$$ Therefore, $\phi(\cdot;\bar{\gamma})$ is also $\mu^{*}$-QG. Thus, for any
$x_{1},x_{2}$ with $\phi\brbra{x_{1};\bar{\gamma}}>\phi\brbra{x_{2};\bar{\gamma}}$,
we obtain
$\phi\brbra{x_{1};\bar{\gamma}}-\phi\brbra{x_{2};\bar{\gamma}}
    \le \phi\brbra{x_{1};\bar{\gamma}}-\min_{x\in X}\phi\brbra{x;\bar{\gamma}}
    \le 2\norm{\phi^{\prime}\brbra{x_{1};\bar{\gamma}}}^{2}/\mu^{*}$,
so
$\mu^{*}\le
    2\norm{\phi^{\prime}\brbra{x_{1};\bar{\gamma}}}^{2}/
    \brbra{\phi\brbra{x_{1};\bar{\gamma}}-\phi\brbra{x_{2};\bar{\gamma}}}$.
Therefore, once such a $\bar{\gamma}$ is known, two sample points suffice to
produce an explicit upper bound on $\mu^{*}$. In the absence of a strict feasible
point, however, the final complexity naturally depends on the initial choice
of $\hat{\mu}_{1}$, which motivates
Corollary~\ref{cor:convergence_opt_convio_small_mu}.

\begin{corollary}\label{cor:convergence_opt_convio_small_mu}
    In Algorithm~\ref{alg:rAPEXC}, suppose that $\bundleSize\geq k_f + k_{\mcal G}$ and $\hat{\mu}_{1}<\mu^{*}$. If the initial penalty estimate is set to $\hat{\gamma}_{1}=2$, then Algorithm~\ref{alg:rAPEXC} produces a point $\bar{y}^{s}$ such that
    \begin{equation*}
        \max\bcbra{f\brbra{\bar{y}^{s}}-f^{*}, \bsbra{\mcal G\brbra{\bar{y}^{s}}}_{+}}\leq\vep\ ,
    \end{equation*}
    after at most
    \begin{equation}\label{eq:num_grad_oracle_small_mu}
        O(1)\,\bundleSize\sqrt{\frac{\norm{\lambda^*}_1}{\hat{\mu}_{1}}(L_{f}+2L_{\mcal G})}
        \brbra{\guessTime+\bsbra{\log_{\tfrac{1}{\theta}}
                \frac{\norm{f^{\prime}\brbra{\bar{y}^{0}}}^{2}+\norm{\mcal G^{\prime}\brbra{\bar{y}^{0}}}^{2}}{\hat{\mu}_{1}\vep}}_+}
    \end{equation}
    first-order oracle calls, where $\guessTime:=\left\lceil \log_{4}\norm{\lambda^*}_1\right\rceil$.
\end{corollary}

\noindent\emph{Proof deferred.} The detailed proof is collected in Subsection~\ref{subsec:rpapex-detailed-proofs}.

Comparing Corollary~\ref{cor:convergence_opt_convio_small_mu} with
Theorem~\ref{thm:convergence_opt_convio} shows how the initial QG estimate
enters the final complexity bound. If $\hat{\mu}_{1}<\mu^{*}$, the leading
factor scales as $1/\sqrt{\hat{\mu}_{1}}$. If
$\hat{\mu}_{1}\geq\mu^{*}$, this factor is replaced by
$1/\sqrt{\mu^{*}}$, and the overestimate affects only the logarithmic term in
$\guessTime$.

\subsection{Detailed proofs\label{subsec:rpapex-detailed-proofs}}

We close this section by collecting the detailed proofs for the Restarted
    {\papex} analysis in their dependency order. Proposition~\ref{prop:convergence_rapex_gap_reduction}
gives the single-call gap-reduction cost for $\pagr$.
Proposition~\ref{prop:flag_false} identifies parameter-adjustment steps with
invalid parameter pairs and bounds their number, and
Proposition~\ref{prop:certificate_generation} proves that every
accepted refresh stores valid $\PWcer$ parameters. The three parameter-adjustment
lemmas then handle the remaining invariants:
Lemma~\ref{lem:tilde_Delta_new_proper} verifies the updated gap estimate,
Lemma~\ref{lem:lower_bound_valid_alg} preserves the lower-bound invariant, and
Lemma~\ref{lem:upper_bound_transfer} bounds the transferred gap estimate. These
results imply Theorem~\ref{thm:convergence_CW_certificate}; the optimality-gap
and constraint-violation guarantees then follow in
Theorem~\ref{thm:convergence_opt_convio} and
Corollary~\ref{cor:convergence_opt_convio_small_mu}.

\begin{proof}[\underline{Proof of Proposition~\ref{prop:convergence_rapex_gap_reduction}}]
    Algorithm~\ref{alg:ACGR} is the penalty analogue of $\mcal{AGR}$, with
    $\tilde l=\phi(\bar y;\gamma)-\theta\Delta$. Before the algorithm returns,
    Line~\ref{enu:prox_center_change} has not been triggered, so
    \[
        \phi(\hat x^{t+1};\gamma)-\tilde l
        > \underline{\phi}+\theta\Delta-\brbra{\phi(\bar y;\gamma)-\theta\Delta}
        \geq (2\theta-1)\Delta .
    \]
    Monotonicity gives the upper bound
    $\phi(\hat x^{t+1};\gamma)-\tilde l\leq\theta\Delta$. Applying
    Proposition~\ref{prop:property_papex} with these two residual bounds yields
    \[
        \frac{1}{2}(t+2)(t+3)(2\theta-1)\Delta-3\theta\Delta
        \leq \bar L_{\gamma}(t)\norm{x^{t+1}-\bar y}^{2}.
    \]
    Hence, once
    $t\geq\sqrt{(6\theta+4\theta\bar L_{\gamma}(t)/\mu)/(2\theta-1)}$, the radius
    test in Line~\ref{enu:radius_test} must be satisfied, proving the bound on
    $t_{\mcal G}$. The bound on $\bar L_{\gamma}(t_{\mcal G}-1)$ follows from the
    same weighted-average argument as in~\eqref{eq:upper_L_phi_s_t_s}, with the
    above lower and upper residual bounds and
    $L_{\gamma,t}\leq2\tilde L_{\gamma,t}$ from Proposition~\ref{prop:upper_bound_Lip}.
    \ifdefined\isarxiv
\end{proof}
\else
\qedsymbol\end{proof}
\fi

\begin{proof}[\underline{Proof of Proposition~\ref{prop:flag_false}}]
    Suppose $(\hat{\gamma}_{s},\hat{\mu}_{s})$ is valid, i.e.,
    $\hat{\mu}_{s}\le\mu^{*}$ and $\hat{\gamma}_{s}\ge2\norm{\lambda^*}_1$. Since
    $\phi(\cdot;2\norm{\lambda^*}_1)$ is $\mu^{*}$-QG, it is also $\hat{\mu}_{s}$-QG.
    Lemma~\ref{lem:lower_bound_valid_alg} then gives
    $\phi(\bar{y}^{s};\hat{\gamma}_{s})-f^{*}\le\tilde{\Delta}_{s}$ and
    $\underline{\phi}\le f^{*}$. Moreover, $\hat{\gamma}_{s}\ge2\norm{\lambda^*}_1$
    implies $\phi(\hat{x}^{s};\hat{\gamma}_{s})\ge f^{*}$. Therefore
    Proposition~\ref{prop:APWG_W_cer_gen_complexity} rules out
    $\paramflag=\false$, and the inequality
    $\phi(\hat{x}^{s};\hat{\gamma}_{s})<\underline{\phi}$ is impossible. Hence
    either event implies $\hat{\mu}_{s}>\mu^{*}$ or $\hat{\gamma}_{s}<2\norm{\lambda^*}_1$.

    Let
    $t_{\mu}:=\max\{\lceil\log_{4}(\hat{\mu}_{1}/\mu^{*})\rceil,0\}$ and
    $t_{\gamma}:=\max\{\lceil\log_{4}(2\norm{\lambda^*}_1/\hat{\gamma}_{1})\rceil,0\}$, so
    $\guessTime=t_{\mu}+t_{\gamma}$. If $\guessTime=0$, then the initial pair is
    already valid and Line~\ref{enu:adjust_parameter1} is never entered. Assume
    now that $\guessTime\ge1$. The routine $\adjustPara$ enumerates the pairs
    $(\hat{\mu}_{1}/4^{j},\hat{\gamma}_{1}\cdot4^{\,i-j})$ with
    $0\le j\le i$ and $i=1,2,\ldots$. At level $i=\guessTime$, the pair with
    $j=t_{\mu}$ equals
    $(\hat{\mu}_{1}/4^{t_{\mu}},\hat{\gamma}_{1}\cdot4^{t_{\gamma}})$, which is
    valid because $\hat{\mu}_{1}/4^{t_{\mu}}\le\mu^{*}$ and
    $\hat{\gamma}_{1}\cdot4^{t_{\gamma}}\ge2\norm{\lambda^*}_1$. Thus no pair with
    $i>\guessTime$ is explored before the first valid pair is reached. Since
    level $i$ contains $i+1$ pairs, the number of visits to
    Line~\ref{enu:adjust_parameter1} is at most
    $\sum_{i=1}^{\guessTime}(i+1)=\guessTime(\guessTime+3)/2$. Finally,
    $0\le j\le i\le\guessTime$ yields
    $\hat{\mu}_{s}=\hat{\mu}_{1}/4^{j}\ge\hat{\mu}_{1}/4^{\guessTime}$ and
    $\hat{\gamma}_{s}=\hat{\gamma}_{1}\cdot4^{\,i-j}\le\hat{\gamma}_{1}\cdot4^{\guessTime}$.
    \ifdefined\isarxiv
\end{proof}
\else
\qedsymbol\end{proof}
\fi

\begin{proposition}
    \label{prop:certificate_generation}
    Whenever Algorithm~\ref{alg:rAPEXC} reaches
    Line~\ref{enu:store_certificate}, write the stored tuple as
    $\Cert=(\bar y^s,\iota_c,\nu_c,\gamma_c)$. Then there exists a point set
    $\mcal P_s$ forming an $(\iota_c,\nu_c,\gamma_c)$-$\PWcer$ for $\bar y^s$ in
    problem~\eqref{eq:initial_FP}. The stored parameters are as follows.
    \begin{enumerate}[label=(\roman*)]
        \item If $\lowerflag=\true$ and
              $[\mcal G(\bar y^s)]_+\leq2\theta\tilde{\Delta}_s/\hat{\gamma}_s$, then
              $
                  (\iota_c,\nu_c,\gamma_c)=
                  \brbra{\sqrt{\frac{2\theta\tilde{\Delta}_s}{\hat{\mu}_s}},
                      \sqrt{\frac{\hat{\mu}_s\theta\tilde{\Delta}_s}{2}},\hat{\gamma}_s}.
              $
        \item Otherwise, we have
              $
                  (\iota_c,\nu_c,\gamma_c)=
                  \brbra{\sqrt{\frac{2(1+\beta)\tilde{\Delta}_s}{\hat{\mu}_s}},
                      \sqrt{\frac{\hat{\mu}_s(1+\beta)\tilde{\Delta}_s}{2}},\hat{\gamma}_s}.
              $
    \end{enumerate}
\end{proposition}

\begin{proof}
    If the first case holds, $\pagr$ returned with $\lowerflag=\true$.
    The lower-bound trigger uses the level
    $\phi(\bar y^s;\hat{\gamma}_s)-\theta\tilde{\Delta}_s$ and radius
    $\sqrt{2\theta\tilde{\Delta}_s/\hat{\mu}_s}$. Together with the feasibility
    check in Line~\ref{enu:store_certificate}, the same certificate-generation argument
    as in Proposition~\red{6.2} of~\cite{partOne}, with $\Wcer$ replaced by
    $\PWcer$, gives the first certificate.

    In the complementary case, reaching Line~\ref{enu:store_certificate} means that the
    validation branch did not reject the current parameters. Hence the call to
    $\pawg$ at Line~\ref{enu:run_APWG} returned $\paramflag=\true$. Applying
    Proposition~\ref{prop:APWG_W_cer_gen_complexity} with
    $\Delta=\tilde{\Delta}_s$ and
    $\iota_{\max}=\sqrt{2(1+\beta)\tilde{\Delta}_s/\hat{\mu}_s}$ gives the second
    certificate. Algorithm~\ref{alg:rAPEXC} stores only the displayed radius,
    tolerance, and penalty parameter in $\Cert$; the point set $\mcal P_s$ is used
    only through its guaranteed existence.
    \ifdefined\isarxiv
\end{proof}
\else
\qedsymbol\end{proof}
\fi

The proof of Theorem~\ref{thm:convergence_CW_certificate} uses three
preliminary lemmas. Lemma~\ref{lem:tilde_Delta_new_proper} verifies the gap
estimate produced by $\adjustPara$ after a parameter update.
Lemma~\ref{lem:lower_bound_valid_alg} shows that Algorithm~\ref{alg:rAPEXC}
preserves the lower-bound invariant under valid parameter estimates.
Lemma~\ref{lem:upper_bound_transfer} bounds how much the gap estimate can grow
during certificate transfer.

\begin{lem}
    \label{lem:tilde_Delta_new_proper}
    Suppose that $\phi(\cdot;\hat{\gamma}_{\text{new}})$ satisfies the $\hat{\mu}_{\text{new}}$-QG condition,  $\hat{\gamma}_{\text{new}}\geq2\norm{\lambda^*}_1$ and there exists a point set $\mcal P_{c}$ that is an $\brbra{\iota_c, \nu_c, \gamma_c}$-$\PWcer$ for point $y_c$, where $\iota_c$, $\nu_c$, and $\gamma_c$ are the corresponding parameters in tuple $\Cert$. Then the update of $\tilde{\Delta}_{\text{new}}$ in Function~\ref{alg:parameter_adjust} satisfies
    \[
        \phi(\hat{y};\hat{\gamma}_{\text{new}})-f^{*}\leq\tilde{\Delta}_{\text{new}}.
    \]
\end{lem}

\begin{proof}
    The condition $\hat{\gamma}_{\text{new}}\geq 2\norm{\lambda^*}_1$ ensures $\min_{x\in X} \phi(x;\hat{\gamma}_{\text{new}}) = f^* $.
    Since $\phi(\cdot;\hat{\gamma}_{\text{new}})$ satisfies the $\hat{\mu}_{\text{new}}$-QG condition, we have
    \begin{equation}
        \begin{aligned}
             & \frac{\hat{\mu}_{\text{new}}}{2}\dist^{2}(\hat{y}, X)\leq\phi(\hat{y};\hat{\gamma}_{\text{new}})-\min_{x\in X}\phi(x;\hat{\gamma}_{\text{new}})                               \\
             & \leq \inner{\phi^{\prime}(\hat{y};\hat{\gamma}_{\text{new}})}{\hat{y}-x_{\gamma}^{*}}\leq\norm{\phi^{\prime}(\hat{y};\hat{\gamma}_{\text{new}})}\norm{\hat{y}-x_{\gamma}^{*}}
        \end{aligned}\ ,
    \end{equation}
    which implies
    \begin{equation}
        f^* = \min_{x\in X}\phi(x;\hat{\gamma}_{\text{new}})\geq\phi(\hat{y};\hat{\gamma}_{\text{new}})-\frac{2\norm{\phi^{\prime}(\hat{y};\hat{\gamma}_{\text{new}})}^{2}}{\hat{\mu}_{\text{new}}}\ .\label{eq:QG_imply_lower_bound_new}
    \end{equation}
    Hence, the second term of Line~\ref{enu:lower_bound_compute} in Function~\ref{alg:parameter_adjust} is a valid lower bound of $f^{*}$. On the other hand, applying Proposition~\ref{prop:CW_certificate_transfer} with the $\iota = \iota_c$, $\nu = \nu_c$, $\gamma = \gamma_c$ and $\tilde{\gamma} = \hat{\gamma}_{\text{new}}$ and Proposition~\ref{prop:QG_for_lower_bound} with $\mu = \hat{\mu}_{\text{new}}$ gives
    \begin{equation}
        \phi(y_c;\max\bcbra{\gamma_c, \hat{\gamma}_{\text{new}}}) - \max\bcbra{\iota_c \bar{\nu}, \frac{2\bar{\nu}^2}{\hat{\mu}_{\text{new}}}}\leq  f^*\ .
    \end{equation}
    Therefore, the first term of Line~\ref{enu:lower_bound_compute} in Function~\ref{alg:parameter_adjust} is also a valid lower bound of $f^{*}$.
    Therefore, this completes the proof of the lemma.
\end{proof}

\begin{lem}
    \label{lem:lower_bound_valid_alg}
    Suppose that $\phi\brbra{\cdot;\hat{\gamma}_{s}}$ satisfies the
    $\hat{\mu}_{s}$-QG condition and that $\hat{\gamma}_{s}\geq2\norm{\lambda^*}_1$.
    Then each update of the gap estimate in Algorithm~\ref{alg:rAPEXC} at
    Line~\ref{enu:adjust_parameter1} or at the unified refresh step
    Line~\ref{enu:store_certificate}
    preserves the invariant
    $\phi\brbra{\bar{y}^{s+1};\hat{\gamma}_{s+1}}-f^{*}\leq\tilde{\Delta}_{s+1}$.
\end{lem}

\begin{proof}
    We consider the parameter-adjustment step and the two cases of the unified
    refresh step, and show that the inequality
    $\phi\brbra{\bar{y}^{s};\hat{\gamma}_{s}}-f^{*}\leq\tilde{\Delta}_{s}$ is
    preserved.

    \textbf{Line}~\ref{enu:adjust_parameter1}: For $s\geq 2$, there exists a $\Cert$ as input of Function~\ref{alg:parameter_adjust}. Hence, by Lemma~\ref{lem:tilde_Delta_new_proper}, the new gap estimate $\tilde{\Delta}_{\text{new}}$ is a valid upper bound of $\phi(\hat{y};\hat{\gamma}_{\text{new}})-f^*$, which is equivalent to our claim that $\phi(\bar{y}^s;\hat{\gamma}_{s})-f^* \leq \tilde{\Delta}_{s}$ since $\bar{y}^s = \hat{y}$,  $\hat{\gamma}_s=\hat{\gamma}_{\text{new}}$ and $\tilde{\Delta}_s = \tilde{\Delta}_{\text{new}}$ in Function~\ref{alg:parameter_adjust}. For $s = 1$, we set the value $\nu = \infty$, which implies $\underline{\phi}$ must be chosen as the second term in Line~\ref{enu:lower_bound_compute} of Function~\ref{alg:parameter_adjust}, which is also a valid lower bound of $f^*$ by Lemma~\ref{lem:tilde_Delta_new_proper}.

    \textbf{Line~\ref{enu:store_certificate} with $\lowerflag=\true$:}
    The lower-bound trigger in $\pagr$ certifies the penalty level
    $\phi\brbra{\bar y^s;\hat{\gamma}_s}-\theta\tilde{\Delta}_s$. Hence
    \[
        \phi\brbra{\bar y^s;\hat{\gamma}_s}-\theta\tilde{\Delta}_s\leq
        \min_{x\in X}\phi\brbra{x;\hat{\gamma}_s} = f^*.
    \]
    Line~\ref{enu:store_certificate}
    sets $\bar y^{s+1}=\hat x^s$ and
    \[
        \tilde{\Delta}_{s+1}
        =\phi\brbra{\bar y^{s+1};\hat{\gamma}_s}
        -\brbra{\phi\brbra{\bar y^s;\hat{\gamma}_s}-\theta\tilde{\Delta}_s}.
    \]
    Therefore
    \[
        \phi\brbra{\bar y^{s+1};\hat{\gamma}_s}-f^*
        \leq \tilde{\Delta}_{s+1}.
    \]

    \textbf{Line~\ref{enu:store_certificate} with $\lowerflag=\false$:}
    By the invariant at the beginning of the refresh step, the current estimate satisfies
    $\phi(\bar y^s;\hat{\gamma}_s)-\tilde{\Delta}_s\leq f^*$. Line~\ref{enu:store_certificate}
    sets $\bar y^{s+1}=\hat x^s$ and
    \[
        \tilde{\Delta}_{s+1}
        =\phi\brbra{\bar y^{s+1};\hat{\gamma}_s}
        -\brbra{\phi\brbra{\bar y^s;\hat{\gamma}_s}-\tilde{\Delta}_s}.
    \]
    Hence
    \[
        \phi\brbra{\bar y^{s+1};\hat{\gamma}_s}-f^*
        \leq \tilde{\Delta}_{s+1}.
    \]
    Together with the assignment
    $\hat{\gamma}_{s+1}=\hat{\gamma}_s$, both refresh cases give the required gap
    bound for the next stage.
    \ifdefined\isarxiv
\end{proof}
\else
\qedsymbol\end{proof}
\fi

\begin{lem}
    \label{lem:upper_bound_transfer}
    Suppose that a point set $\mcal P$ forms an $\brbra{\iota_{c},\nu_{c},\gamma_{c}}$-$\PWcer$ for a point $y_c$ in the problem~\eqref{eq:initial_FP}. Let $\mu_c>0$ satisfy $\iota_{c}=\sqrt{2\Delta/\mu_{c}}$ and $\nu_{c}=\sqrt{\mu_{c}\Delta/2}$. If Function~\ref{alg:parameter_adjust} with $\Cert = \brbra{y_c, \iota_c, \nu_c, \gamma_c}$, it returns
    $\brbra{\hat{\gamma}_{\text{new}},\hat{\mu}_{\text{new}},\hat{y},\tilde{\Delta}_{\text{new}}}$ satisfying
    \begin{equation*}
        \tilde{\Delta}_{\text{new}}\leq \Delta\max\bcbra{2\eta+1,\brbra{1+4\eta+4\eta^{2}}\frac{\mu_{c}}{\hat{\mu}_{\text{new}}}},
        \qquad
        \eta:=\max\bcbra{\frac{\hat{\gamma}_{\text{new}}}{\gamma_{c}}-1,0}.
    \end{equation*}
\end{lem}

\begin{proof}
    Let $\hat{y}$ be the point selected at Line~\ref{enu:bar_y} of Function~\ref{alg:parameter_adjust}. By the definition of $\tilde{\Delta}$ at Line~\ref{enu:gap_parameter_adjust},
    \begin{equation}
        \tilde{\Delta}_{\text{new}}\leq \phi(\hat{y};\hat{\gamma}_{\text{new}})-\phi(y_c;\max\bcbra{\gamma_{c},\hat{\gamma}_{\text{new}}})
        +\max\bcbra{\iota_{c}\bar{\nu},\frac{2\bar{\nu}^{2}}{\hat{\mu}_{\text{new}}}}
        \leq \max\bcbra{\iota_{c}\bar{\nu},\frac{2\bar{\nu}^{2}}{\hat{\mu}_{\text{new}}}}\ ,
        \label{eq:upper_tilde_D}
    \end{equation}
    where the last inequality uses that $\phi(x;\gamma)$ is nondecreasing in $\gamma$.

    Next, let $\eta:=\max\bcbra{\frac{\hat{\gamma}_{\text{new}}}{\gamma_{c}}-1,0}$. Since $\iota_{c}{\nu}_{c}=\Delta$, we have
    \begin{equation}
        \iota_{c}\bar{\nu}
        =\iota_{c}\brbra{\max\bcbra{\hat{\gamma}_{\text{new}}-\gamma_{c},0}\cdot\frac{2\Delta}{\gamma_{c}\iota_{c}}+\nu_{c}}
        =\Delta(2\eta+1)\ .
        \label{eq:max_upper_1}
    \end{equation}
    Moreover, using $\iota_{c}=\sqrt{2\Delta/\mu_{c}}$ and $\nu_{c}=\sqrt{\mu_{c}\Delta/2}$,
    \begin{equation}
        \frac{2\bar{\nu}^{2}}{\hat{\mu}_{\text{new}}}
        =\frac{2}{\hat{\mu}_{\text{new}}}\brbra{\max\bcbra{\hat{\gamma}_{\text{new}}-\gamma_{c},0}\cdot\frac{2\Delta}{\gamma_{c}\iota_{c}}+\nu_{c}}^{2}
        =\Delta\brbra{1+4\eta+4\eta^{2}}\frac{\mu_{c}}{\hat{\mu}_{\text{new}}}\ .
        \label{eq:max_upper_2}
    \end{equation}
    Substituting~\eqref{eq:max_upper_1} and~\eqref{eq:max_upper_2} into~\eqref{eq:upper_tilde_D} gives the desired bound.
    \ifdefined\isarxiv
\end{proof}
\else
\qedsymbol\end{proof}
\fi

\begin{proof}[\underline{Proof of Theorem~\ref{thm:convergence_CW_certificate}}]
    We index every parameter pair visited by Algorithm~\ref{alg:rAPEXC}
    (including the initial pair) as
    \[
        (\hat{\mu}_s,\hat{\gamma}_s)
        =\brbra{\hat{\mu}_1/4^j,\hat{\gamma}_1\cdot4^{i-j}},
        \qquad 0\leq j\leq i\leq \guessTime .
    \]
    Throughout the proof, write $L_{\text{avg}}=L_{\text{avg}}(\vep)$
    (defined in~\eqref{eq:def_L_avg}). For a transition from one indexed pair to the
    next, let $\tilde{\Delta}_{(1)}^{(j,i)}$ and
    $\tilde{\Delta}_{(2)}^{(j,i)}$ denote the values of $\tilde{\Delta}_s$
    immediately before and after the corresponding call to $\adjustPara$.
    We call an indexed pair certified if the algorithm has recorded certificate
    parameters for that pair and the corresponding $\PWcer$ is guaranteed to
    exist. For each fixed level $i$, define
    \[
        J_i:=\bcbra{j\in\{0,\ldots,i\}:
            \brbra{\hat{\gamma}_{(2)}^{(j,i)},\hat{\mu}_{(2)}^{(j,i)}}
            \text{ is certified}},
    \]
    let $J_i(q)$ be the $q$-th smallest element of $J_i$, and set
    $\bar j_i:=\max J_i$ when $J_i\neq\varnothing$. In a worst-case oracle
    count, we may assume that ${\pawg}$ runs to completion for every
    explored pair $(j,i)$ without returning at Line~\ref{enu:false_direct} of Algorithm~\ref{alg:APWG}.

    \textbf{Sufficient condition for the target certificate.}
    By Proposition~\ref{prop:certificate_generation}, whenever Algorithm~\ref{alg:rAPEXC}
    reaches Line~\ref{enu:store_certificate}, the parameters stored in $\Cert$ admit a
    point set that is a $\PWcer$ for
    $\bar y^s$ in problem~\eqref{eq:initial_FP}. Since $\theta<1+\beta$, the
    case distinction in Line~\ref{enu:store_certificate} shows that the stored
    $\iota_c$ is at most
    $\sqrt{2(1+\beta)\tilde{\Delta}_s/\hat{\mu}_s}$ and the stored $\nu_c$ is at
    most $\sqrt{\hat{\mu}_s(1+\beta)\tilde{\Delta}_s/2}$.
    Hence this certificate is an
    $\brbra{\vep/\hat{\mu}_s,\vep,\hat{\gamma}_s}$-$\PWcer$ whenever
    \begin{equation}\label{eq:suff_eps}
        \tilde{\Delta}_s\leq\frac{\vep^2}{2\hat{\mu}_s(1+\beta)}.
    \end{equation}

    \textbf{Cost of one call to $\pagr$ or ${\pawg}$.}
    Consider a call $\mcal C$ made with indexed pair $(j,i)$, so
    $\hat{\gamma}_{\mcal C}=\hat{\gamma}_1 4^{i-j}$. By the definition of
    $L_{\text{avg}}$, its weighted mean satisfies
    $\bar L_{\hat{\gamma}_{\mcal C}}^{\mcal C}(N_{\mcal C})
        \leq 4^{i-j}L_{\text{avg}}$. Since the QG estimate in this call is
    $\hat{\mu}_1/4^j$,
    Propositions~\ref{prop:convergence_rapex_gap_reduction}
    and~\ref{prop:APWG_W_cer_gen_complexity} imply the single-call bound
    \begin{equation}\label{eq:single_call_cost}
        O(1)\,\bundleSize\max\bcbra{1,
            \sqrt{\frac{4^j\bar L_{\hat{\gamma}_{\mcal C}}^{\mcal C}(N_{\mcal C})}{\hat{\mu}_1}}}
        \leq O(1)\,\bundleSize\max\bcbra{1,
        \sqrt{\frac{4^iL_{\text{avg}}}{\hat{\mu}_1}}} .
    \end{equation}
    Denote the final upper bound in~\eqref{eq:single_call_cost} by $\kappa(j,i)$.
    Define
    \begin{equation}\label{eq:kappa_defn}
        \kappa_{+}:=\bundleSize\,2^{\guessTime}
        \sqrt{\frac{L_{\text{avg}}}{\hat{\mu}_{1}}}.
    \end{equation}
    Then $\kappa(j,i)\leq O(1)2^{i-\guessTime}\kappa_{+}$ for every explored
    pair $(j,i)$.

    We decompose the remaining oracle count into four terms, as illustrated in
    Figure~\ref{fig:function-constrained-complexity}.

    \textbf{Pairs that have not produced certificates $(S_1(\vep))$.}
    For stages that end with a parameter adjustment, summing
    \eqref{eq:single_call_cost} over the explored pairs gives
    \begin{equation}
        \label{eq:S1_vep}
        S_1(\vep)\leq \sum_{i=0}^{\guessTime}\sum_{j=0}^{i}\kappa(j,i)
        \leq O(1)\kappa_{+}\sum_{i=0}^{\guessTime}(i+1)2^{i-\guessTime}
        \leq O(1)\,\guessTime \kappa_{+}.
    \end{equation}

    \textbf{Initial phase before the first parameter update $(S_2(\vep))$.}
    Before the first call to $\adjustPara$, Algorithm~\ref{alg:rAPEXC} uses
    $(\hat{\gamma}_1,\hat{\mu}_1)$ and
    $\tilde{\Delta}_1=2\norm{\phi^{\prime}(\bar y^0;\hat{\gamma}_1)}^2/\hat{\mu}_1$.
    If the sufficient condition~\eqref{eq:suff_eps} has not yet been met, then
    \begin{equation}
        \label{eq:S2eps_rone}
        S_2(\vep)\leq O(1)\kappa_{+}
        \brbra{1+\log_{\tfrac{1}{\theta}}
            \frac{\norm{\phi^{\prime}(\bar y^0;\hat{\gamma}_1)}^2}{\vep^2}} .
    \end{equation}

    \textbf{Final fixed-parameter phase $(S_3(\vep))$.}
    Let $(j^*,i^*)$ be the last pair returned by $\adjustPara$ before
    \eqref{eq:suff_eps} holds. After this update, the parameter pair is fixed
    until the target certificate is obtained. The standard restart accounting for
    successful visits to Line~\ref{enu:store_certificate},
    together with the gradient-bound term in
    Line~\ref{enu:lower_bound_compute} of $\adjustPara$,
    gives
    \begin{equation}\label{eq:S3_upper}
        S_3(\vep)\leq O(1)\kappa_{+}
        \brbra{1+\log_{\tfrac{1}{\theta}}
            \frac{\max_{r=0,\ldots,\guessTime}
                \norm{\phi^{\prime}\brbra{\bar y^0;\hat{\gamma}_1\cdot4^r}}^2}{\vep^2}} .
    \end{equation}

    \textbf{Certified intermediate pairs $(S_4(\vep))$.}
    For a fixed level $i$, if $J_i=\varnothing$, then level $i$ contributes
    nothing to $S_4(\vep)$. Otherwise, define the pair after the last certified
    one by
    \[
        \brbra{j_i^+,i_i^+}:=
        \begin{cases}
            \brbra{\bar{j}_i+1,i}, & \bar{j}_i<i, \\
            \brbra{0,i+1},         & \bar{j}_i=i.
        \end{cases}
    \]
    For each fixed $i$ with $J_i\neq\varnothing$, the same transition accounting
    gives
    \[
        S_4^i(\vep)\leq O(1)\cdot 2^{i-\mcal T}\kappa_+\brbra{
        \log_{\tfrac{1}{\theta}}
        \frac{\tilde{\Delta}_{(2)}^{\brbra{J_i(1),i}}}
        {\tilde{\Delta}_{(1)}^{\brbra{j_i^+,i_i^+}}}+i}.
    \]
    Indeed, for consecutive certified indices at the same level, the parameters
    in Lemma~\ref{lem:upper_bound_transfer} satisfy
    \[
        \eta=\max\bcbra{4^{J_i(q)-J_i(q+1)}-1,0}=0.
    \]
    Thus Lemma~\ref{lem:upper_bound_transfer} gives
    \[
        \frac{\tilde{\Delta}_{(2)}^{(J_i(q+1),i)}}
        {\tilde{\Delta}_{(1)}^{(J_i(q+1),i)}}
        \leq4^{J_i(q+1)-J_i(q)} ,
    \]
    and these factors telescope to at most $4^i$. The remaining logarithm is
    bounded by the gradient-bound term in Line~\ref{enu:lower_bound_compute} and by the fact that
    \eqref{eq:suff_eps} has not yet been reached; hence
    \[
        \log_{\tfrac{1}{\theta}}
        \frac{\tilde{\Delta}_{(2)}^{\brbra{J_i(1),i}}}
        {\tilde{\Delta}_{(1)}^{\brbra{j_i^+,i_i^+}}}
        \leq O(1)\brbra{
            i+\log_{\tfrac{1}{\theta}}
            \frac{\max_{r=0,\ldots,\guessTime}
                \norm{\phi^{\prime}\brbra{\bar y^0;\hat{\gamma}_1\cdot4^r}}^2}{\vep^2}} .
    \]
    Summing over $i$ yields
    \begin{equation}\label{eq:final_upper_bound_S_4}
        S_4(\vep)\leq O(1)\kappa_{+}
        \brbra{\guessTime+\log_{\tfrac{1}{\theta}}
            \frac{\max_{r=0,\ldots,\guessTime}
                \norm{\phi^{\prime}\brbra{\bar y^0;\hat{\gamma}_1\cdot4^r}}^2}{\vep^2}} ,
    \end{equation}
    where the inequality holds by
    $\max_{j:(j,i)\in\mcal I(\vep)}\kappa(j,i)\leq
        O(1)2^{i-\guessTime}\kappa_{+}$ and the summability of
    $\sum_{i=0}^{\guessTime}2^{i-\guessTime}$ and
    $\sum_{i=0}^{\guessTime}i2^{i-\guessTime}$.

    The total number of first-order oracle calls before obtaining the target
    certificate is at most
    $S_1(\vep)+S_2(\vep)+S_3(\vep)+S_4(\vep)$. Finally,
    \begin{equation}\label{eq:norm_upper_bound}
        \max_{r=0,\ldots,\guessTime}
        \norm{\phi^{\prime}\brbra{\bar y^0;\hat{\gamma}_1\cdot4^r}}^2
        \leq 2\cdot4^{2\guessTime}\brbra{
            \norm{f^{\prime}\brbra{\bar y^0}}^2+\hat{\gamma}_1^2\norm{\mcal G^{\prime}\brbra{\bar y^0}}^2}.
    \end{equation}
    The factor $4^{2\guessTime}$ in~\eqref{eq:norm_upper_bound} contributes only
    an additional $O(1)\guessTime$ term inside the logarithm. Combining
    \eqref{eq:S1_vep}, \eqref{eq:S2eps_rone}, \eqref{eq:S3_upper},
    \eqref{eq:final_upper_bound_S_4}, and~\eqref{eq:kappa_defn} proves
    \eqref{eq:complexity_rapexc_lavg}.

    If $\bundleSize\geq k_f + k_{\mcal G}$, then
    Proposition~\ref{prop:upper_bound_Lip} gives, for any call $\mcal C$ using pair
    $(j,i)$,
    \[
        \bar L_{\hat{\gamma}_{\mcal C}}^{\mcal C}(N_{\mcal C})
        \leq O(1)(L_f+\hat{\gamma}_1 4^{i-j}L_{\mcal G}).
    \]
    Multiplying by the rescaling factor gives
    \[
        4^{j-i}\bar L_{\hat{\gamma}_{\mcal C}}^{\mcal C}(N_{\mcal C})
        \leq O(1)(4^{j-i}L_f+\hat{\gamma}_1L_{\mcal G}).
    \]
    Since $j\leq i$, we have $4^{j-i}\leq1$, and therefore
    $L_{\text{avg}}\leq O(1)(L_f+\hat{\gamma}_1L_{\mcal G})$ by
    \eqref{eq:def_L_avg}. Substituting
    this bound into~\eqref{eq:complexity_rapexc_lavg} gives
    \eqref{eq:complexity_rapexc}.
    \ifdefined\isarxiv
\end{proof}
\else
\qedsymbol\end{proof}
\fi

\begin{proof}[\underline{Proof of Theorem~\ref{thm:convergence_opt_convio}}]
    Since $\norm{\lambda^*}_1\geq1$, the initialization $\hat{\gamma}_{1}=2$ satisfies
    $\hat{\gamma}_{1}\leq2\norm{\lambda^*}_1$. Together with $\hat{\mu}_{1}\geq\mu^{*}$,
    Proposition~\ref{prop:flag_false} gives
    \[
        \guessTime
        =\left\lceil \log_{4}\frac{\hat{\mu}_{1}}{\mu^{*}}\right\rceil
        +\left\lceil \log_{4}\norm{\lambda^*}_1\right\rceil,
        \qquad
        4^{\guessTime}\le \frac{16\hat{\mu}_{1}\norm{\lambda^*}_1}{\mu^{*}},
    \]
    where the second bound uses $4^{\lceil a\rceil}\le 4\cdot 4^a$.

    Set
    $\delta:=\sqrt{\mu^{*}\vep/\brbra{16\norm{\lambda^*}_1(2\norm{\lambda^*}_1-1)}}$. We first show that any
    $\brbra{\delta/\hat{\mu}_{s},\delta,\hat{\gamma}_{s}}$-$\PWcer$ for
    $\bar{y}^{s}$ already implies~\eqref{eq:epsilon_condition}. By
    Proposition~\ref{prop:CW_certificate_transfer} with
    $\tilde{\gamma}=2\norm{\lambda^*}_1$, this certificate also gives a
    $\brbra{\delta/\hat{\mu}_{s},\bar{\nu},
            \max\bcbra{2\norm{\lambda^*}_1,\hat{\gamma}_{s}}}$-$\PWcer$ for $\bar y^s$, where
    \[
        \bar{\nu}:=\max\bcbra{2\norm{\lambda^*}_1-\hat{\gamma}_{s},0}
        \cdot\frac{2\delta}{\hat{\gamma}_{s}}+\delta .
    \]
    Applying Proposition~\ref{prop:QG_for_lower_bound} with $\mu=\mu^*$ gives
    \[
        \phi\brbra{\bar{y}^{s};\max\bcbra{2\norm{\lambda^*}_1,\hat{\gamma}_{s}}}-f^{*}
        \leq \max\Bcbra{\frac{\delta\bar{\nu}}{\hat{\mu}_{s}},\frac{2\bar{\nu}^{2}}{\mu^{*}}}.
    \]
    Since $\hat{\gamma}_{1}=2$ and every update multiplies $\hat{\gamma}_{s}$ by
    $4$, we have $\hat{\gamma}_{s}\ge2$. Hence, if
    $\hat{\gamma}_{s}\ge2\norm{\lambda^*}_1$, then $\bar{\nu}=\delta$; otherwise
    $2\le\hat{\gamma}_{s}<2\norm{\lambda^*}_1$, so
    $\bar{\nu}=\delta+2\brbra{2\norm{\lambda^*}_1-\hat{\gamma}_{s}}\delta/\hat{\gamma}_{s}
        =\brbra{4\norm{\lambda^*}_1/\hat{\gamma}_{s}-1}\delta\le(2\norm{\lambda^*}_1-1)\delta$. Thus
    \begin{equation}\label{eq:bar_nu_bound}
        \bar{\nu}\le(2\norm{\lambda^*}_1-1)\delta.
    \end{equation}
    Moreover, Proposition~\ref{prop:flag_false} and the bound on $4^{\guessTime}$ imply
    $\hat{\mu}_{s}\ge\hat{\mu}_{1}/4^{\guessTime}\ge\mu^{*}/(16\norm{\lambda^*}_1)$.
    Combining this with~\eqref{eq:bar_nu_bound}, we obtain
    \[
        \frac{\delta\bar{\nu}}{\hat{\mu}_{s}}
        \le \frac{16\norm{\lambda^*}_1(2\norm{\lambda^*}_1-1)\delta^{2}}{\mu^{*}}
        =\vep
        \quad\text{and}\quad
        \frac{2\bar{\nu}^{2}}{\mu^{*}}
        \le \frac{2(2\norm{\lambda^*}_1-1)^{2}\delta^{2}}{\mu^{*}}
        =\frac{2\norm{\lambda^*}_1-1}{8\norm{\lambda^*}_1}\vep
        \le \vep.
    \]
    Therefore,
    $f\brbra{\bar{y}^{s}}-f^{*}
        \le\phi\brbra{\bar{y}^{s};\max\bcbra{2\norm{\lambda^*}_1,\hat{\gamma}_{s}}}-f^{*}\le\vep$.
    Also, the definition of $\PWcer$ gives
    \[
        \bsbra{\mcal G\brbra{\bar{y}^{s}}}_{+}
        \le \frac{2\delta^{2}}{\hat{\mu}_{s}\hat{\gamma}_{s}}
        \le \frac{16\norm{\lambda^*}_1\delta^{2}}{\mu^{*}}
        =\frac{\vep}{2\norm{\lambda^*}_1-1}
        \le \vep.
    \]
    Thus every $\brbra{\delta/\hat{\mu}_{s},\delta,\hat{\gamma}_{s}}$-$\PWcer$ yields a point
    satisfying~\eqref{eq:epsilon_condition}.

    Now apply Theorem~\ref{thm:convergence_CW_certificate} with target accuracy
    $\delta$, $\hat{\gamma}_{1}=2$, and $L_{\text{avg}}=L_{\text{avg}}(\delta)$.
    If the initial pair is already valid, the same estimate follows from the
    no-adjustment case of the proof of Theorem~\ref{thm:convergence_CW_certificate}.
    After at most
    \[
        O(1)\,\bundleSize\,2^{\guessTime}\sqrt{\frac{L_{\text{avg}}}{\hat{\mu}_{1}}}
        \brbra{\guessTime+\bsbra{\log_{\tfrac{1}{\theta}}
                \frac{\norm{f^{\prime}\brbra{\bar{y}^{0}}}^{2}+\norm{\mcal G^{\prime}\brbra{\bar{y}^{0}}}^{2}}{\delta^{2}}}_+}
    \]
    first-order oracle calls, the algorithm constructs the corresponding certificate
    for some $\bar{y}^{s}$. Finally, the bound on $4^{\guessTime}$ implies
    \[
        2^{\guessTime}\sqrt{\frac{L_{\text{avg}}}{\hat{\mu}_{1}}}
        =\sqrt{\frac{4^{\guessTime}L_{\text{avg}}}{\hat{\mu}_{1}}}
        \le 4\sqrt{\frac{\norm{\lambda^*}_1L_{\text{avg}}}{\mu^{*}}}.
    \]
    Using also $\delta^{-2}=16\norm{\lambda^*}_1(2\norm{\lambda^*}_1-1)/(\mu^{*}\vep)$ and
    absorbing additional constants into $O(1)\mcal{T}$, we obtain the complexity
    bound~\eqref{eq:num_grad_oracle_lavg}. If
    $\bundleSize\geq k_f + k_{\mcal G}$, then the final part of
    Theorem~\ref{thm:convergence_CW_certificate} gives
    $L_{\text{avg}}\leq O(1)(L_f+2L_{\mcal G})$, which yields
    \eqref{eq:num_grad_oracle}.
    \ifdefined\isarxiv
\end{proof}
\else
\qedsymbol\end{proof}
\fi

\begin{proof}[\underline{Proof of Corollary~\ref{cor:convergence_opt_convio_small_mu}}]
    The proof follows that of Theorem~\ref{thm:convergence_opt_convio}. Since $\norm{\lambda^*}_1\geq1$ and $\hat{\gamma}_{1}=2$, Proposition~\ref{prop:flag_false} gives
    \[
        \guessTime=\left\lceil \log_{4}\norm{\lambda^*}_1\right\rceil,
        \qquad
        4^{\guessTime}\le 4\norm{\lambda^*}_1.
    \]
    Set $\delta:=\sqrt{\hat{\mu}_{1}\vep/\brbra{4\norm{\lambda^*}_1(2\norm{\lambda^*}_1-1)}}$. As in the proof of Theorem~\ref{thm:convergence_opt_convio}, any $\brbra{\delta/\hat{\mu}_{s},\delta,\hat{\gamma}_{s}}$-$\PWcer$ for $\bar{y}^{s}$ satisfies
    \[
        \phi\brbra{\bar{y}^{s};\max\bcbra{2\norm{\lambda^*}_1,\hat{\gamma}_{s}}}-f^{*}
        \le \max\Bcbra{\frac{\delta\bar{\nu}}{\hat{\mu}_{s}},\frac{2\bar{\nu}^{2}}{\mu^{*}}},
        \qquad
        \bar{\nu}\le(2\norm{\lambda^*}_1-1)\delta.
    \]
    The only change is that now
    \[
        \hat{\mu}_{s}\ge \frac{\hat{\mu}_{1}}{4^{\guessTime}}
        \ge \frac{\hat{\mu}_{1}}{4\norm{\lambda^*}_1}.
    \]
    Hence
    \[
        \frac{\delta\bar{\nu}}{\hat{\mu}_{s}}\le \frac{4\norm{\lambda^*}_1(2\norm{\lambda^*}_1-1)\delta^{2}}{\hat{\mu}_{1}}=\vep,
        \qquad
        \frac{2\bar{\nu}^{2}}{\mu^{*}}
        \le \frac{2(2\norm{\lambda^*}_1-1)^{2}\delta^{2}}{\mu^{*}}
        =\frac{(2\norm{\lambda^*}_1-1)\hat{\mu}_{1}}{2\norm{\lambda^*}_1\mu^{*}}\vep
        \le \vep,
    \]
    because $\hat{\mu}_{1}<\mu^{*}$. Therefore
    $f\brbra{\bar{y}^{s}}-f^{*}\le\vep$. Also,
    \[
        \bsbra{\mcal G\brbra{\bar{y}^{s}}}_{+}
        \le \frac{2\delta^{2}}{\hat{\mu}_{s}\hat{\gamma}_{s}}
        \le \frac{4\norm{\lambda^*}_1\delta^{2}}{\hat{\mu}_{1}}
        =\frac{\vep}{2\norm{\lambda^*}_1-1}
        \le \vep.
    \]
    Thus every $\brbra{\delta/\hat{\mu}_{s},\delta,\hat{\gamma}_{s}}$-$\PWcer$ yields a point satisfying the desired bounds.

    Applying Theorem~\ref{thm:convergence_CW_certificate} with target accuracy $\delta$ and $\hat{\gamma}_{1}=2$ shows that the algorithm obtains such a certificate after at most
    \[
        O(1)\,\bundleSize\,2^{\guessTime}\sqrt{\frac{L_f+2L_{\mcal G}}{\hat{\mu}_{1}}}
        \brbra{\guessTime+\bsbra{\log_{\tfrac{1}{\theta}}
                \frac{\norm{f^{\prime}\brbra{\bar{y}^{0}}}^{2}+\norm{\mcal G^{\prime}\brbra{\bar{y}^{0}}}^{2}}{\delta^{2}}}_+}
    \]
    first-order oracle calls. Using
    $2^{\guessTime}\le 2\sqrt{\norm{\lambda^*}_1}$ and
    $\delta^{-2}=4\norm{\lambda^*}_1(2\norm{\lambda^*}_1-1)/(\hat{\mu}_{1}\vep)$, and absorbing
    absolute constants into $O(1)$, we obtain~\eqref{eq:num_grad_oracle_small_mu}.
\end{proof}

\section{Numerical Study\label{sec:numerical_study}}
In this section, we evaluate our algorithms on three classes of function-constrained problems: convex QCQPs (Section~\ref{subsec:qcqp}), binary Neyman-Pearson classification (Section~\ref{subsec:binary_np}), and fairness-aware classification (Section~\ref{subsec:fairness}). Projection QP subproblems are solved with MOSEK~\cite{mosek}, with Clarabel~\cite{Clarabel_2024} used as a backup. 
Unless otherwise stated, external reference values are computed with MOSEK for Convex QCQP and with Ipopt~\cite{wachter2006implementation} for the other problems.
All experiments were conducted on a {Mac mini M2 Pro with 32 GB} of
RAM.

Before presenting the experiments, we introduce the baseline used for comparison and basic implementation details for Restarted $\papex$.

\textbf{RLS description.} To the best of our knowledge, Restarting Level Set method (RLS)~\cite{lin2025adaptiveparameterfreeprojectionfreerestarting} is the closest existing parameter-free method for function-constrained optimization that exploits an error-bound condition to obtain a better convergence rate. 
RLS uses two restart parameters, $\alpha$ and $B$. In all experiments, we set
$\alpha=0.5$ and $B=0.9$\footnote{Here, $B$ denotes the RLS restart parameter;
elsewhere, it denotes the bundle size}. Since RLS requires a computable strictly
feasible point, we obtain a Slater point $\tilde{x}$ by solving
$\min_{x\in X}\mcal G(x)$, using MOSEK for the convex QCQPs and Ipopt for the
classification problems.
RLS also requires a lower bound $r<f^*$ on the optimal value. For the convex
QCQPs, we set $r=\min_{x\in X}f(x)$; for the classification problems, we follow
the recommendation in~\cite{lin2025adaptiveparameterfreeprojectionfreerestarting}
and set $r=0$.
RLS is parallel and requires specifying the number of threads. Given the target tolerance $\vep$, the number of threads is recommended to be  $\tilde{K}+1$, where 
\begin{equation}\label{eq:K_tilde}
    \tilde{K}
=\left\lceil
\frac{r-\brbra{f(\tilde{x})-\mcal G(\tilde{x})}}{\alpha \mcal G(\tilde{x})}
\max\bcbra{
\log\brbra{\frac{f(\tilde{x})-\mcal G(\tilde{x})-r}{\alpha\vep}},
1}
\right\rceil,
\qquad r<f^* .
\end{equation}
For smooth problems, we solve the level-set subproblem using the accelerated prox-linear method with backtracking line search~\cite{drusvyatskiy2019efficiency}; for nonsmooth problems, we use the subgradient method. These choices follow the implementation in~\cite{lin2025adaptiveparameterfreeprojectionfreerestarting}.

\textbf{General setting for Penalty $\apex$.} Unless otherwise specified, we set $\theta = 0.75$, $\beta = 1.0$, $\bundleSize=50$, and $\hat{\gamma}_{1}=2$ in Algorithm~\ref{alg:rAPEXC}. 
{We implement $\onestepplus$ with two \uline{early-stopping} gates inside the $\onestepplus$. The first gate is certificate-driven: after each projection subproblem, the routine compares the objective value with the prescribed certificate radius. If the relevant radius threshold is crossed, the current call stops and returns the certificate flag without generating the remaining cuts. 
The second gate is exactly the complement of the slow-descent event~\eqref{eq:slow_event}. Once it passes, the outer iteration has obtained the descent estimate needed in the proof, so the implementation terminates the current $\onestepplus$ call and avoids the unused oracle evaluations. Thus, the numerical code keeps the same radius and descent tests as the analysis, but uses them as online stopping rules within each bundle.}

\subsection{Convex Quadratically Constrained Quadratic Programming\label{subsec:qcqp}}
We focus on the following convex Quadratically Constrained Quadratic Programming (QCQP) problem:
\begin{equation}\label{eq:convex_qcqp}
\min_{x \in \mbb R^n} f(x):=\frac{1}{2} x^{\top}Q_0 x + c_0^{\top}x \ \ \st \ g_i(x):=C\cdot \brbra{\frac{1}{2} x^{\top}Q_i x+c_i^{\top}x+d_i}\leq 0, \ i = 1, \ldots, m,
\end{equation}
where $C>0$ is a constraint-scaling constant, $Q_i\succeq 0$ for $i=0,\ldots,m$, and $c_i\in\mbb R^n$ ($i=0,\ldots,m$) and $d_i\in\mbb R$ ($i=1,\ldots,m$) are sampled independently from normal distributions.  

\paragraph{Performance of Restarted {\papex} on Convex QCQP. }
\ \ 

\textbf{Experiment setting for small-scale convex QCQP.}
We evaluate Restarted $\papex$ on small-scale convex QCQP instances while varying the condition number $\kappa^+ = \tfrac{L_f + 2\gamma^+ L_{\mcal G}}{\mu_f + 2\gamma^+ \mu_{\mcal G}}$ with $\gamma^+ = \max\bcbra{\norm{\lambda^*}_1, 1}$\footnote{Assumption~\ref{assu:dual_exist} requires $\norm{\lambda^*}_1 \geq 1$. Hence, we focus on penalty value $\gamma^+$ in the following  experiments}, the number of constraints $m$, and the initial QG estimate $\hat{\mu}_1$. Here, $\mu_f$, $\mu_{\mcal G}$, $L_f$, and $L_{\mcal G}$ denote the strong convexity and Lipschitz constants of the objective and constraint functions. In all runs, we set $n=100$ and $\bundleSize=50$. Each $Q_i$ has minimum eigenvalue $\mu=1$, and its maximum eigenvalue is chosen from $L\in\bcbra{5,10,100,1000}$. The eigenvalues of each $Q_i$ are evenly spaced in an arithmetic sequence from $\mu$ to $L$. For small-scale instances, each $Q_i$ is generated with eigenvalues evenly spaced from $\mu$ to $L$ and with distinct random eigenvectors. We set $C=0.01$ in~\eqref{eq:convex_qcqp} to rescale the constraints and control the scale of $\norm{\lambda^*}_1$. We terminate Restarted $\papex$ when it generates a certificate ($\Cert = \brbra{y_c, \iota_c, \nu_c, \gamma_c}$) with $\iota_c \nu_c < 10^{-6}$ (The parameter choice makes sure $\iota_c \nu_c = 2\nu_c^2 / \hat{\mu}_S$) and summarize the results in Table~\ref{tab:small-scale_problems}.

\begin{table}[htbp]
  \centering
  \caption{Comparison of Restarted $\papex$ on convex QCQP with different initial QG estimates $\hat{\mu}_1$. The table reports the numerically active constraint count at optimal solution $\abs{\mcal A^*}=\abs{\{i\in[m]: \abs{g_i(x^*)}\leq 10^{-6}\}}$, empirical condition number $\hat{\kappa} = \frac{L_f + 2\hat{\gamma}_S L_{\mcal G}}{\mu_f + 2\hat{\gamma}_S \mu_{\mcal G}}$, final estimates $\hat{\gamma}_S$ and $\hat{\mu}_S$, oracle calls, residual $\max\bcbra{f(x)-f^*,\bsbra{\mcal G(x)}_+}$, the early-stopping ratio = the number of early stops $\onestepplus$ divided by the total number of $\onestepplus$.\label{tab:small-scale_problems}}
  {
  \setlength{\tabcolsep}{2.8pt}
    \begin{tabular}{crrrrrrrrrrr}
    \toprule
    $m$     & \multicolumn{1}{c}{$\kappa^+$} & \multicolumn{1}{c}{$\norm{\lambda^*}_1$} & \multicolumn{1}{c}{$\abs{\mcal A^*}$} & \multicolumn{1}{c}{$\hat{\kappa}$} & \multicolumn{1}{c}{$\hat{\mu}_1$} & \multicolumn{1}{c}{$\hat{\gamma}_S$} & \multicolumn{1}{l}{$\hat{\mu}_S$} & \multicolumn{1}{c}{\begin{tabular}{@{}c@{}}First-order\\oracle calls\end{tabular}} & \multicolumn{1}{c}{$\max\left\{\begin{array}{@{}c@{}} f(x)-f^*,\\ \left[\mcal G(x)\right]_+ \end{array}\right\}$} & \multicolumn{1}{c}{\begin{tabular}{@{}c@{}}Early\\stopping\\ratio\end{tabular}} & \multicolumn{1}{c}{\begin{tabular}{@{}c@{}}Oracle calls\\per\\ $\onestepplus$\end{tabular}} \\
    \midrule
\multirow{16}{*}{$50$} & \multirow{4}{*}{$5$} & \multirow{4}{*}{$16.80$} & \multirow{4}{*}{$3$} & $6.60$ & $1$ & $32$ & $1$ & $323$ & $1.75\mathrm{E}{-08}$ & $100\%$ & $2.58$ \\
 &  &  &  & $0.66$ & $10$ & $32$ & $10$ & $347$ & $3.71\mathrm{E}{-09}$ & $100\%$ & $3.21$ \\
 &  &  &  & $1.06$ & $100$ & $32$ & $6.25$ & $360$ & $1.21\mathrm{E}{-08}$ & $100\%$ & $2.98$ \\
 &  &  &  & $0.42$ & $1000$ & $32$ & $15.625$ & $250$ & $3.61\mathrm{E}{-09}$ & $100\%$ & $3.13$ \\
\cmidrule(lr){2-12}
 & \multirow{4}{*}{$10$} & \multirow{4}{*}{$7.60$} & \multirow{4}{*}{$3$} & $10.80$ & $1$ & $8$ & $1$ & $402$ & $9.02\mathrm{E}{-08}$ & $100\%$ & $5.09$ \\
 &  &  &  & $1.08$ & $10$ & $8$ & $10$ & $438$ & $6.80\mathrm{E}{-08}$ & $100\%$ & $7.30$ \\
 &  &  &  & $0.43$ & $100$ & $8$ & $25$ & $401$ & $1.46\mathrm{E}{-07}$ & $100\%$ & $6.17$ \\
 &  &  &  & $0.69$ & $1000$ & $8$ & $15.625$ & $408$ & $9.98\mathrm{E}{-08}$ & $100\%$ & $5.51$ \\
\cmidrule(lr){2-12}
 & \multirow{4}{*}{$100$} & \multirow{4}{*}{$0.61$} & \multirow{4}{*}{$3$} & $102.00$ & $1$ & $2$ & $1$ & $999$ & $2.42\mathrm{E}{-07}$ & $100\%$ & $15.61$ \\
 &  &  &  & $10.20$ & $10$ & $2$ & $10$ & $943$ & $6.62\mathrm{E}{-08}$ & $100\%$ & $17.79$ \\
 &  &  &  & $1.02$ & $100$ & $2$ & $100$ & $885$ & $9.61\mathrm{E}{-08}$ & $100\%$ & $20.58$ \\
 &  &  &  & $1.63$ & $1000$ & $2$ & $62.5$ & $885$ & $1.48\mathrm{E}{-07}$ & $100\%$ & $17.02$ \\
\cmidrule(lr){2-12}
 & \multirow{4}{*}{$1000$} & \multirow{4}{*}{$0.06$} & \multirow{4}{*}{$5$} & $1020.00$ & $1$ & $2$ & $1$ & $1408$ & $3.46\mathrm{E}{-07}$ & $100\%$ & $19.56$ \\
 &  &  &  & $102.00$ & $10$ & $2$ & $10$ & $1463$ & $1.94\mathrm{E}{-07}$ & $100\%$ & $22.51$ \\
 &  &  &  & $10.20$ & $100$ & $2$ & $100$ & $1349$ & $2.30\mathrm{E}{-07}$ & $100\%$ & $24.53$ \\
 &  &  &  & $1.02$ & $1000$ & $2$ & $1000$ & $1281$ & $3.61\mathrm{E}{-07}$ & $100\%$ & $27.26$ \\
\midrule
\multirow{16}{*}{$200$} & \multirow{4}{*}{$5$} & \multirow{4}{*}{$17.50$} & \multirow{4}{*}{$4$} & $6.60$ & $1$ & $32$ & $1$ & $381$ & $1.56\mathrm{E}{-09}$ & $100\%$ & $2.95$ \\
 &  &  &  & $0.66$ & $10$ & $32$ & $10$ & $379$ & $1.37\mathrm{E}{-07}$ & $100\%$ & $3.68$ \\
 &  &  &  & $1.06$ & $100$ & $32$ & $6.25$ & $384$ & $4.20\mathrm{E}{-09}$ & $100\%$ & $3.25$ \\
 &  &  &  & $0.42$ & $1000$ & $32$ & $15.625$ & $224$ & $8.10\mathrm{E}{-08}$ & $100\%$ & $3.29$ \\
\cmidrule(lr){2-12}
 & \multirow{4}{*}{$10$} & \multirow{4}{*}{$8.15$} & \multirow{4}{*}{$5$} & $13.20$ & $1$ & $32$ & $1$ & $689$ & $1.16\mathrm{E}{-08}$ & $100\%$ & $6.82$ \\
 &  &  &  & $1.32$ & $10$ & $32$ & $10$ & $792$ & $7.14\mathrm{E}{-08}$ & $100\%$ & $9.00$ \\
 &  &  &  & $0.53$ & $100$ & $32$ & $25$ & $558$ & $2.75\mathrm{E}{-09}$ & $100\%$ & $8.33$ \\
 &  &  &  & $0.85$ & $1000$ & $32$ & $15.625$ & $711$ & $1.31\mathrm{E}{-07}$ & $100\%$ & $7.73$ \\
\cmidrule(lr){2-12}
 & \multirow{4}{*}{$100$} & \multirow{4}{*}{$0.72$} & \multirow{4}{*}{$4$} & $102.00$ & $1$ & $2$ & $1$ & $978$ & $1.49\mathrm{E}{-07}$ & $100\%$ & $15.52$ \\
 &  &  &  & $10.20$ & $10$ & $2$ & $10$ & $973$ & $1.29\mathrm{E}{-07}$ & $100\%$ & $18.02$ \\
 &  &  &  & $1.02$ & $100$ & $2$ & $100$ & $923$ & $1.15\mathrm{E}{-07}$ & $100\%$ & $20.07$ \\
 &  &  &  & $1.63$ & $1000$ & $2$ & $62.5$ & $889$ & $2.96\mathrm{E}{-07}$ & $100\%$ & $18.52$ \\
\cmidrule(lr){2-12}
 & \multirow{4}{*}{$1000$} & \multirow{4}{*}{$0.07$} & \multirow{4}{*}{$4$} & $1020.00$ & $1$ & $2$ & $1$ & $1629$ & $1.70\mathrm{E}{-07}$ & $100\%$ & $21.16$ \\
 &  &  &  & $102.00$ & $10$ & $2$ & $10$ & $1603$ & $2.90\mathrm{E}{-07}$ & $99\%$ & $23.93$ \\
 &  &  &  & $10.20$ & $100$ & $2$ & $100$ & $1547$ & $2.16\mathrm{E}{-07}$ & $98\%$ & $26.67$ \\
 &  &  &  & $1.02$ & $1000$ & $2$ & $1000$ & $1228$ & $3.25\mathrm{E}{-07}$ & $98\%$ & $27.91$ \\
    \bottomrule
    \end{tabular}%
  }
\end{table}%

\begin{figure}
\raggedright{}%
\begin{minipage}[t]{0.45\columnwidth}%
\begin{center}
\includegraphics[width=0.98\textwidth,height=0.8\textwidth]{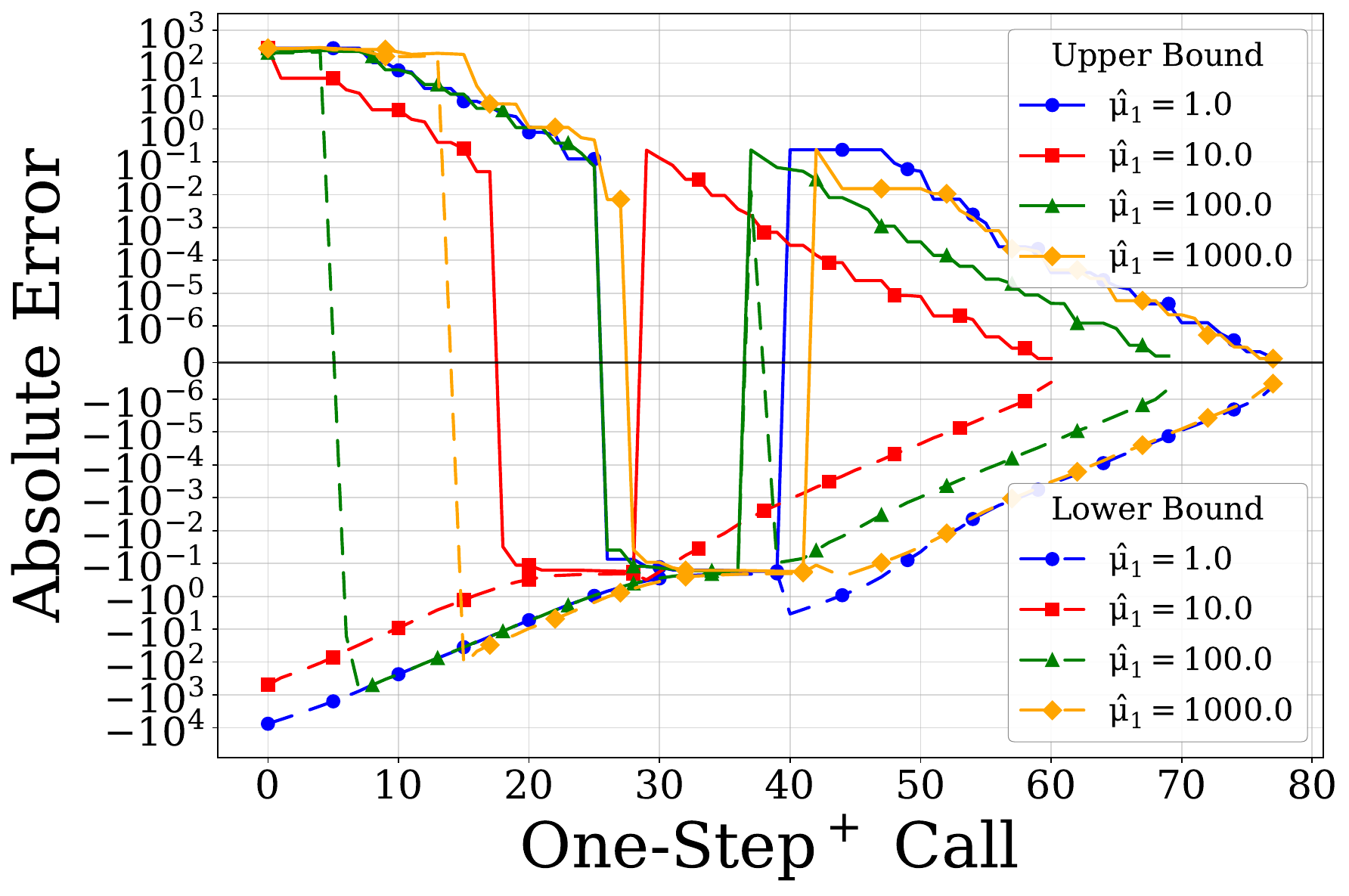}%
\caption{Absolute error of upper bound and lower bound of Restarted {\papex} vs. $\onestepplus$ calls on {convex QCQP} with $m=50,\kappa^+=10$ in Table~\ref{tab:small-scale_problems}.\label{fig:rapex_convergence_50_10}}
\par\end{center}%
\end{minipage}\hfill%
\begin{minipage}[t]{0.45\columnwidth}%
\begin{center}
\includegraphics[width=0.98\textwidth,height=0.8\textwidth]{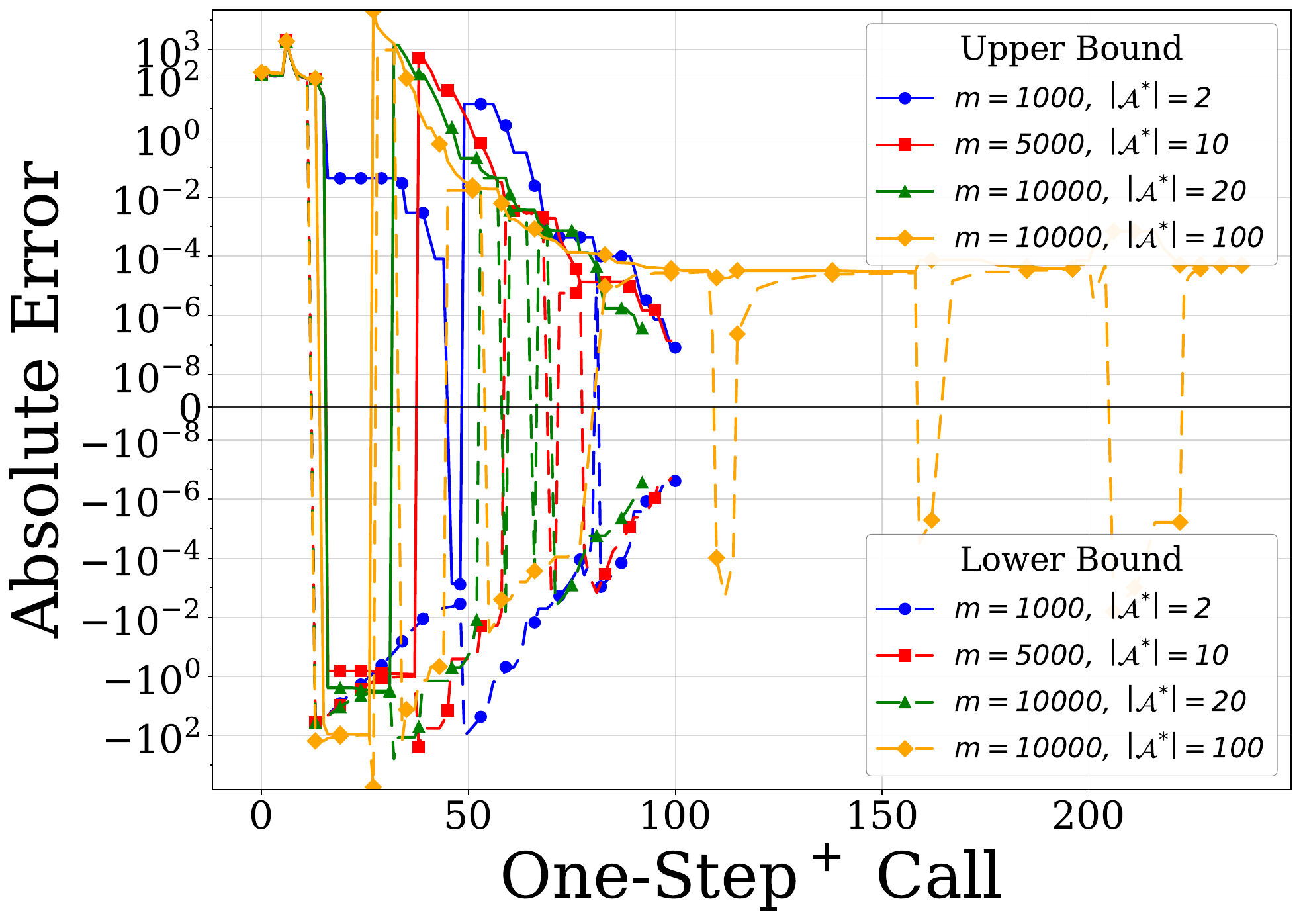}%
\caption{Absolute error of upper bound and lower bound of Restarted {\papex} vs. $\onestepplus$ calls on {convex QCQP} with different active constraint number in Table~\ref{tab:controlled_active_qcqp}. \label{fig:controlled_active_qcqp}}
\par\end{center}%
\end{minipage}
\end{figure}

\textbf{Key observations.}
Table~\ref{tab:small-scale_problems} leads to five observations.
First, the $\PWcer$ provides a reliable termination criterion: in all tested cases in Table~\ref{tab:small-scale_problems}, the residual $\max\bcbra{f(x)-f^*,[\mcal G(x)]_+}$ is below $10^{-6}$ whenever the certificate product satisfies $\iota_c \nu_c < 10^{-6}$.
Second, Restarted $\papex$ is robust to the initial QG estimate $\hat{\mu}_1$: for all tested values $\hat{\mu}_1\in\bcbra{1,10,100,1000}$, the residual $\max\bcbra{f(x)-f^*,[\mcal G(x)]_+}$ remains below $10^{-6}$. 
Third, the empirical condition-number estimate $\hat{\kappa}$ need not exceed $\kappa^+$. In most cases, $\hat{\kappa}<\kappa^+$; when $\hat{\kappa}>\kappa^+$, the overestimate is modest in the reported instances, with the largest ratio $\hat{\kappa}/\kappa^+$ about $1.62$. 
Fourth, Restarted $\papex$ performs well with only $\bundleSize=50$ cuts, including cases with $m=200$. This suggests that the method exploits the local piecewise-smooth structure encountered along the iterates, rather than requiring the bundle size to match the conservative worst-case bound $\max\bcbra{k_f,\sum_{i\in \activeconstrSet}k_{g_i}}$. The small values of $\abs{\mcal A^*}$ in the table are consistent with this explanation: $\abs{\mcal A^*}$ denotes the number of constraints that are active at the optimal solution to tolerance $10^{-6}$, i.e., $\abs{\mcal A^*}=\abs{\{i\in[m]: \abs{g_i(x^*)}\leq 10^{-6}\}}$. This implies that when the local region around the optimal solution involves fewer active pieces, Restarted $\papex$ can exploit this local structure.
Fifth, the practical early-stopping implementation is effective: the early-stopping ratio is $100\%$ in most cases, and the average number of oracle calls per $\onestep$ is far smaller than the full cut budget $B=50$.

We further plot the convergence of the upper bound for penalty function $\bar{\phi}^{s} = \brbra{\phi(\bar{y}^s;\hat{\gamma}_s)}$ and lower bound $\underline{\phi}^s = \brbra{\phi(\bar{y}^{s};\hat{\gamma}_s) - \tilde{\Delta}_s}$ in Figure~\ref{fig:rapex_convergence_50_10}, where the $y$-axis denotes the absolute error $\bar{\phi}^s - f^*$  or $\underline{\phi}^s - f^*$. 
Two observations can be made from the convergence plots.
First, Figure~\ref{fig:rapex_convergence_50_10} shows pronounced oscillations in the early iterations. During this phase, the method repeatedly invokes the $\adjustPara$ subroutine, and the computed upper and lower bounds need not lie on opposite sides of the optimal value. Once the parameter estimates stabilize, both bounds converge monotonically. 
Second, the upper and lower bounds converge at similar rates, which supports using the $\PWcer$ to monitor termination.

\begin{table}[htbp]
  \centering
  \caption{Restarted $\papex$ results on convex QCQP instances with controlled active-set sizes. The table reports the final estimates $\hat{\gamma}_S$ and $\hat{\mu}_S$, runtime, oracle calls, residual $\max\bcbra{f(x)-f^*,\bsbra{\mcal G(x)}_+}$, early-stopping ratio, and oracle calls per total step for each instance. All runs terminate after generating a certificate satisfying $\iota_c\nu_c \leq 10^{-6}$.\label{tab:controlled_active_qcqp}}
  {\small
  \setlength{\tabcolsep}{3.5pt}
    \begin{tabular}{crrcrrrrrrr}
    \toprule
    \multicolumn{1}{l}{$B$} & \multicolumn{1}{l}{$n$} & \multicolumn{1}{l}{$m$} & \multicolumn{1}{l}{$\abs{\mcal A^*}$}  & \multicolumn{1}{l}{$\hat{\gamma}_S$} & \multicolumn{1}{l}{$\hat{\mu}_S$} & \multicolumn{1}{l}{Time (sec)} & \multicolumn{1}{c}{\begin{tabular}{@{}c@{}}First-order \\oracle calls\end{tabular}} & \multicolumn{1}{c}{$\max\left\{\begin{array}{@{}c@{}} f(x)-f^*,\\ \left[\mcal G(x)\right]_+ \end{array}\right\}$} & \multicolumn{1}{c}{\begin{tabular}{@{}c@{}}Early\\stopping\\ratio\end{tabular}} & \multicolumn{1}{c}{\begin{tabular}{@{}c@{}}Oracle calls\\per\\ $\onestepplus$ \end{tabular}} \\
    \midrule
    \multirow{4}[2]{*}{$50$} & \multicolumn{1}{c}{\multirow{4}[2]{*}{$100$}} & $1000$  & \multicolumn{1}{r}{$2$}   & $128$   & $62.5$ & $8.4$  & $252$ & \multicolumn{1}{r}{$1.78\mathrm{E}{-08}$} & $100\%$ & $2.52$ \\
          &       & $5000$  & \multicolumn{1}{r}{$10$}  & $32$  & $3.91$ & $35.4$  & $593$ & \multicolumn{1}{r}{$5.43\mathrm{E}{-09}$} & $100\%$ & $5.99$ \\
          &       & $10000$ & \multicolumn{1}{r}{$20$}  & $32$  & $15.6$ & $91.2$ & $876$ & \multicolumn{1}{r}{$2.61\mathrm{E}{-08}$} & $100\%$ & $9.52$ \\
          &       & $10000$ & \multicolumn{1}{r}{$100$} & $32768$ & $1000$ & $241.0$ & $9080$ & \multicolumn{1}{r}{$5.10\mathrm{E}{-05}$} & $32.91\%$ & $38.31$ \\
    \bottomrule
    \end{tabular}%
  }
\end{table}%

\textbf{Experiment setting for convex QCQPs with a controllable number of active constraints.} 
As discussed in the third remark below Theorem~\ref{thm:ideal_papex}, the empirical performance appears to  depend more on the number of active constraints  encountered along the trajectory rather than on the total number of constraints. To examine this dependence empirically, we use the construction in Appendix~\ref{sec:controlled_active_qcqp} to generate QCQP instances with a prescribed active-set size at the optimal solution, reported as $\abs{\mcal A^*}$ in Table~\ref{tab:controlled_active_qcqp}. We run Restarted $\papex$ with $\hat{\mu}_1 = 1000.0$ and terminate it when it generates a certificate tuple $\Cert=(y_c,\iota_c,\nu_c,\gamma_c)$ satisfying $\iota_c\nu_c<10^{-6}$.

\textbf{Key observations.} Table~\ref{tab:controlled_active_qcqp} reports the controlled-active QCQP results, and Figure~\ref{fig:controlled_active_qcqp} plots the evolution of the corresponding upper bound for penalty function $\bar{\phi}^{s} = \brbra{\phi(\bar{y}^s;\hat{\gamma}_s)}$ and lower bound $\underline{\phi}^s = \brbra{\phi(\bar{y}^{s};\hat{\gamma}_s) - \tilde{\Delta}_s}$. Three observations are drawn from Table~\ref{tab:controlled_active_qcqp} and Figure~\ref{fig:controlled_active_qcqp}.  
First, the $\PWcer$ provides a reliable termination signal in the first three instances. 
Second, the method performs well when the prescribed active-set size is smaller than the bundle size. Moreover, in the two instances with $m=10000$, increasing $\abs{\mcal A^*}$ from $20$ to $100$ while keeping $\bundleSize=50$, the method still reaches the same certificate tolerance but requires substantially more oracle calls. 
Third, the certificate should not be interpreted as certifying the final parameter estimates themselves. The theory does not assert that a small product $\iota_c\nu_c$ implies the parameter estimates are valid, i.e., $\hat{\gamma}_s\geq2\gamma^+$ and $\hat{\mu}_s\leq\mu^*$. In the instance with $m=10000$ and $\abs{\mcal A^*}=100$, Figure~\ref{fig:controlled_active_qcqp} shows that the computed lower-bound curve can remain above $f^*$ even at termination, which implies $\hat{\gamma}_s<2\max\bcbra{\norm{\lambda^*}_1, 1}$ or $\hat{\mu}_s > \mu^*$.


\begin{table}[htbp]
  \centering
  \caption{Restarted $\papex$ results on large-scale convex QCQP instances. All runs use the stopping test $\iota_c \nu_c \leq \vep$. \label{tab:large_scale_qcqp}}
  {
  \setlength{\tabcolsep}{4.5pt}
    \begin{tabular}{crrrrrrrrrrr}
    \toprule
    \multicolumn{1}{l}{$B$} & \multicolumn{1}{l}{$n$} & \multicolumn{1}{l}{$m$} & \multicolumn{1}{l}{$\abs{\mcal{A}^*}$} & \multicolumn{1}{l}{$\vep$} & \multicolumn{1}{l}{$\hat{\gamma}_S$} & \multicolumn{1}{l}{$\hat{\mu}_S$} & \multicolumn{1}{l}{Time (sec)} & \multicolumn{1}{c}{\begin{tabular}{@{}c@{}}First-order\\oracle calls\end{tabular}} & \multicolumn{1}{c}{\begin{tabular}{@{}c@{}}One-step\\calls\end{tabular}} & \multicolumn{1}{c}{\begin{tabular}{@{}c@{}}Early\\stopping\\ratio\end{tabular}} & \multicolumn{1}{c}{\begin{tabular}{@{}c@{}}Oracle calls\\per\\ $\onestepplus$\end{tabular}} \\
    \midrule
    \multirow{4}[2]{*}{$5$} & 10000 & 50  & 4 & $10^{-6}$ & 2 & 3.91 & 103.74 & 544 & 117 & $37.61\%$ & 4.65 \\
          & 20000 & 50    & 5 & $10^{-6}$ & 2 & 3.91 & 337.56 & 700 & 146 & $17.81\%$ & 4.79 \\
          & 10000 & 100   & 1 & $10^{-3}$ & 32 & 1000 & 107.23 & 512 & 113 & $23.01\%$ & 4.53 \\
          & 20000 & 100   & 3 & $10^{-3}$ & 32 & 1000 & 229.08 & 466 & 110 & $29.09\%$ & 4.24 \\
    \bottomrule
    \end{tabular}%
  }
\end{table}%

\textbf{Experiment setting for large-scale convex QCQP.}
For large-scale instances, we reduce memory usage by using a common eigenvector basis for all matrices $Q_i$ and sampling the corresponding eigenvalues independently from $\mathrm{Unif}[1,10]$. We consider $n\in\{10000,20000\}$ and $m\in\{50,100\}$, set $C=1.0$ in~\eqref{eq:convex_qcqp}, and initialize Restarted $\papex$ with $\hat{\mu}_1=1000.0$ and $\bundleSize=5$. We terminate Restarted $\papex$ when it generates a certificate satisfying $\iota_c\nu_c<10^{-6}$ for $m=50$ and use the looser tolerance $\iota_c\nu_c<10^{-3}$ for $m=100$. 
The results are reported in Table~\ref{tab:large_scale_qcqp}. Because MOSEK does not provide reliable optimal reference values for these instances, we omit the residual $\max\bcbra{f(x)-f^*,[\mcal G(x)]_+}$. The quantity reported as $\abs{\mcal{A}^*}$ is the number of constraints whose values are within $10^{-8}$ of $\mcal G(x)$ at the best solution found; it is therefore an empirical active-set count rather than a certified active-set size at an optimal solution.

\textbf{Key finding.} Table~\ref{tab:large_scale_qcqp} suggests one observation for the large-scale QCQP instances. Restarted $\papex$ remains effective with $\bundleSize=5$, which is much smaller than the number of constraints. The runtime generally increases with the problem dimension and remains manageable for the largest tested instance, $n=20000$ and $m=100$. Hence, Restarted $\papex$ shows scalability on these large-scale convex QCQP instances.

\begin{table}[htbp]
  \centering
  \caption{Problem statistics, the RLS parameter $\tilde{K}$, and restarted $\papex$ certificate outputs. Here, $\tilde{K}$ is computed from~\eqref{eq:K_tilde} with $\vep=10^{-6}$ and $r=0$; $\hat{\gamma}_S$ and $\hat{\mu}_S$ denote the final penalty and QG-modulus estimates, respectively, and $\iota_c\nu_c$ denotes the product of the two parameters in the stored certificate tuple. \label{tab:problem_stats_and_RLS_threads}}%
    \begin{tabular}{cccrrrr}
    \toprule
    \multicolumn{3}{c}{Problem} & \multicolumn{1}{c}{\shortstack{Binary\\NPC}} & \multicolumn{1}{c}{\shortstack{Multi\\NPC}} & \multicolumn{1}{c}{\shortstack{Fairness\\Smooth}} & \multicolumn{1}{c}{\shortstack{Fairness\\Nonsmooth}} \\
    \cmidrule{4-7}
    \multicolumn{3}{c}{Dataset} & Arcene & Dry Bean & \multicolumn{2}{c}{COMPAS} \\
    \midrule
    \multicolumn{1}{c}{\multirow{3}[2]{*}{\shortstack{Data\\statistics}}} & \multicolumn{2}{c}{$d$} & 200   & 16   & \multicolumn{2}{c}{37}     \\
    & \multicolumn{2}{c}{$m$} & 1     & 7     &\multicolumn{2}{c}{2} \\
    & \multicolumn{2}{c}{$n$} & 100   & 700   & \multicolumn{2}{c}{4115} \\
    \midrule
    \multicolumn{1}{c}{\multirow{6}[3]{*}{\shortstack{Restarted\\Penalty\\APEX}}} & \multicolumn{1}{c}{\multirow{3}[2]{*}{\shortstack{$\hat{\mu}_1$\\$=1$}}} & \multicolumn{1}{l}{$\hat{\gamma}_S$} & 8     & 2     & 2     & 2 \\
          &       & \multicolumn{1}{l}{$\hat{\mu}_S$} & 0.25  & 1 & 1     & 1 \\
          &       & \multicolumn{1}{l}{$\iota_c\nu_c$} & 8.14E-06 & 1.26E-05 & 2.20E-06 & 4.18E-06 \\
\cmidrule{2-7}          & \multicolumn{1}{c}{\multirow{3}[1]{*}{\shortstack{$\hat{\mu}_1$\\$=10$}}} & \multicolumn{1}{l}{$\hat{\gamma}_S$} & 8     & 8     & 2     & 2 \\
          &       & \multicolumn{1}{l}{$\hat{\mu}_S$} & 0.15625 & 10   & 2.5   & 0.625 \\
          &       & \multicolumn{1}{l}{$\iota_c\nu_c$} & 3.99E-05 & 1.26E-05 & 2.12E-06 & 6.72E-06 \\
          \midrule
    RLS   & \multicolumn{2}{c}{$\tilde{K}$} & 341   & 61    & 407   & 309 \\
    \bottomrule
    \end{tabular}%
\end{table}%

\subsection{Neyman-Pearson classification\label{subsec:binary_np}}
We next consider Neyman-Pearson classification (NPC) problems~\cite{tong2016survey}, which minimize the loss on a target class while controlling the loss on the remaining classes.

\textbf{Binary NPC.}
For binary labels $y_i\in\bcbra{-1,1}$, we use the logistic loss and solve
\begin{equation}\label{eq:binary_np}
    \begin{aligned}
        \min_{w\in \mbb R^{d}}\quad
        & \frac{1}{n_1}\sum_{i=1}^{n}\onebf_{y_i=1}\log\brbra{1+\exp(-y_i x_i^{\top} w)}
        + \frac{\rho}{2}\norm{w}^2\\
        \st\quad
        & \frac{1}{n_{-1}}\sum_{i=1}^{n}\onebf_{y_i=-1}\log\brbra{1+\exp(-y_i x_i^{\top} w)}
        \leq \kappa,
    \end{aligned}
\end{equation}
where $n_1$ and $n_{-1}$ denote the numbers of samples in the positive and negative classes, respectively. In the binary experiments, we set $\kappa=0.5$ and $\rho=0.01$, and use the Arcene dataset~\cite{miscArcene167}.

\textbf{Multi-class NPC.}
For the multi-class case, we minimize the average cross-entropy loss subject to an upper bound on each class-wise loss:
\begin{equation}\label{eq:multi_np}
    \begin{aligned}
        \min_{w_1,\ldots,w_m \in \mbb R^{d}}\quad
        & -\frac{1}{n}\sum_{i=1}^{n}\sum_{j=1}^{m}\onebf_{y_i = j}\log(p_{ij}) + \frac{\rho}{2}\sum_{j=1}^{m}\norm{w_j}^2\\
        \st\quad
        & -\frac{1}{n_j}\sum_{i=1}^{n}\onebf_{y_i = j}\log(p_{ij})
        \leq \kappa_j,\qquad j=1,\ldots,m,
    \end{aligned}
\end{equation}
where $n_j = \sum_{i=1}^{n}\onebf_{y_i = j}$ and
$p_{ij}=\exp(x_i^{\top} w_j)/\sum_{\ell=1}^{m}\exp(x_i^{\top} w_\ell)$. In the multi-class experiments, we set $\rho = 0.01$, $\kappa_j=0.8$ for every $j=1,\ldots, m$ and use the Dry Bean dataset~\cite{Koklu2020MulticlassCO}.

We compare Restarted $\papex$ with RLS on the binary and multi-class NPC problems.  We stop each method once $\max\bcbra{f(x)-f^*,[\mcal G(x)]_+}\leq10^{-6}$ or the oracle-call budget reaches $500{,}000$. For Restarted $\papex$, we test the inital QG estimate $\hat{\mu}_1 \in \bcbra{1,10}$. Table~\ref{tab:problem_stats_and_RLS_threads} reports the problem sizes and the corresponding parameter settings and results.  The convergence curves are shown in Figures~\ref{fig:np_convergence_binary} and~\ref{fig:np_convergence_multi}. Since RLS is parallel, we plot its oracle counts on a single-thread-equivalent scale.

Table~\ref{tab:problem_stats_and_RLS_threads} and Figures~\ref{fig:np_convergence_binary}--\ref{fig:np_convergence_multi} lead to three observations. First, Figure~\ref{fig:np_convergence_binary} shows that Restarted $\papex$ requires fewer first-order oracle calls than RLS on the binary NPC instance. This gap is larger when the prescribed quantity $\tilde{K}$ for RLS is large (see Table~\ref{tab:problem_stats_and_RLS_threads}), which is consistent with the dependence on $\tilde{K}$~\cite{lin2025adaptiveparameterfreeprojectionfreerestarting}. Second, Restarted $\papex$ reaches $\max\bcbra{f(x)-f^*,[\mcal G(x)]_+}\leq10^{-6}$ on both NPC instances, and the final certificate products $\iota_c\nu_c$ are of the same order almostly, supporting the use of the $\PWcer$ as an effective termination criterion.

\begin{figure}
\raggedright{}%
\begin{minipage}[t]{0.45\columnwidth}%
\begin{center}
\includegraphics[width=0.98\textwidth,height=0.8\textwidth]{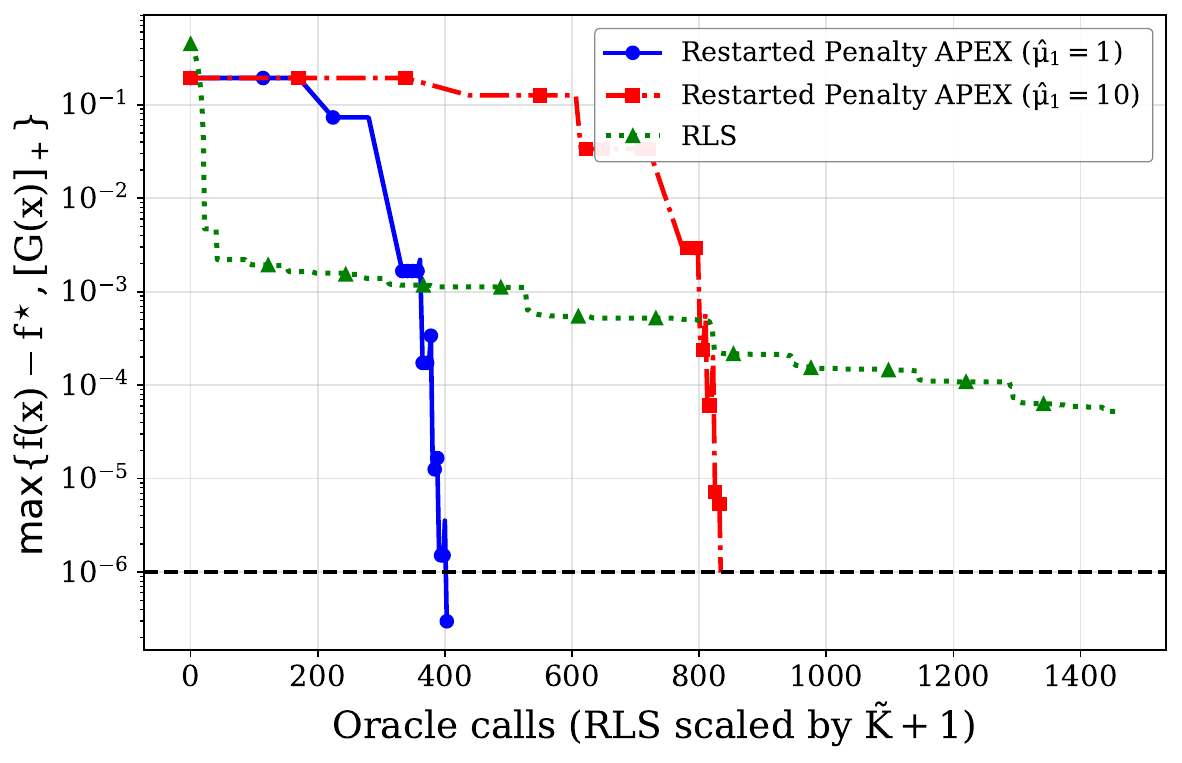}%
\caption{Restarted $\papex$ and RLS on binary NPC, with thread-normalized RLS oracle counts. \label{fig:np_convergence_binary}}
\par\end{center}%
\end{minipage}\hfill%
\begin{minipage}[t]{0.45\columnwidth}%
\begin{center}
\includegraphics[width=0.98\textwidth,height=0.8\textwidth]{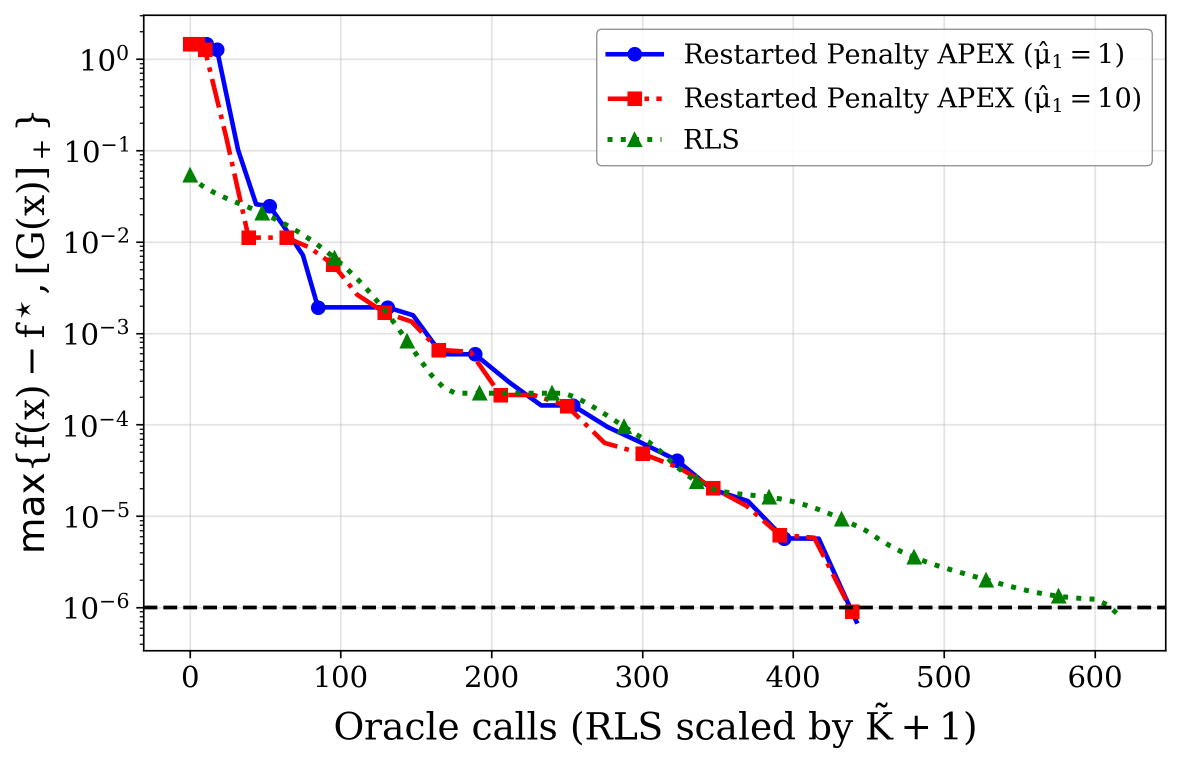}%
\caption{Restarted $\papex$ and RLS on multi-class NPC, with thread-normalized RLS oracle counts. \label{fig:np_convergence_multi}}
\par\end{center}%
\end{minipage}
\end{figure}

\begin{figure}
\raggedright{}%
\begin{minipage}[t]{0.45\columnwidth}%
\begin{center}
\includegraphics[width=0.98\textwidth,height=0.8\textwidth]{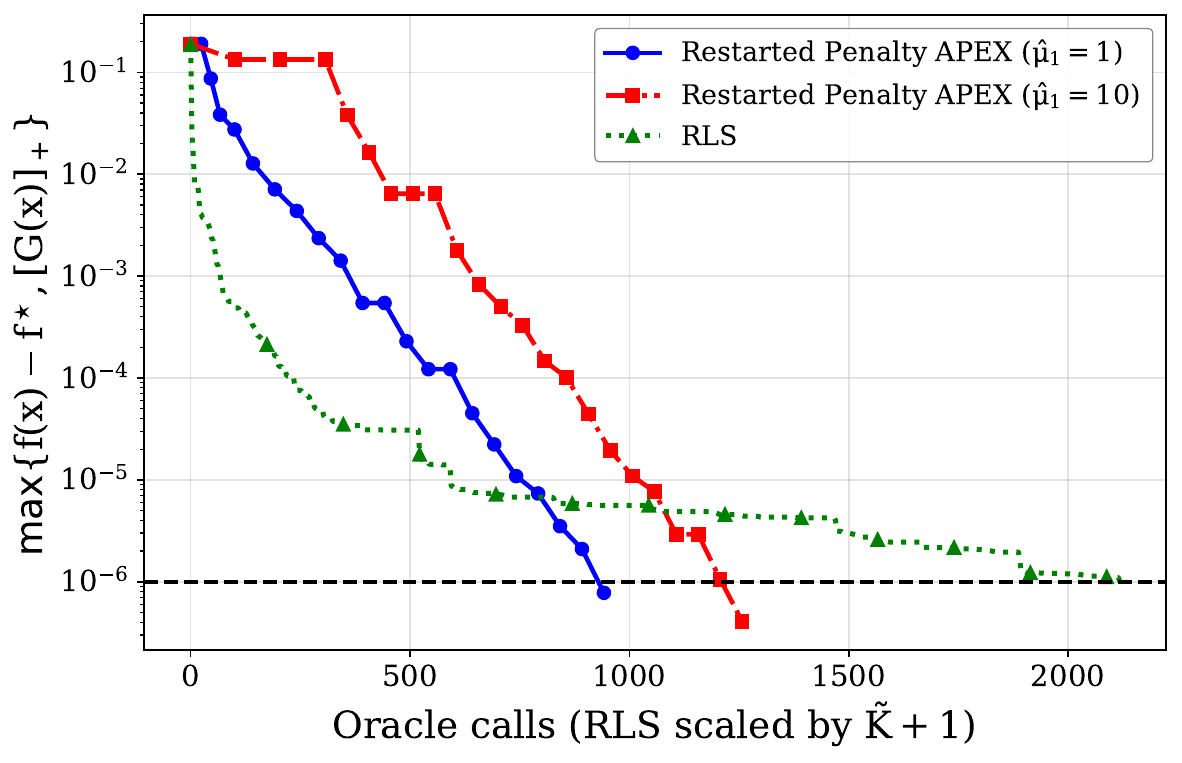}%
\caption{Restarted $\papex$ and RLS on the smooth convex fairness surrogate, with thread-normalized RLS oracle counts. \label{fig:fairness_convergence_smooth}}
\par\end{center}%
\end{minipage}\hfill%
\begin{minipage}[t]{0.45\columnwidth}%
\begin{center}
\includegraphics[width=0.98\textwidth,height=0.8\textwidth]{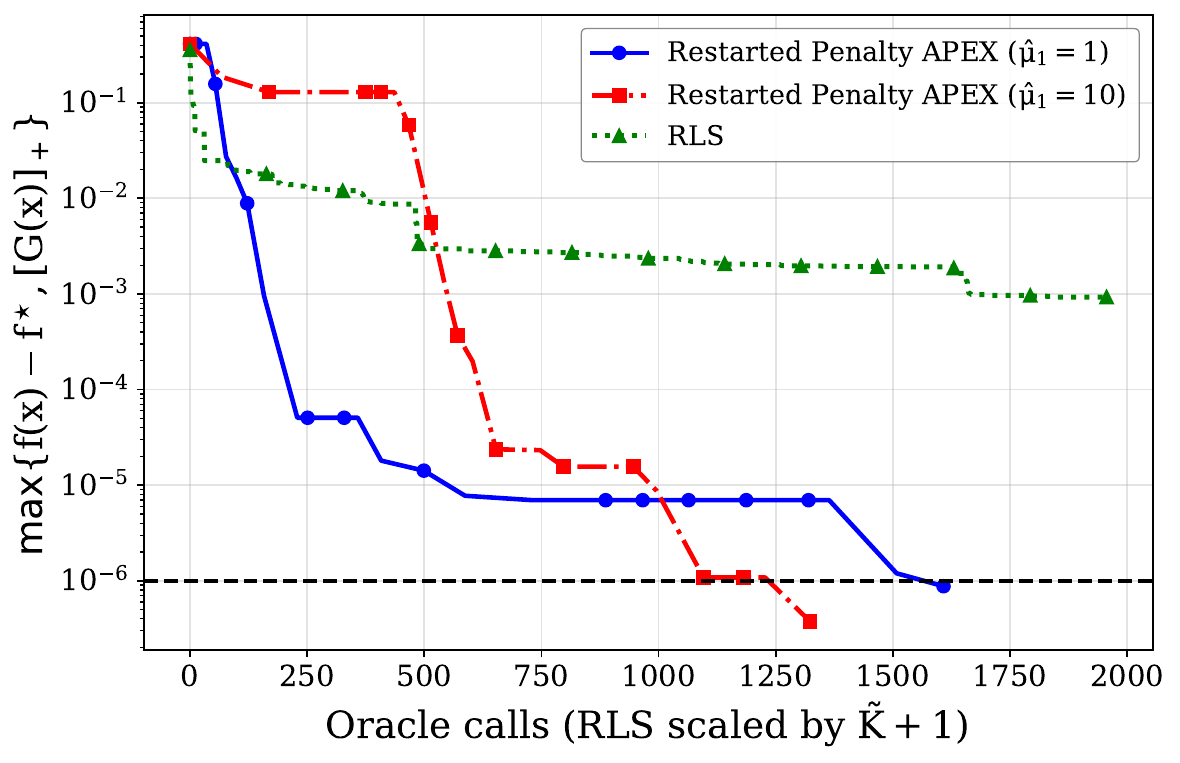}%
\caption{Restarted $\papex$ and RLS on the nonsmooth convex fairness surrogate, with thread-normalized RLS oracle counts. \label{fig:fairness_convergence_nonsmooth}}
\par\end{center}%
\end{minipage}
\end{figure}

\subsection{Classification with fairness constraints\label{subsec:fairness}}
Consider a classification dataset $\bcbra{x_i, y_i}_{i=1}^n$, where $x_i \in \mathbb{R}^d$ is the feature vector and $y_i \in \{-1, 1\}$ is the label. Fairness-aware classification can be written as~\cite{goh2016satisfying}
{\small\begin{equation}\label{eq:fairness_aware_classification}
    \begin{aligned}
        \min_{w\in \mbb R^d, \norm{w}^2 \leq r^2}\quad & \frac{1}{n}\sum_{i=1}^n \ell(y_i x_i^{\top} w)\\
        \st\quad &
        \frac{1}{\abs{S_1}}\sum_{i \in S_1}\sigma(x_i^\top w)
        \geq
        \frac{\kappa}{\abs{S_2}}\sum_{i \in S_2}\sigma(x_i^\top w),\\
        &
        \frac{1}{\abs{S_2}}\sum_{i \in S_2}\sigma(x_i^\top w)
        \geq
        \frac{\kappa}{\abs{S_1}}\sum_{i \in S_1}\sigma(x_i^\top w).
    \end{aligned}
\end{equation}}
We set $r=10$ and $\kappa=0.9$. The parameter $\kappa$ controls the allowable disparity between the two sensitive groups $S_1$ and $S_2$, of  sizes are $\abs{S_1}$ and $\abs{S_2}$. The loss $\ell$ is non-increasing and convex, and $\sigma(x_i^\top w)=\max\bcbra{0,\min\bcbra{1,0.5+x_i^\top w}}\in[0,1]$ is a clipped score for predicting class $+1$. The two fairness constraints require the predicted positive rates of the two groups to remain within a multiplicative factor $\kappa$. Using $\sigma(x_i^\top w) = 1 - \sigma(-x_i^\top w)$, the first constraint in~\eqref{eq:fairness_aware_classification} is equivalent to $\frac{\kappa}{\abs{S_2}}\sum_{x_i \in S_2}\sigma(x_i^\top w) + \frac{1}{\abs{S_1}}\sum_{x_i \in S_1}\sigma(-x_i^\top w)\leq 1$.

For the fairness experiments, we use the COMPAS dataset from ProPublica~\footnote{\url{https://github.com/propublica/compas-analysis}}. The dataset contains criminal history, jail and prison time, demographic variables, and other covariates. We use $n=4{,}115$ examples to define the objective and $\abs{S_1}+\abs{S_2}=3{,}099$ examples to define the constraints. The two sensitive groups have sizes $\abs{S_1}=2{,}462$ and $\abs{S_2}=637$.

Since these constraints are nonconvex, we follow~\cite{lin2020data,lin2025adaptiveparameterfreeprojectionfreerestarting} and consider two convex surrogates, one smooth and one nonsmooth.

\textbf{Smooth convex surrogate.\label{subsubsec:fairness_smooth}}
For the smooth surrogate, we replace $\sigma(x_i^\top w)$ with $\log_4(1+\exp(x_i^\top w))$ and use the logistic loss $\ell(y_i x_i^\top w)=\ln (1+\exp(-y_i x_i^\top w))$. The resulting objective and constraints are differentiable with Lipschitz continuous gradients.

\textbf{Nonsmooth convex surrogate.\label{subsubsec:fairness_nonsmooth}}
For the nonsmooth surrogate, we use the hinge loss $\ell(y_i x_i^\top w)=\max\bcbra{0,1-y_i x_i^\top w}$ and replace $\sigma(x_i^\top w)$ by $\max\bcbra{0,0.5+x_i^\top w}$.

Table~\ref{tab:problem_stats_and_RLS_threads} reports the problem sizes, parameter settings, and Restarted $\papex$ outputs, while Figures~\ref{fig:fairness_convergence_smooth} and~\ref{fig:fairness_convergence_nonsmooth} show the convergence curves. The results are consistent with the NPC experiments: Restarted $\papex$ uses fewer oracle calls than RLS and reaches $\max\bcbra{f(x)-f^*,[\mcal G(x)]_+}\leq 10^{-6}$ for both surrogate formulations. Furthermore, the $\PWcer$ in Table~\ref{tab:problem_stats_and_RLS_threads} satisfies $\iota_c\nu_c \approx 10^{-5}$, which supports the use of $\PWcer$ as an effective termination criterion. For the smooth surrogate, Figure~\ref{fig:fairness_convergence_smooth} shows linear convergence, consistent with the theory since the bundle size $\bundleSize=50$ exceeds the number of constraints $m=2$. In contrast, Figure~\ref{fig:fairness_convergence_nonsmooth} does not show linear convergence, as expected for the nonsmooth surrogate.

\renewcommand \thepart{}
\renewcommand \partname{}

\bibliographystyle{abbrvnat}
\bibliography{ref.bib}


\newpage
\appendix

\addcontentsline{toc}{section}{Appendix}
\part{Appendix} 

\section{Equivalent reformulation of the $\PWcer$ subproblem}
The projection subproblem~\eqref{eq:subproblems_plus}, used by $\onestepplus$ within the $\PWcer$-generation routine $\pawg$, admits an equivalent epigraph reformulation. Specifically, we introduce auxiliary variables $u$ and $v$ to upper-bound the objective and constraint cuts, respectively, with $v\geq 0$ accounting for the positive part:
\begin{equation}
\begin{aligned}
\min_{x,u,v}\quad
&\frac12\|x-\bar y\|^2\\
\text{s.t.}\quad
&x\in X\cap\bar X(t,0),\\
&u\ge \ell_f(x;\underline{x}^{t,j}),\qquad j=1,\ldots,i,\\
&v\ge \ell_{\mcal G}(x;\underline{x}^{t,j}),\qquad j=1,\ldots,i,\\
&v\ge 0,\\
&u+\gamma v\le \tilde{l} .
\end{aligned}
\end{equation}
All additional constraints are linear in $(x,u,v)$, and $\bar X(t,0)$ is a halfspace. Therefore, the subproblem preserves the same basic computational structure as a convex quadratic program over the simple set $X$.

\section{A convex QCQP construction with a prescribed number of active constraints}
\label{sec:controlled_active_qcqp}

This appendix describes a simple way to generate convex QCQP instances for which
both the optimizer and the active set are known by construction. We consider convex QCQP with $X=\mbb R^n$: 
\begin{equation}\label{eq:controlled_qcqp_model}
\begin{aligned}
\min_{x\in\mbb R^n}\quad
& f(x):=\frac12 x^\top Q_0 x+c_0^\top x\\
\st\quad
& g_i(x):=\frac12 x^\top Q_i x+c_i^\top x+d_i\leq0,
\qquad i=1,\ldots,m,
\end{aligned}
\end{equation}
where $Q_0 \succ 0$ and $Q_i\succeq0$ for $i=1,\ldots,m$. The aggregate constraint
violation is
\[
\mcal G(x):=\max_{1\leq i\leq m}g_i(x),
\qquad
\bsbra{\mcal G(x)}_{+}:=\max\bcbra{\mcal G(x),0}.
\]

Fix a target number of active constraints $m^*\in\{0,\ldots,m\}$ and choose an
active set $\mcal A^*\subseteq\{1,\ldots,m\}$ with $\abs{\mcal A^*}=m^*$; for example, one may take
$\mcal A^*=\{1,\ldots,m^*\}$, with $\mcal A^*=\emptyset$ when $m^*=0$. The goal is to
construct~\eqref{eq:controlled_qcqp_model}
so that a prescribed point $x^\star\in\mbb R^n$ is the optimizer and satisfies
\begin{equation}\label{eq:controlled_active_set}
g_i(x^\star)=0\quad i\in \mcal A^*,
\qquad
g_i(x^\star)<0\quad i\notin \mcal A^*.
\end{equation}
This construction is useful for testing how the method behaves as the number of
active constraints changes while keeping both the optimizer and a KKT certificate
analytically available.

The construction proceeds backward from the desired KKT system. We first draw
$x^\star\sim \mathcal N(0,I_n)$ and generate symmetric quadratic matrices with a
shared orthonormal eigenvector matrix $U$:
\[
Q_0=U\mathrm{diag}(q_{01},\ldots,q_{0n})U^{\top},\qquad
Q_i=U\mathrm{diag}(q_{i1},\ldots,q_{in})U^{\top},\quad i=1,\ldots,m,
\]
where $U^{\top}U=UU^{\top}=I_n$.
We draw the objective eigenvalues independently as
\[
q_{0j}\sim \mathrm{Unif}[\ell_0,u_0],
\]
where $0<\ell_0\leq u_0$, so $Q_0\succ0$. We draw the constraint eigenvalues
independently as
\[
q_{ij}\sim \mathrm{Unif}[\ell_g,u_g],
\]
where $0\leq\ell_g\leq u_g$, so $Q_i\succeq0$. In our experiments,
$\ell_0=u_0=1$ and $(\ell_g,u_g)=(0.05,0.20)$.

Next, we draw the dual multipliers for the prescribed active constraints as
\[
\lambda_i\sim \mathrm{Unif}[\lambda_{\min},\lambda_{\max}],
\qquad i\in\mcal A^*,
\]
where $0<\lambda_{\min}\leq\lambda_{\max}$. We set $\lambda_i=0$ for
$i\notin\mcal A^*$. In our experiments,
$(\lambda_{\min},\lambda_{\max})=(0.5,2.0)$. For each active constraint
$i\in\mcal A^*$, choose the desired active-gradient vector
$\nabla g_i(x^\star)$ and define
\begin{equation}\label{eq:active_constraint_coefficients}
c_i:=\nabla g_i(x^\star)-Q_i x^\star,
\qquad
d_i:=-\frac12(x^\star)^\top Q_i x^\star-c_i^\top x^\star .
\end{equation}
Then $g_i(x^\star)=0$ for all $i\in \mcal A^*$. For each inactive constraint
$i\notin \mcal A^*$, choose $c_i\in\mbb R^n$ and a slack $s_i>0$, and set
\begin{equation}\label{eq:inactive_constraint_coefficients}
d_i:=-\frac12(x^\star)^\top Q_i x^\star-c_i^\top x^\star-s_i .
\end{equation}
Then $g_i(x^\star)=-s_i<0$, so exactly the constraints in $\mcal A^*$ are active at
$x^\star$.

It remains to choose the linear term in the objective. Define
\begin{equation}\label{eq:controlled_qcqp_c0}
c_0:=-Q_0x^\star-\sum_{i\in \mcal A^*}\lambda_i\brbra{Q_i x^\star+c_i}.
\end{equation}
Using $\lambda_i=0$ for $i\notin \mcal A^*$, we obtain
\[
\nabla f(x^\star)+\sum_{i=1}^m\lambda_i\nabla g_i(x^\star)
=Q_0x^\star+c_0+\sum_{i=1}^m\lambda_i\brbra{Q_i x^\star+c_i}
=0.
\]
Together with~\eqref{eq:controlled_active_set}, $\lambda_i\geq0$, and
$\lambda_i g_i(x^\star)=0$, this verifies the KKT conditions at $x^\star$. Since
the problem is convex and $Q_0\succ0$, these KKT conditions are sufficient for
global optimality, and the strict convexity of $f$ on the convex feasible region
implies that $x^\star$ is the unique global minimizer. Thus the reference values
used in the experiments can be set to $x_{\rm ref}=x^\star$ and
$f_{\rm ref}^\star=f(x^\star)$.

\end{document}